\documentclass[11pt,dvipsnames]{article}
\usepackage[margin=2cm]{geometry}
\usepackage{amsmath,amssymb,amsfonts,amsthm}
\usepackage{mathrsfs}
\usepackage{mathtools}
\usepackage{tikz}
\usepackage{bm}
\usepackage{graphicx}
\usepackage{enumerate}
\usepackage{natbib}
\usepackage{color}
\usepackage{setspace}
\usepackage{enumitem}
\usepackage[retainorgcmds]{IEEEtrantools}
\usepackage{bbm}
\usepackage{authblk}
\usepackage{breqn}

\theoremstyle{definition}
\newtheorem{definition}{Definition}

\newtheorem{lemma}{Lemma}
\newtheorem{example}{Example}
\newtheorem{assumption}{Assumption}

\newtheorem{proposition}{Proposition}
\theoremstyle{definition}
\newtheorem{remark}{Remark}

\newcommand{\RV} {\ensuremath{\normalfont \text{RV}}}
\newcommand{\sgnmm}{\text{sgn}} 
\newcommand{\sgn}{\text{sgn}}
\newcommand{\supp} {\ensuremath{\normalfont \text{supp}}}
\newcommand{\GG} {\ensuremath{\mathcal{G}}}

 \newcommand{\PP}{\ensuremath{\mathbb{P}}}
\newcommand{\RR}{\ensuremath{\mathbb{R}}}

\newcommand{\level}{\ensuremath{t}}

\newcommand{\calC}{\ensuremath{\mathcal{C}}}
\newcommand{\calD}{\ensuremath{\mathcal{D}}}

\newcommand{\Vset}{\ensuremath{\mathcal{V}}}
\newcommand{\Eset}{\ensuremath{\mathcal{E}}}
\newcommand{\sm}{\ensuremath{\setminus}}
\newcommand{\bigCI}{\mathrel{\text{\scalebox{1.07}{$\perp\mkern-10mu\perp$}}}}

\definecolor{bleudefrance}{rgb}{0.19, 0.55, 0.91}
\definecolor{darkgreen}{rgb}{0.0, 0.2, 0.13}
\usepackage{hyperref}
\hypersetup{colorlinks=true,linkcolor=darkgreen,citecolor=MidnightBlue}

\begin{document}
\title{Geometric extremal graphical models and coefficients of extremal dependence on block graphs}
\author[1]{Ioannis Papastathopoulos} 
\author[2]{Jennifer L. Wadsworth}
\affil[1]{University of Edinburgh, U.K.}
\affil[2]{Lancaster University, U.K.}

\doublespacing
\maketitle
\begin{abstract}
  We introduce the concept of geometric extremal graphical models,
  which are defined through the gauge function of the limit set
  obtained from suitably scaled random vectors in light-tailed
  margins. For block graphs, we prove results relating to the
  propagation of various extremal dependence coefficients along the
  graph. A particular focus is placed on coefficients that link to the
  framework of conditional extreme value theory, which are especially
  interesting when variables do not all attain their most extreme
  values simultaneously. We also consider results related to the case
  when variables do exhibit joint extreme behaviour. Through the
  recent translation of the geometric approach for multivariate
  extremes to a statistical modelling framework, geometric extremal
  graphical models, and results relating to them, pave the way for an
  approach to modelling of high dimensional extremes with complex
  extremal dependence structures.
\end{abstract}




\section{Introduction}
\label{sec:intro}

\subsection{Setting}
\label{sec:setting}
Estimates are often required of rare-event probabilities that involve multiple variables. Examples include the probability that many sites on a river network experience flooding, that several stock market holdings simultaneously decrease sharply, or that different components of a weather system (e.g., wind, rain, temperature) are extreme in a damaging combination. The modelling of multivariate extremes entails using suitable characterizations of the joint tail of a distribution to yield a statistical model from which extrapolations further into the tail are possible.

When it comes to high dimensional modelling of extremes, much progress has been made in the context of spatial extremes, where models can now be defined and fitted to hundreds or thousands of observation locations \citep{Simpsonetal23,Shietal24,Hazraetal25}. This is because the spatial nature of the data permits highly structured models, based on simplifying assumptions that dependence decays with distance, and involving relatively few parameters in comparison to the number of locations. When the data is more genuinely multivariate, most extreme value modelling is typically limited to low $(< 10)$ dimensions. One reason for this is the high degree of complexity in suitably characterizing the extremal dependence structure of an arbitrary multivariate vector (see, e.g., \cite{WadsworthCampbell24} for a discussion). Nonetheless, even under simplifying assumptions on the extremal dependence, potential models can still entail large numbers of parameters, rendering estimation and interpretation extremely difficult. For this reason, the concept of sparsity in multivariate extreme value modelling has become of recent interest; see \citet{EngelkeIvanovs21} for a review. One notion of sparsity is to try and simplify high dimensional modelling tasks through a graphical specification of dependence. Our focus in this work is to outline a framework for precisely this task, in the context of a geometric approach for multivariate extremes.

\subsection{Background on graphical models}
\label{sec:GMbackground}

Graphical models represent a way to simplify the probabilistic expression and statistical analysis of a joint distribution in high dimensions. The key notion is the concept of conditional independence between variables. Let $\bm{X} \in \mathbb{R}^d$, and for some index set $J \subseteq \{1,\ldots,d\}$, denote $\bm{X}_J = (X_j: j \in J)$. Assuming the existence of joint densities, we let $f_{J}(\bm{x}_J)$ denote the marginal density of $\bm{X}_J$, and for disjoint index sets $I,J$ define the joint density $f_{I,J}(\bm{x}_I,\bm{x}_J)$ to be that of $\bm{X}_{I \cup J}$. Conditional densities are defined as $f_{I \mid J}(\bm{x}_{I}\mid \bm{x}_{J})= f_{I,J}(\bm{x}_{I},\bm{x}_{J})/f_{J}(\bm{x}_{J})$. For disjoint indexing sets $I,J,K$ we say $\bm{X}_I$ is conditionally independent of $\bm{X}_J$ given $\bm{X}_K$ if $f_{I,J\mid K}(\bm{x}_I,\bm{x}_J\mid\bm{x}_K) = f_{I\mid K}(\bm{x}_I\mid\bm{x}_K)f_{J\mid K}(\bm{x}_J\mid\bm{x}_K)$, and we write $\bm{X}_I \bigCI \bm{X}_J \mid \bm{X}_K$. An undirected graphical model, which we work with here, provides a succinct way to encode these conditional independences.

An undirected graph $\GG$ is defined by its sets of vertices $\Vset$ and edges $\Eset \subset \Vset \times \Vset$. For the graph $\GG = (\Vset,\Eset)$, let $X_i$ denote the random variable associated with vertex $i$. A random vector $\bm X$ is said to follow a graphical model with
conditional independence graph $\GG=(\Vset, \Eset)$ if its
distribution satisfies the pairwise Markov property relative to
$\mathcal{G}$, that is, if
$X_i\bigCI X_j \mid \bm X_{\Vset\sm \{i,j\}}$ for all
$(i,j)\notin \Eset$. By the Hammersley--Clifford theorem, this is
equivalent to the global Markov property when $\bm{X}$ has a positive
continuous density, $f$, which means that the conditional independence
relationships hold when conditioning on vertices $S \subset \Vset$ that \emph{separate} $i$ and $j$, meaning all paths from $i$ to $j$ intersect a node in $S$ \citep{Besag74, grim18}. For example, in the first panel of Figure~\ref{fig:graphs} below, $X_1 \bigCI X_9 \mid \bm{X}_{\{3,4,6\}}$.

In this work we focus on a simple class of undirected graphical models, termed \emph{block graphs}. These are a subset of the possible \emph{decomposable} graphical models, which may be described through sets of indices known as \emph{cliques} and \emph{separators}. Here we define a clique as a group of indices that form a maximal fully-connected subgraph, while the separators are intersections of (and hence subsets of) the cliques. When
$\bm X$ follows a decomposable graphical model, $f$ factorizes according to
\begin{equation}
  f(\bm x) = \prod_{C\in \mathcal{C}} f_C(\bm x_C)/\prod_{D\in
    \mathcal{D}} f_{D}(\bm x_D), \label{eq:decompGM}
\end{equation}
where $\mathcal{C}$ is the set of cliques and $\mathcal{D}$ the set of
separators. Cliques typically represent much smaller sets of indices
than $\Vset$, and the separators are subsets of cliques, so that equation~\eqref{eq:decompGM}
simplifies the expression of $f$ in terms of lower dimensional
densities. Block graphs are decomposable graphs where the separators are all singleton sets. 

\subsection{Graphical models for extremes}
\label{sec:GMExtreme}

The concept of (undirected) graphical models for extremes was introduced in \citet{EngelkeHitz20}. Specifically, extremal graphical models were established for a class of distributions known as multivariate generalized Pareto distributions \citep{RootzenTajvidi06,Rootzenetal18}, for which the usual notion of conditional independence does not apply because the support of the distribution is not a product space. Instead, \citet{EngelkeHitz20} defined a graphical factorization not on probability density functions, but on so-called exponent measure densities, and showed that this has an interpretation in terms of usual conditional independences when conditioning to restrict the support of the random vector to a product space. The use of multivariate generalized Pareto distributions, however, means that this approach is restricted to random vectors that are fully asymptotically dependent, implying that joint extremes of all variables occur with a similar frequency to marginal extremes, or more specifically that
\begin{align}
\chi_{\Vset} = \lim_{u \to 1} \Pr(F_j(X_j)>u, \forall~j \in \Vset) / (1-u) >0. \label{eq:chipos}
\end{align}
The formulation in \citet{EngelkeHitz20} actually required even stricter assumptions that no subvectors $\bm{X}_J$, $J\subset \Vset$ experience simultaneous extremes while $\bm{X}_{\Vset\sm J}$ are smaller order.
These are particularly stringent requirements for the high dimensional case that graphical models are intended to facilitate, and so limits the applicability of such models in practice. Graphical models on exponent measures, which do not require assumption~\eqref{eq:chipos}, have been outlined more recently in \citet{Engelkeetal25}, but these have not yet given rise to applicable statistical methodology. 

Other related work includes \citet{LeeCooley22} and
\citet{Gongetal24}, who concurrently developed the notion of partial
tail correlation for inferring extremal conditional independence,
under the same framework of multivariate regular variation as
\citet{EngelkeHitz20}. \citet{GissiblKluppelberg18} and
\citet{Gissibletal21} consider so-called max-linear models on directed
acyclic graphs. Such models are generally only suited to multivariate
data arising as componentwise maxima and can be difficult to interpret
for statistical analysis due to lack of densities. Seeking
applicability to a broader range of extremal dependence structures,
\citet{CaseyPapastathopoulos23} \color{black} develop graphical
modelling ideas in the setting of \textit{conditional extreme value
  theory} \citep{hefftawn04,heffres07}, where the resulting
theoretical forms are diverse and rich, presenting flexibility but
practical challenges for implementation. \color{black}
\citet{Farrelletal24} adopt a pragmatic modelling approach, by
imposing graphical structure on the so-called residuals of the
conditional extreme value model. Similar ideas were explored in the
discussion to \citet{EngelkeHitz20} \citep{Wadsworth20}.

\subsection{Geometric extremes}
\label{sec:geometricintro}

When working with extremes of high dimensional random vectors, a huge variety of extremal dependence scenarios can occur, in the sense of which groups of variables can be extreme simultaneously. \citet{Goixetal17} and \citet{Simpsonetal20} considered approaches to estimating groups of variables experiencing co-extreme behaviour. The dependence coefficients introduced by \citet{Simpsonetal20} for this purpose were shown to be connected to the \emph{limit set} of the random vector in light-tailed margins, where it exists, by \citet{NoldeWadsworth22}. This \emph{geometric} representation of multivariate extremes has recently been translated to a statistical model for the multivariate tail by \citet{WadsworthCampbell24} and \citet{Papastathopoulosetal24}. A major advantage of this new modelling framework in comparison to existing approaches for multivariate extremes is its ability to be able to capture all kinds of complex multivariate extremal dependence.

In this paper we introduce the concept of graphical models in the geometric extremes framework, termed \emph{geometric extremal graphical models}, and prove a selection of results relating to the propagation of extremal dependence coefficients along the graph. This definition and results lay the groundwork for statistical inference for geometric extremal graphical models, using the approaches in \citet{WadsworthCampbell24} and \citet{Papastathopoulosetal24}. In particular, while the focus of this work is on various theoretical properties of geometric extremal graphical models, the definition is readily exploitable in a statistical inference setting.

In Section~\ref{sec:geometric} we provide further background on the geometric representation for multivariate extremes and introduce geometric extremal graphical models. In Section~\ref{sec:Dependencecoeffs}, we show that a key summary of extremal dependence that relates to the conditional extremes framework of \citet{hefftawn04} and \citet{heffres07} has a natural factorization over the structure of a block graph. Furthermore, a second coefficient from this framework is shown to aggregate through a product or maximum operation, depending on whether the first coefficient is zero or not. In Section~\ref{sec:jointex}, we provide some results relating to the occurrence of joint extremes. Section~\ref{sec:discussion} concludes.

\section{Geometric representation of multivariate extremes}
\label{sec:geometric}
\subsection{Limit sets}

Consider a random vector $\bm{X}$ with common light-tailed margins, and let $\bm{X}_k$, $k=1,\ldots,n$ represent independent and identically distributed copies of $\bm{X}$. Denote the scaled $n$-point sample cloud by
\[
  N_n = \left\{\frac{\bm X_1}{r_n},\ldots, \frac{\bm{X}_n}{r_n}\right\}, 
\]
where the sequence $r_n$ depends on the precise form of the marginals of $\bm{X}$: a suitable choice of sequence is one that is asymptotically equivalent to the $1-1/n$ quantile. We are interested in cases where the random set $N_n$ converges in probability onto a compact limit set $G$ containing at least two points. Necessary and sufficient conditions for this have been given in \citet{Balkemaetal10}. \citet{Davisetal88} and \citet{KinoshitaResnick91} also consider convergence on to limit sets under various assumptions.

The shape of the limit set $G$ is affected by both the particular choice of light-tailed margins, and the extremal dependence structure between the components of $\bm{X}$. Our particular interest is in the information encoded by $G$ on the extremal dependence structure, and as such we take a copula-like approach and consider standardized margins. Two particularly clean choices are standard exponential and Laplace margins, i.e., where a single component $X_{k,i}$, $i=1,\ldots,d$ has respective densities
\begin{align*}
    f^E(x) = \exp(-x) \mathbbm{1}(x>0),\qquad \mbox{or}\qquad f^L(x) = \frac{1}{2}\exp(-|x|).
\end{align*}

Exponential margins are simplest to work with when only positive extremal association arises, but Laplace margins provide more detail when negative association can arise as well \citep{NoldeWadsworth22,Papastathopoulosetal24}. 

Suppose specifically now that the random vector $\bm{X} = (X_i : i \in \Vset)$ has
standard exponential or Laplace margins, and Lebesgue
joint density $f(\bm{x})$. A sufficient condition for the rescaled
$n$-point sample cloud
\[
  N_n = \left\{\frac{\bm X_1}{\log n},\ldots, \frac{\bm{X}_n}{\log
      n}\right\}
\]
to converge onto a limit set $G = \{\bm x \in \mathcal{S}^d\,:\,g(\bm x) \leq 1\}$ as
$n \to \infty$ is that the density $f$ satisfies
\begin{equation}                
  -\log f(t\bm x_{\level})/t \to g(\bm x), \quad \text{$\level \to \infty$},\quad \bm x_{\level} \to \bm x, \quad \bm x \in \mathcal{S}^d,
  \label{eq:gauge_densconv}
\end{equation}
for continuous $g$, where $\mathcal{S}^d = \mathcal{S} \times \cdots\times\mathcal{S}$, with $\mathcal{S}=[0,\infty)$ for exponential margins, and $(-\infty,\infty)$ for Laplace margins \citep{BalkemaNolde10,NoldeWadsworth22}. The limit set $G$ is star-shaped, and the 1-homogeneous function $g$, which describes the boundary of $G$, is termed the \emph{gauge function}.\
Lower dimensional marginal gauge functions are found through the
following minimization operation \citep[][Proposition
2.4]{NoldeWadsworth22}:
\begin{equation}
  g_J(\bm{x}_J) = \min_{x_i \in \mathcal{S}\,:\, i \not \in J} g(\bm{x}),
  \label{eq:min_gauge}
\end{equation}
where $\bm{x}_J = (x_j:j \in J)$ for any index set
$J \subseteq\{1,\ldots,d\}$, and $g_{J}$ is the marginal gauge
function for the variables in $J$.\ Note that exponential/Laplace margins
entails $g_{\{j\}}(x_j) = |x_j|$, which imposes a general constraint
$g(\bm{x}) \geq \max(|x_i|: i \in \Vset)$.

\subsection{Geometric extremal graphical models}
When $\bm{X}$ follows a decomposable graphical model and the density
convergence~\eqref{eq:gauge_densconv} holds for both full joint and lower dimensional densities, then
\begin{align}
  g(\bm x) = \sum_{C \in \mathcal{C}} g_{C}(\bm x_{C}) - \sum_{D \in \mathcal{D}} g_{D}(\bm x_{D}).
  \label{eq:extDGM}
\end{align}
Notice that while factorization~\eqref{eq:decompGM} and
convergence~\eqref{eq:gauge_densconv} imply~\eqref{eq:extDGM}, the converse
is not true in general. Equation~\eqref{eq:extDGM} forms the basis of our definition of a geometric extremal graphical model.\ 
\begin{definition}[Geometric extremal graphical model]
  The random vector $\bm X$ is said to follow a (decomposable) geometric extremal graphical model relative to a graph $\GG$ if
  convergence~\eqref{eq:gauge_densconv} and equation~\eqref{eq:extDGM}
  hold. Factorization~\eqref{eq:decompGM} is not required to hold.
\end{definition}

\begin{remark}
  The definition of extremal graphical models can be extended simply
  to more complex types of graphical model. For example, if
  $\mathcal{G}$ is not decomposable, then \eqref{eq:extDGM} still
  holds with $\mathcal{C}$ replaced by the set of prime components,
  that is, the maximal subgraphs of $\mathcal{G}$ that cannot be
  decomposed. 
\end{remark}
As mentioned, our primary focus in this work is on block graphs, a
special type of decomposable graphical model in which the separator
sets are singletons. Because of the fact that
$g_{\{j\}}(x_j) = |x_j|$, this has a simplifying effect on the form of
the geometric extremal graphical model.

\begin{definition}
    A \emph{block geometric extremal graphical model} is defined through the gauge function
\begin{align}
    g(\bm{x}) = \sum_{C \in \mathcal{C}} g_C(\bm{x}_C) - \sum_{D \in \mathcal{D}} |x_D|. \label{eq:bgg}
\end{align}
\end{definition}

For later use, we also define special cases of block geometric extremal graphical models: tree and chain geometric extremal graphical models. A tree is a block graph in which all cliques are of size two, while a chain is a tree graph for which all separators appear only once in $\mathcal{D}$.

\begin{definition}
 A \emph{tree geometric extremal graphical model} with vertex set $\Vset$ and edge set $\Eset = \mathcal{C}$ is defined through the gauge function 
\begin{align}
  g(\bm{x}) = \sum_{(i,j) \in \mathcal{E}} g_{\{i,j\}}(x_i,x_j) - |x_i| - |x_j|
+ \sum_{k \in \Vset} |x_k|. \label{eq:tegm}
\end{align}
\end{definition}

\begin{definition}
 A \emph{chain geometric extremal graphical model} with vertex set $\Vset=\{1,2,\ldots,d\}$ has edge set $\Eset=\mathcal{C}=\{(1,2),(2,3),\ldots,(d-1,d)\}$ is, defined through the gauge function 
\begin{align}
  g(\bm{x}) &=\sum_{k=1}^{d-1} g_{\{k,k+1\}}(x_{k},x_{k+1}) - \sum_{k=2}^{d-1} |x_k|. \label{eq:cegm}
\end{align}
\end{definition}

We remark that in the case of exponential margins, the absolute value bars in equations~\eqref{eq:bgg},~\eqref{eq:tegm} and~\eqref{eq:cegm} are not necessary, since all $x_j \geq 0$.

\begin{example}\normalfont
\label{ex:3graphs}
Figure~\ref{fig:graphs} displays examples of general block, tree and chain graphs. In the first panel, the cliques and separators are $\mathcal{C}=\{\{1,2,3\},\{3,4\},\{4,5\},\{4,6\},\{6,7,8,9\}\}$ and $\mathcal{D} = \{3,4,4,6\}$. In the second panel, the cliques and separators are $\mathcal{C}=\{\{1,2\},\{2,3\},\{2,4\},\{4,5\},\{4,6\}\}$, $\mathcal{D}=\{2,2,4,4\}$. In the third panel the cliques and separators are $\mathcal{C}=\{\{1,2\},\{2,3\},\{3,4\},\{4,5\}\}$, $\mathcal{D}=\{2,3,4\}$.  

\end{example}

\begin{figure}
    \centering
\includegraphics[width=0.3\linewidth]{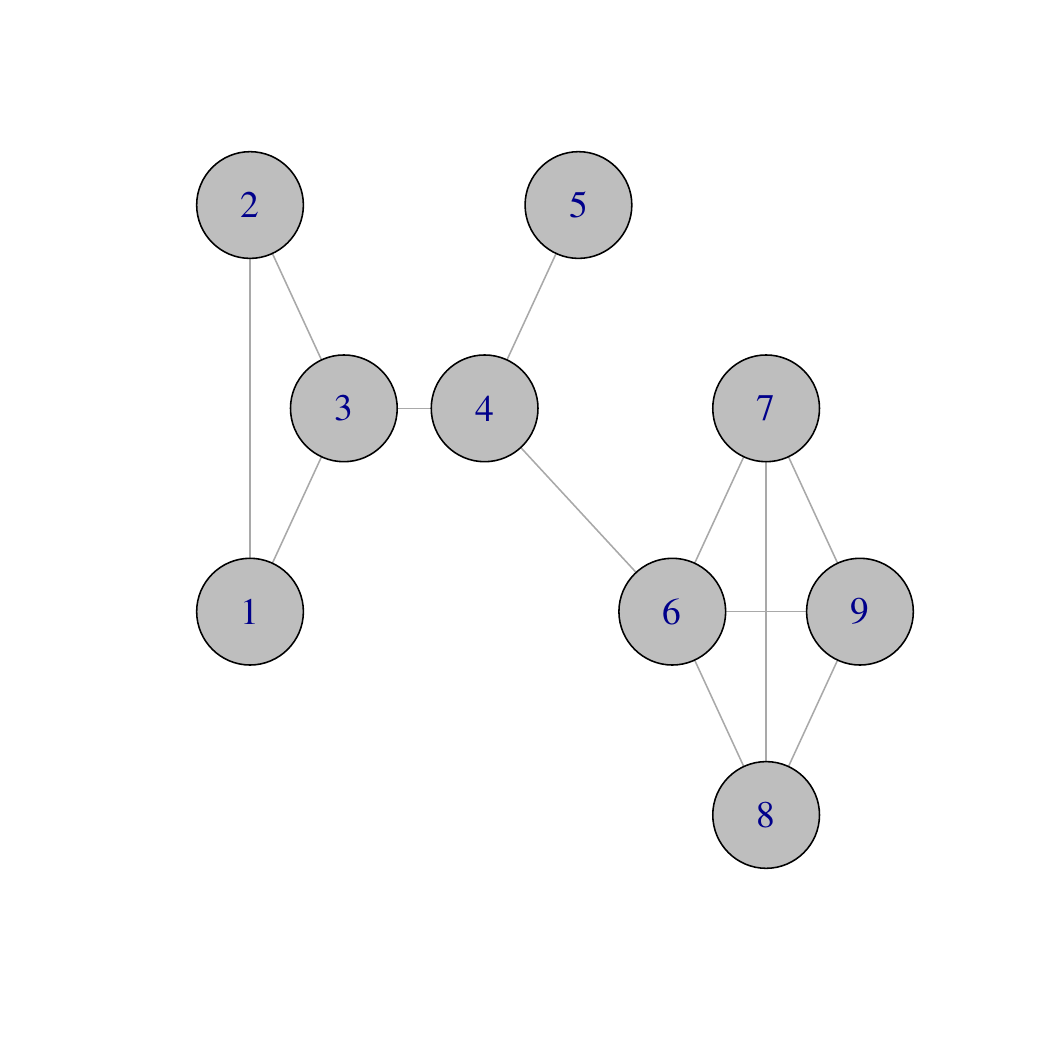}
\includegraphics[width=0.3\linewidth]{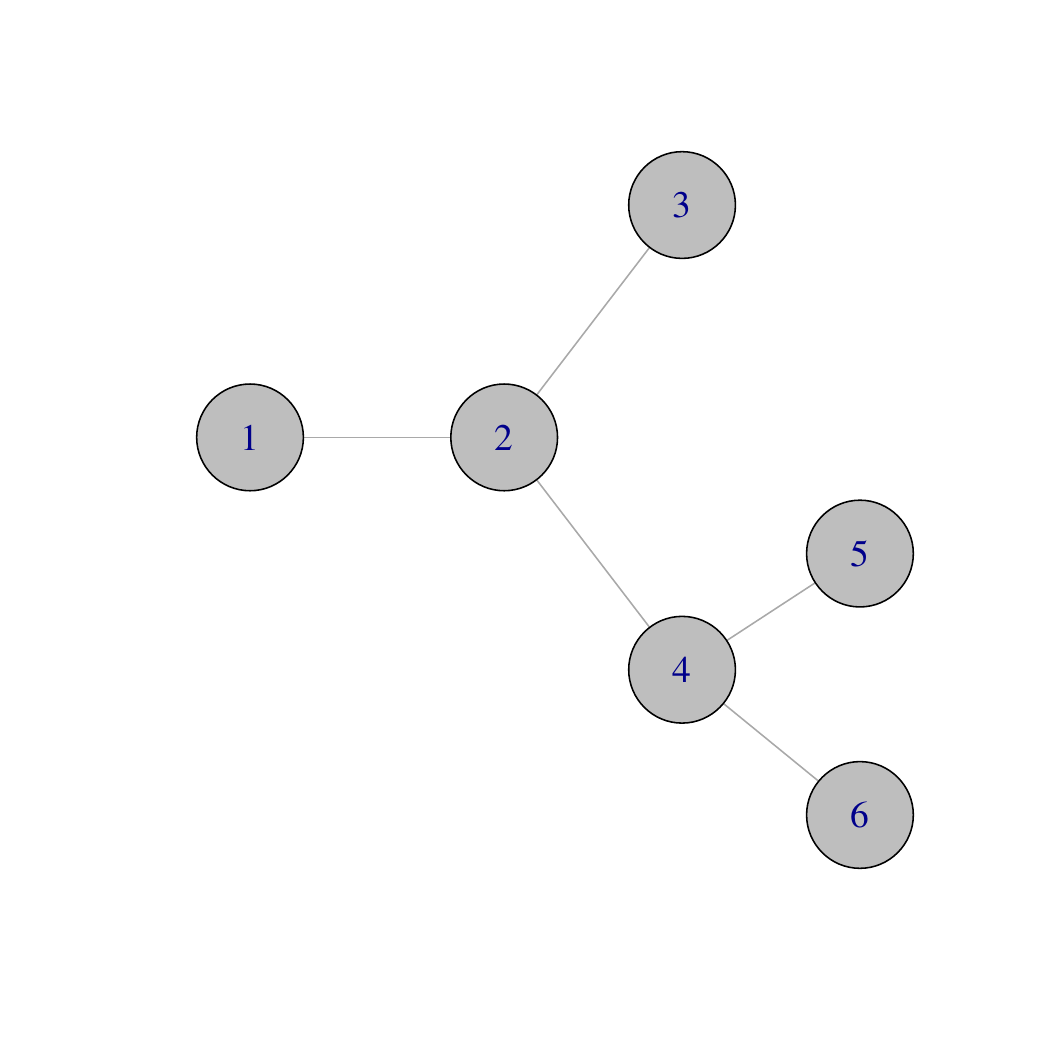}
\includegraphics[width=0.3\linewidth]{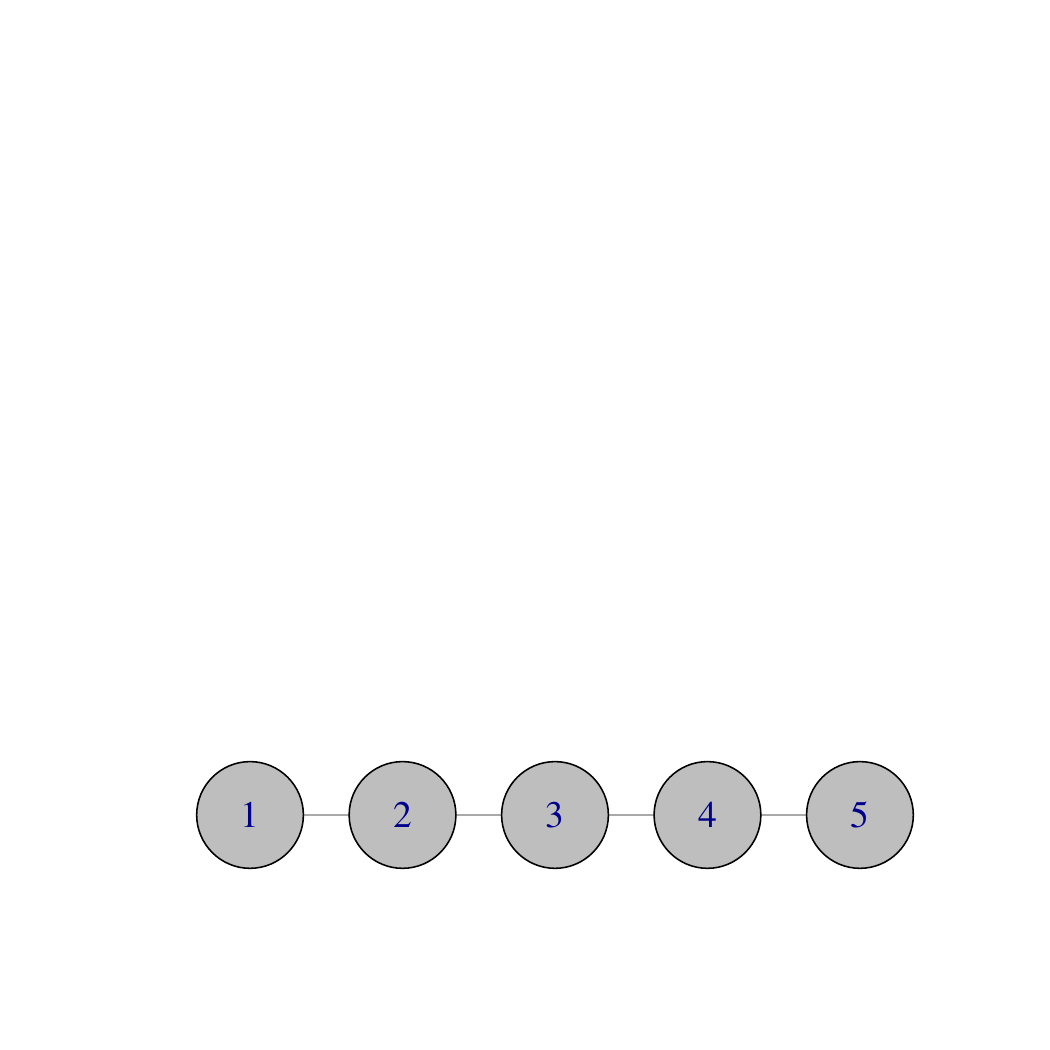}
    \caption{Examples of block, tree and chain graphs, respectively (see Example~\ref{ex:3graphs}).}
    \label{fig:graphs}
\end{figure}

\section{Conditional extremes dependence coefficients on the graph}
\label{sec:Dependencecoeffs}
In this section, we outline results related to the dependence coefficients that arise in the framework of the conditional extreme value model \citep{hefftawn04,heffres07}. We focus initially on exponential margins, subsequently adapting results to the Laplace margin case. This simplifies the presentation in the exponential margin case, which is useful when only positive dependence is observed between variables.

\subsection{Exponential margins}
\label{sec:depexp}

\subsubsection{Conditional extremes assumptions and notation}
\label{sec:cexassnot}

\citet{NoldeWadsworth22} linked several different extremal dependence coefficients to the geometry of the limit set $G$, following on from the work of \citet{Nolde14}. One important coefficient, $\bm{\alpha}_{\Vset \sm \{i\}}$, defined in Assumption~\ref{ass:ce_conditions} below, links to the conditional extreme value framework of \citet{hefftawn04} and \citet{heffres07}.

The basic assumption assumption in conditional extreme value theory is
that for each $i\in\Vset$, there exist functions
$\bm a_{\Vset \sm \{i\} \mid i}\,:\,\RR \to \RR^{d-1}$ and
$\bm b_{\Vset \sm \{i\} \mid i}\,:\,\RR \to \RR_+^{d-1}$, such
that
\begin{equation}
  \PP\left(X_i-\level > x, \frac{\bm X_{-i}- \bm a_{\Vset
        \sm \{i\} \mid i}(X_{i})}{\bm b_{\Vset\sm \{i\}\mid
        i}(X_i)}\leq \bm z~\Big|~X_{i} >
    \level \right) \to \exp(-x)\,K_{\Vset \sm \{i\} \mid
    i}(\bm z), \quad \text{as $\level\to \infty$,}
  \label{eq:ce_dist_conv}
\end{equation}
at continuity points
$(x, \bm z) \in \mathbb{R}_+\times \mathbb{R}^{d-1}$ of
the limit. The quantity $K_{\Vset \sm \{i\} \mid i}$ is a distribution
function on $\mathbb{R}^{d-1}$ satisfying
$\lim_{\level\to\infty}K_{\Vset\sm \{i\} \mid i}(\bm t_j) = 1$
for all $j\in \Vset\sm \{i\}$, and $\bm t_j$ denotes the vector
that has its $j$th element equal to $t$ and has all its other elements
equal to infinity.

When joint densities and relevant limits exist, application of
L'H\^{o}pital's rule gives that convergence~\eqref{eq:ce_dist_conv} is
equivalent to
\begin{equation}
  \PP\left(\frac{\bm X_{-i}- \bm a_{\Vset
        \sm \{i\} \mid i}(\level)}{\bm b_{\Vset\sm \{i\}\mid
        i}(\level)}\leq \bm z~\Big|~X_{i} =
    \level \right) \to K_{\Vset \sm \{i\} \mid
    i}(\bm z),  \quad \text{as $\level\to \infty$.}
  \label{eq:ce_dist_conv_kernel}
\end{equation}
Below in Assumption \ref{ass:ce_conditions}, we also require
convergence of the joint density and restriction on the support of the marginal distributions of $K_{\Vset \sm \{i\} \mid
    i}(\bm z)$.
  \begin{assumption}
    \label{ass:ce_conditions}
    For all $i \in \Vset$,
    \begin{equation}
      \frac{\partial^{d-1}}{\partial \bm z} \PP\left(\frac{\bm X_{-i} -
          \bm a_{\Vset \sm \{i\}\mid i}(\level)}{\bm b_{\Vset \sm \{i\}\mid i }(\level)}
        \leq \bm z \mid X_{i} = \level\right)
      \to \frac{\partial^{d-1}}{\partial \bm z}\,K_{\Vset \sm \{i\} \mid i}(\bm
      z) =: k_{\Vset \sm \{i\}\mid i}(\bm z), \quad \text{as $\level\to \infty$.}
      \label{eq:ce_densconv}
    \end{equation}
    
    Additionally,
    \begin{itemize}[wide=0\parindent]
    \item[$(i)$] For each $j\in\Vset \sm \{i\}$, the support $\supp(K_{j\mid i})$ of the marginal
      distribution $K_{j\mid i}$ of
      $K_{\Vset \sm \{i\} \mid i}$ includes $(0,\infty)$;
    \item[$(ii)$]
      $\bm \alpha_{\Vset \sm \{i\} \mid i}:=\lim_{\level \to
        \infty} \bm a_{\Vset \sm \{i\} \mid i}(t)/t$ exists in
      $[0,1]^{d-1}$.
    \end{itemize}
  \end{assumption}

\citet{NoldeWadsworth22}
showed that under a bivariate version of Assumption~\ref{ass:ce_conditions} and convergence
\eqref{eq:gauge_densconv}, then for
$j\in \Vset\sm \{i\}$,
\begin{equation}
  \alpha_{j\mid i} = \max \{\widetilde{\alpha}_{j\mid
    i}\,:\,g_{\{i,j\}}(1, \widetilde{\alpha}_{j\mid
    i})=1\}.
  \label{eq:gauge_alpha_condition}
\end{equation}
The main reason that \citet{NoldeWadsworth22} considered only
bivariate representations for conditional extremes, is that the
dependence functions $a_{j\mid i}$ and $b_{j\mid i}$,
$j \in \Vset \sm \{i\}$, are determined by pairwise dependences
between $(X_i, X_j)$. However, to make connections between
$\alpha$-coefficients in general dimensional settings, as we do below, we
require joint convergence properties.

\subsubsection{\texorpdfstring{$\alpha$}{a}-coefficients}
\label{sec:alpha_coeffs}
Proposition~\ref{prop:joint_convergence} extends equation~\eqref{eq:gauge_alpha_condition} to multidimensional settings. In the below, we define the function $g_{\{i\} \cup \{\Vset \sm \{i\}\}}(x_i,\bm{x}_{\Vset \sm \{i\}}) = g(\bm{x})$ to be the gauge function with arguments re-ordered from $\Vset$ to $\{i,\Vset \sm \{i\}\}$.

\begin{proposition}
  \label{prop:joint_convergence}
  Suppose that for $\bm X=(X_i\,:\,i\in \Vset)$, the
  convergence \eqref{eq:gauge_densconv} and Assumption
  \ref{ass:ce_conditions}
  holds.
  Then,
  \begin{itemize}[wide=0\parindent]
  \item[$(i)$] for all $i\in \Vset$,
    $g_{\{i\} \cup \{\Vset \sm \{i\}\}}(1, \bm \alpha_{\Vset
      \sm \{i\} \mid i}) = 1$;

  \item[$(ii)$] for all $i\in \Vset$, if there are multiple
    vectors $\bm \alpha$ satisfying
    $g_{\{i\} \cup \{\Vset \sm \{i\}\}}(1, \bm \alpha) = 1$,
    then the coordinate-wise maximum such vector $\bm \alpha^\star$
    also satisfies
    $g_{\{i\} \cup \{\Vset \sm \{i\}\}}(1, \bm \alpha^\star)
    = 1$ and
    $\bm \alpha_{\Vset \sm \{i\} \mid i} = \bm
    \alpha^\star$.
  \end{itemize}
\end{proposition}

\begin{remark}
\label{rmk:alphamin}
Since $g(\bm x) \geq \max(x_i: i \in \Vset)$, Proposition~\ref{prop:joint_convergence} implies that
$\bm \alpha_{\Vset \sm \{i\}}$ is a global minimizer, that is,
$g_{\{i\}\cup \{\Vset\sm \{i\}\}}(1, \bm
\alpha_{\Vset\sm\{i\}\mid i}) \leq g_{\{i\}\cup
  \{\Vset\sm \{i\}\}}(1, \bm x)$ for all $\bm
x\in\RR_+^{d-1}$. 
\end{remark}

The proof of Proposition~\ref{prop:joint_convergence} is in
Appendix~\ref{sec:proofs}.

One of our main results, given in Proposition~\ref{prop:block_graph_alphas}, concerns the relationship between $\alpha$-coefficients when the random vector $\bm{X}$ follows a geometric extremal graphical model based on a block graph. For any pair of vertices $(i,j) \in \Vset$, we find that $\alpha_{j\mid i}$ is given by the product of pathwise $\alpha$-coefficients along the unique shortest path between $\{i\}$ and $\{j\}$. 
\begin{proposition}
  \label{prop:block_graph_alphas}
  Suppose that $\bm X= (X_i:i \in \Vset)$ follows a block geometric
  extremal graphical model with graph $\GG=(\Vset, \Eset)$, and that
  the assumptions in Proposition \ref{prop:joint_convergence} hold.\
  For any two distinct vertices $i$ and $j$ in $\Vset$, let
  $\{i=v_0, v_1,\ldots,v_{m_{i j}}=j\}$ be the vertices along the
  unique shortest path of length $m_{ij} \geq 1$ from $i$ to $j$. Then
  \begin{equation}
    \alpha_{j\mid i} = \prod_{k=1}^{m_{i j}}
    \alpha_{v_k\mid v_{k-1}}.
    \label{eq:product_alphas}
  \end{equation}
\end{proposition}

A proof of Proposition~\ref{prop:block_graph_alphas} is in
Appendix~\ref{sec:proofs}. We finish this section by illustrating Proposition~\ref{prop:block_graph_alphas} with a simple example. The intuition that may be obtained from this example underlies the general proof.

\begin{example}\normalfont
  Consider the tree geometric extremal graphical model, as defined in equation~\eqref{eq:tegm}, with $\Vset=\{1,2,3,4\}$, $\mathcal{C}=\mathcal{E} = \{(1,2),(2,3),(2,4)\}$, and $\mathcal{D}=\{2,2\}$. Recalling that for exponential margins, $x_i \geq 0$ for all $i \in \Vset$, we have
  \begin{align*}
    g(x_1,x_2,x_3,x_4) = g_{\{1,2\}}(x_1,x_2)+g_{\{2,3\}}(x_2,x_3) +
    g_{\{2,4\}}(x_2,x_4) - 2x_2.
  \end{align*}
  Suppose that we wish to calculate the coefficient $\alpha_{4|3}$. We
  start by noting that
  \begin{align}
    g_{\{3,4\}}(x_3,x_4) = \min_{x_1\geq 0, x_2 \geq 0}  g_{\{1,2\}}(x_1,x_2)+ g_{\{2,3\}}(x_2,x_3) + g_{\{2,4\}}(x_2,x_4) - 2x_2. \label{eq:alphaeg1}
  \end{align}
Consider specifically $g_{\{3,4\}}(1,\alpha_{4|3}) = 1$. Then we can write
  \begin{align*}
    g_{\{3,4\}}(1,\alpha_{4|3}) = g_{\{1,2\}}(x_1^\star,x_2^\star)+ g_{\{2,3\}}(x_2^\star,1) + g_{\{2,4\}}(x_2^\star,\alpha_{4|3}) - 2x_2^\star,
  \end{align*}
  where $x_1^\star,x_2^\star$ represent the values at which the minimum in equation~\eqref{eq:alphaeg1} is
  achieved for $x_3=1, x_4=\alpha_{4|3}$. Using the observation in Remark~\ref{rmk:alphamin}, $x_{1}^\star = \alpha_{1|3}$ and $x_2^\star = \alpha_{2|3}$, i.e.,
    \begin{align*}
    g_{\{3,4\}}(1,\alpha_{4|3}) = g_{\{1,2\}}(\alpha_{1|3},\alpha_{2|3})+ g_{\{2,3\}}(\alpha_{2|3},1) + g_{\{2,4\}}(\alpha_{2|3},\alpha_{4|3}) - 2\alpha_{2|3}.
  \end{align*}
  We split the right-hand side into a sum
  of two components:
  \begin{align}
    g_{\{2,3\}}(\alpha_{2|3},1) &= 1 \label{eq:comp1}\qquad\mbox{and}\\
    g_{\{1,2\}}(\alpha_{1|3},\alpha_{2|3})+ g_{\{2,4\}}(\alpha_{2|3},\alpha_{4|3}) - 2\alpha_{2|3} &\geq 0. \label{eq:comp2}
  \end{align}
  Equation~\eqref{eq:comp1} holds due to equation~\eqref{eq:gauge_alpha_condition}; equation~\eqref{eq:comp2} arises
  since $g_{\{1,2\}}(\alpha_{1|3},\alpha_{2|3}) \geq \alpha_{2|3}$ and
  $g_{\{2,4\}}(\alpha_{2|3},\alpha_{4|3}) \geq \alpha_{2|3}$. However, because we require $g_{\{3,4\}}(1,\alpha_{4|3}) = 1$, we must have
  that~\eqref{eq:comp2} equals 0. Suppose initially that $\alpha_{2|3} > 0$, so that using homogeneity of $g$,
  \begin{align*}
   g_{\{1,2\}}(\alpha_{1|3},\alpha_{2|3})+ g_{\{2,4\}}(\alpha_{2|3},\alpha_{4|3}) - 2\alpha_{2|3}  &= \alpha_{2|3} g_{\{1,2\}}(\alpha_{1|3}/\alpha_{2|3},1)+ \alpha_{2|3} g_{\{2,4\}}(1,\alpha_{4|3}/\alpha_{2|3}) - 2\alpha_{2|3} \\
& = 0.
  \end{align*}
  From this we deduce
  $\alpha_{1|3}/\alpha_{2|3} = \alpha_{1|2}$, and
  $\alpha_{4|3}/\alpha_{2|3} = \alpha_{4|2}$, i.e.,
  $\alpha_{4|3} = \alpha_{4|2} \alpha_{2|3}$, which is
  the product of pathwise coefficients. If we have $\alpha_{2|3}=0$ then 
    \begin{align*}
   g_{\{1,2\}}(\alpha_{1|3},\alpha_{2|3})+ g_{\{2,4\}}(\alpha_{2|3},\alpha_{4|3}) - 2\alpha_{2|3} &=  g_{\{1,2\}}(\alpha_{1|3},0)+ g_{\{2,4\}}(0,\alpha_{4|3}) = 0,
  \end{align*}
  which implies $g_{\{1,2\}}(\alpha_{1|3},0)=g_{\{2,4\}}(0,\alpha_{4|3})=0$, but since $g_{\{2,4\}}(0,\alpha_{4|3}) \geq \max(0, \alpha_{4|3})$ this implies $\alpha_{4|3}=0$ (similarly $\alpha_{1|3}$).
\end{example}

\subsubsection{\texorpdfstring{$\beta$}{b}-coefficients}
\label{sec:beta_coefficients}
The $\alpha$-coefficients described in Section~\ref{sec:alpha_coeffs} are
the key descriptors of the location normalization functions
$\bm{a}_{\Vset \sm \{i\}\mid i}$ in the conditional extremes convergence
assumption~\eqref{eq:ce_dist_conv}. The scale normalization functions
$\bm{b}_{\Vset \sm \{i\}\mid i}$ are generally characterized through a
regular variation assumption, and the key dependence quantity is the
index of regular variation. If $u:\mathbb{R}^+\to\mathbb{R}^+$ is a regularly varying
function at infinity with index $\gamma \in \mathbb{R}$, then for
$x>0$, $\lim_{t\to \infty} u(tx)/u(t) = x^\gamma$. We write
$u \in \RV_\gamma^\infty$. Regular variation at $0^+$ is defined similarly, and we
write $u \in \RV_\gamma^{0^+}$. The vector of scale normalization functions
$\bm{b}_{\Vset \sm \{i\}\mid i}$ is assumed to have components
$b_{j|i} \in \RV^\infty_{\beta_{j|i}}$ for $\beta_{j|i} \in [0,1)$. The focus of this
section is the structure of these $\beta$-coefficients along the
graph. 
The related analysis is considerably more complicated than the case of
the $\alpha$-coefficients. 
\color{black}
We use a
key result from \cite{NoldeWadsworth22}, who demonstrated that under a
bivariate version of convergence \eqref{eq:gauge_densconv}, the
function $g_{\{i,j\}}$ can determine the coefficient $\beta_{j\mid i}$ in the
sense that if
\begin{equation}
  g_{\{i,j\}}(1, \alpha_{j\mid i} + \cdot)-1\in
  \text{RV}_{1/(1-\beta_{j\mid i})}^{0^+},
  \label{eq:original_beta_condition}
\end{equation}
then $b_{j\mid i}\in \text{RV}_{\beta_{j\mid i}}^\infty$.\

In Proposition \ref{prop:block_graph_betas} below, we show that
$\beta_{j\mid i}$ also conforms to a structure along the shortest path
in $\GG$, but its form can be much more intricate due to the
dependence on the $\alpha$-coefficients along this path. This dependency
complicates the structure of $\beta_{j\mid i}$ beyond the
straightforward product form identified in
Proposition~\ref{prop:block_graph_alphas} for $\alpha_{j\mid i}$. However,
the analysis of $\alpha$-coefficients can inspire a strategy for addressing
the complexity of $\beta$-coefficients, leading to the development of a
relatively simple recurrence relation, given in equation
\eqref{eq:beta_recurrence} of
Proposition~\ref{prop:block_graph_betas}. We briefly explain this
strategy through a revised analysis of $\alpha$-coefficients.

Due to the marginalization properties of graphical gauge functions,
which mirror those seen in the marginalization of probabilistic
graphical models \citep{Koster02}, the form for
$\alpha_{j\mid i}$ can be formally written as the solution to the
recurrence relation
$\alpha_{j\mid i} = \alpha_{\pi\mid i} \alpha_{j\mid \pi}$, with initial condition
$\alpha_{i\mid i} = 1$, where $\pi$ denotes the penultimate node in
the shortest path from $i$ to $j$ in $\GG$. Specifically, in
Lemma~\ref{lem:blockgraphtopath} of Appendix~\ref{app:proofbetas}, we
show that for block graphs, we can express $g_{\{i,j\}}(x_i,x_j)$ in
terms of a chain graph with nodes lying on the shortest path between
$i$ and $j$, while in Lemma~\ref{lem:graph_marginalization}, we show
that marginalizing a chain graph also gives a chain
graph. Consequently, when $\GG$ is a block graph, then marginalizing
over all nodes except for $(i$, $\pi$,
$j)$ 
results in the chain graphical gauge
\begin{equation}  
  g_{\{i, \pi, j\}}(x_i, x_\pi, x_j)= g_{\{i, \pi\}}(x_{i}, x_{\pi}) +
  g_{\{\pi, j\}}(x_{\pi}, x_{j}) - x_{\pi},
  \label{eq:chain_gauge}
\end{equation}
which, by virtue of Proposition \ref{prop:block_graph_alphas}, shows
formally why the form of $\alpha_{j\mid i}$ also arises as the solution
of the aforementioned recurrence relation. These structural
properties facilitate a method of building the
recurrence relation \eqref{eq:beta_recurrence} in Proposition~\ref{prop:block_graph_betas}. The ingredients are equation~\eqref{eq:original_beta_condition}, the relation
\begin{equation}
  g_{\{i, j\}}(1, \alpha_{j\mid i} + x)= \min_{\varepsilon \in \mathcal{S}-\{\alpha_{\pi\mid i}\}} g_{\{i,\pi,
    j\}}(1,\alpha_{\pi\mid i} + \varepsilon, \alpha_{j\mid i} + x),
  \label{eq:min_trivariate_gauge}
\end{equation}
where
$\mathcal{S}-\{\alpha_{k\mid i}\} = \{x_k-\alpha_{k\mid i}\,:\, x_k \in \mathcal{S}\}$, which
follows from \eqref{eq:min_gauge}, and Assumption~\ref{ass:edge-RV}
below, which imposes a simple edge‑level regularity and assigns each
edge a single scaling exponent.
\begin{assumption}
  \label{ass:edge-RV}
  For each edge $\{v,v'\}\in\Eset$, with $\alpha_{v'\mid v}$ the rightmost
  minimizer of $y\mapsto g_{\{v,v'\}}(1,y)$, there exists
  $\beta_{v'\mid v}\in[0,1)$ such that
  $g_{\{v,v'\}}(1,\alpha_{v'\mid v}+ \cdot )-1\in \RV_{\sigma_{v'\mid v}}^{0}$ with
  $\sigma_{v'\mid v}:=1/(1-\beta_{v'\mid v})\ge 1$. Moreover,
  $g_{\{v,v'\}}$ is positive, continuous, $1$-homogeneous, satisfying
  $g_{\{v,v'\}}(x,y)\ge\max(x,y)$, and the identity
  \eqref{eq:chain_gauge} holds.
\end{assumption}
Near the contact point where the unit level set of $g_{\{v,v'\}}$ intersects with the upper bound $\{(x,y):\max(x,y)=1\}$ at $(1,\alpha_{v'\mid v})$, the bivariate gauge increases in a fixed power‑law manner. This
gives each edge a single scaling exponent $\beta_{v'\mid v}$. Together with
the chain identity \eqref{eq:chain_gauge}, these edge exponents
combine along paths and determine the pairwise exponent
$\beta_{j\mid i}$ via the recurrence \eqref{eq:beta_recurrence}.\color{black}

\begin{proposition}
  \label{prop:block_graph_betas}
  Suppose that $\bm X = (X_i:i \in \Vset)$ follows a geometric extremal graphical model
  relative to a block graph $\GG=(\Vset, \Eset)$, and that Assumption~\ref{ass:edge-RV} holds.\
  For any two distinct vertices $i$ and $j$ in $\Vset$, let
  $\{i=v_0, v_1,\ldots,v_{m_{i j}-1}=\pi, v_{m_{i j}}=j\}$ be the
  vertices along the unique shortest path of length $m_{ij} \geq 2$ from
  $i$ to $j$. Then,
  $g_{\{i, j\}}(1,\alpha_{j\mid i}+ \cdot )-1\in \RV_{\sigma_{j\mid
      i}}^{0^+}$ where $\sigma_{j\mid i}=1/(1-\beta_{j\mid i})$ with
  \begin{IEEEeqnarray}{rCl}
    \beta_{j\mid i} &=&
    \begin{dcases} 
      \max(\beta_{\pi\mid i},
      \beta_{j\mid \pi}) & \text{for }\alpha_{\pi\mid i} \neq 0 \text{~and~} \alpha_{j\mid
                  \pi} \neq 0, \\
      \beta_{\pi\mid i}  & \text{for }\alpha_{\pi\mid i} = 0 \text{~and~} \alpha_{j\mid
                       \pi} \neq 0, \\
      \beta_{j\mid \pi} & \text{for }\alpha_{\pi\mid i} \neq 0 \text{~and~} \alpha_{j\mid
                 \pi} = 0,\\
      \beta_{\pi\mid i} \beta_{j\mid \pi} & \text{for }\alpha_{\pi\mid i} = 0 \text{~and~}
                               \alpha_{j\mid \pi} = 0.
    \end{dcases}
    \label{eq:beta_recurrence}
  \end{IEEEeqnarray}
\end{proposition}
\noindent
A proof of Proposition~\ref{prop:block_graph_betas} is in
Appendix~\ref{app:proofbetas}. Two special cases of this proposition are worth highlighting:
\[
  \beta_{j\mid i} = \max_{k=1,\dots, m_{i j}} \beta_{v_k\mid v_{k-1}},\qquad
  \text{when $\min_{k=1,\dots, m_{i j}} \alpha_{v_k \mid v_{k-1}} > 0$}
\]
and
\[
  \beta_{j\mid i} = \prod_{k=1}^{m_{i j}} \beta_{v_k\mid v_{k-1}},\qquad
  \text{when $\max_{k=1,\dots, m_{i j}} \alpha_{v_k \mid v_{k-1}} = 0$}.
\]

\noindent
We illustrate Proposition~\ref{prop:block_graph_betas} with an example.
\begin{example}
\label{ex:betas}
    We consider a chain graph with four nodes and various parametric gauge functions along the three edges $\Eset = \{(1,2),(2,3),(3,4)\}$. By Lemma~\ref{lem:blockgraphtopath} this is equivalent to considering any nodes separated by a path of length three along a block graph. The unit level set of the gauge functions $g_{\{1,4\}}(x_1,x_4)$, obtained via numerical minimization, are displayed in Figure~\ref{fig:betaillustrations}. To illustrate the $\beta$ parameter more clearly, the bottom row of Figure~\ref{fig:betaillustrations}  displays $\log(g_{\{1,4\}}(1,\alpha_{4|1}+x)-1)$ against $\log(x)$ for $x>0$ small, with a superimposed slope of $1/(1-\beta_{4|1})$. In each of the cases below, the title gives gauges in the order $g_{\{1,2\}}$ -- $g_{\{2,3\}}$ -- $g_{\{3,4\}}$. The pairwise gauge functions used are given in Table~\ref{tab:pairwisegauges}. These gauge functions satisfy Assumption~\ref{ass:edge-RV} with $\alpha$ and $\beta$ values given in the final two columns of Table~\ref{tab:pairwisegauges}. All pairwise gauges in Table~\ref{tab:pairwisegauges} are symmetric in their arguments, meaning that for $(k,l) \in \mathcal{E}$,
    $\alpha_{k|l}=\alpha_{l|k}=\alpha$, $\beta_{k|l}=\beta_{l|k}=\beta$.
    \vspace{0.2cm}

\begin{table}
    \centering
        \caption{Pairwise gauge functions used in Example~\ref{ex:betas}.}
    \label{tab:pairwisegauges}
\begin{tabular}{lllll}
\textbf{Name} & $g(x_1,x_2)$ & Range & $\alpha$ & $\beta$\\\hline
  \textbf{Logistic}   &  $(x_1+x_2)/\theta+(1-2/\theta)\min(x_1,x_2)$ & $\theta \in (0,1)$ & $1$ & $0$\\
\textbf{Gaussian} & $(x_1+x_2-2\rho(x_1x_2)^{1/2})/(1-\rho^2)$ & $\rho\in [0,1)$ & $\rho^2$ & $1/2$\\
\textbf{Inverted Logistic}& $(x_1^{1/\theta}+x_2^{1/\theta})^{\theta}$ & $\theta \in (0,1]$ & $0$ &$1-\theta$\\
\textbf{Square} & $\max\{(x_1 - x_2)/\theta, (x_2 - x_1)/\theta, (x_1 + x_2)/(2 - \theta)\}$ & $\theta\in(0,1)$ & $1-\theta$ &$0$
\end{tabular}

\end{table}

\noindent
    \textbf{(a): Logistic -- Gaussian -- Logistic}. We have $\alpha_{4|1}=\alpha_{4|3}\alpha_{3|2}\alpha_{2|1} = 1 \times \rho_{23}^2 \times 1 = \rho_{23}^2 $. Taking $\rho_{23}>0$, all $\alpha$-coefficients are positive, giving $\beta_{4|1} = \max\{\beta_{4|3},\beta_{3|1}\}$, where $\beta_{3|1} = \max\{\beta_{3|2},\beta_{2|1}\} = \max\{1/2,0\}=1/2$. Hence $\beta_{4|1} = \max\{0,1/2\} = 1/2$.\vspace{0.2cm}

    \noindent
    \textbf{(b): Gaussian -- Gaussian -- Inverted Logistic}. Here $\alpha_{4|1}=\alpha_{4|3}\alpha_{3|2}\alpha_{2|1} = 0 \times \rho_{12}^2 \times \rho_{23}^2 = 0$. Since $\alpha_{4|3}=0$, $\beta_{4|1}=\beta_{4|3} = 1-\theta_{34}$, where $\theta_{34}\in(0,1]$ is the inverted logistic dependence parameter. \vspace{0.2cm}

    \noindent
    \textbf{(c): Inverted Logistic -- Logistic -- Inverted Logistic}.  Here $\alpha_{4|1}=\alpha_{4|3}\alpha_{3|2}\alpha_{2|1} = 0 \times 1 \times 0 = 0$. Since both $\alpha_{4|3}=0$ and $\alpha_{3|1}=\alpha_{3|2}\alpha_{2|1}=0$, $\beta_{4|1}=\beta_{4|3}\beta_{3|1}$. As in (b), $\beta_{4|3} = 1-\theta_{34}$, while $\beta_{3|1} = \beta_{2|1} = 1-\theta_{12}$ since $\alpha_{3|2}>0$ and $\alpha_{2|1}=0$. Overall therefore $\beta_{4|1}= (1-\theta_{34})\times(1-\theta_{12})$.\vspace{0.2cm}

    \noindent
    \textbf{(d): Logistic -- Square -- Square}. Here $\alpha_{4|1}=\alpha_{4|3}\alpha_{3|2}\alpha_{2|1} =  (1-\theta_{23}) \times (1-\theta_{34}) \times 1 > 0$. Since all $\alpha$-coefficients are positive, $\beta_{4|1} = \max\{\beta_{4|3},\beta_{3|1}\} = \max\{\beta_{4|3},\max\{\beta_{3|2},\beta_{3|1}\}\} = \max\{0,0,0\} = 0$. 

    \begin{figure}
      \centering    \includegraphics[width=0.24\linewidth]{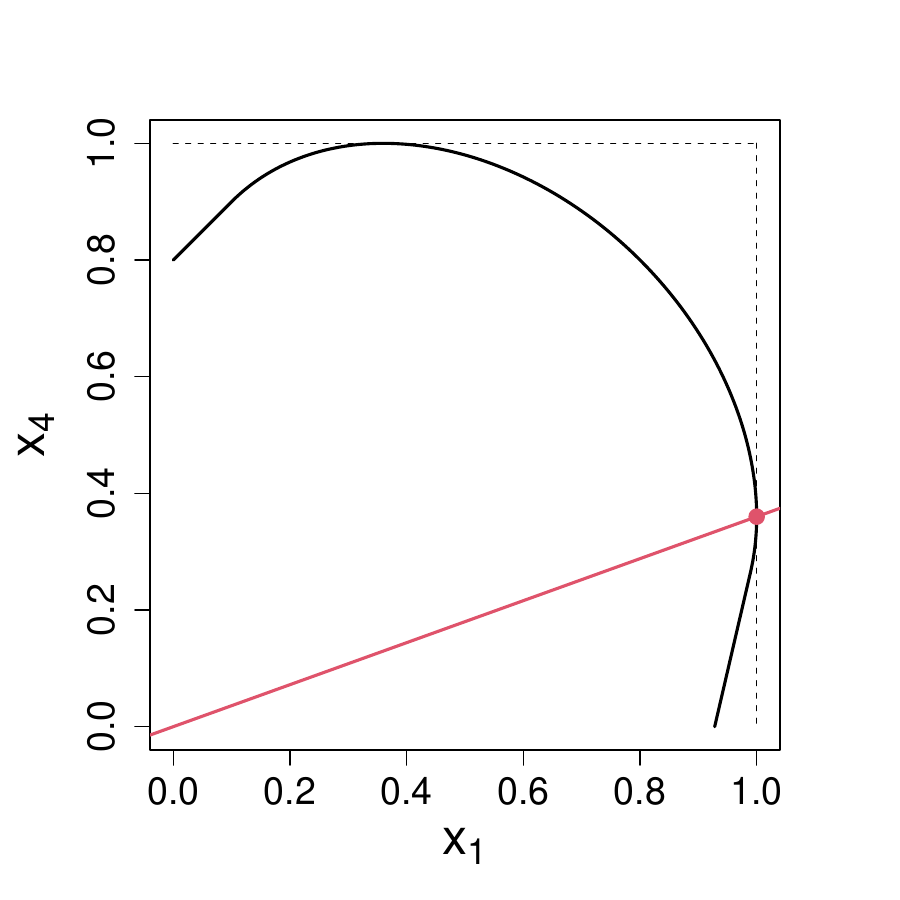}
      \includegraphics[width=0.24\linewidth]{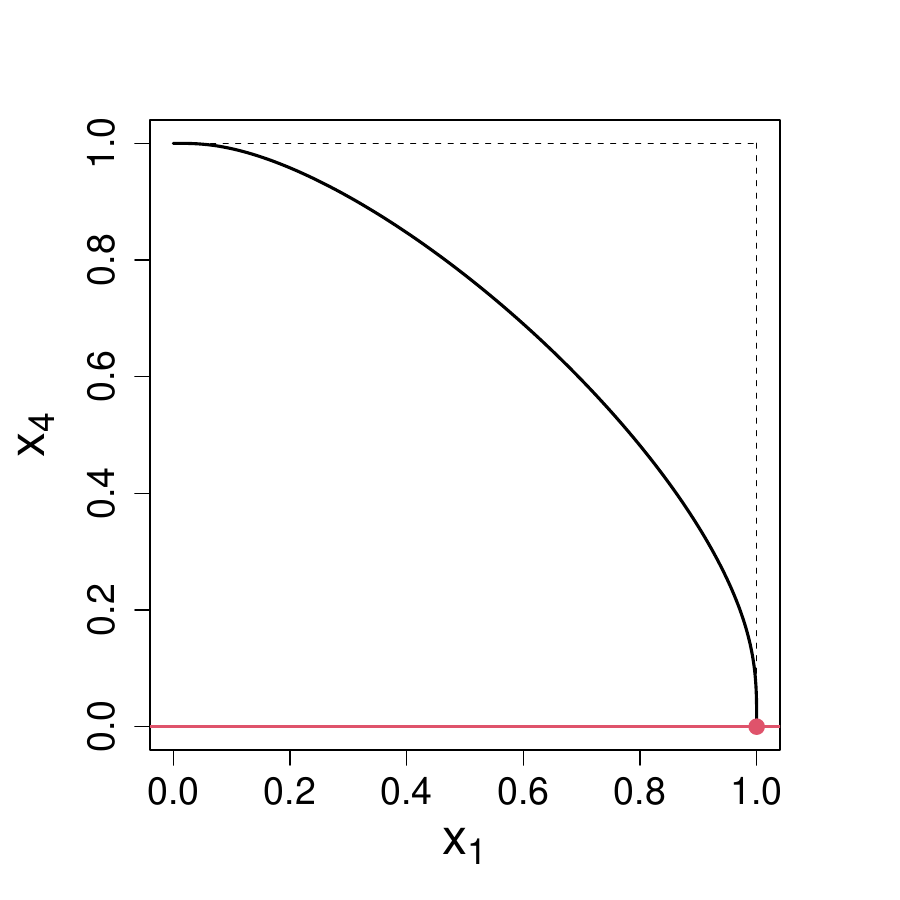}
      \includegraphics[width=0.24\linewidth]{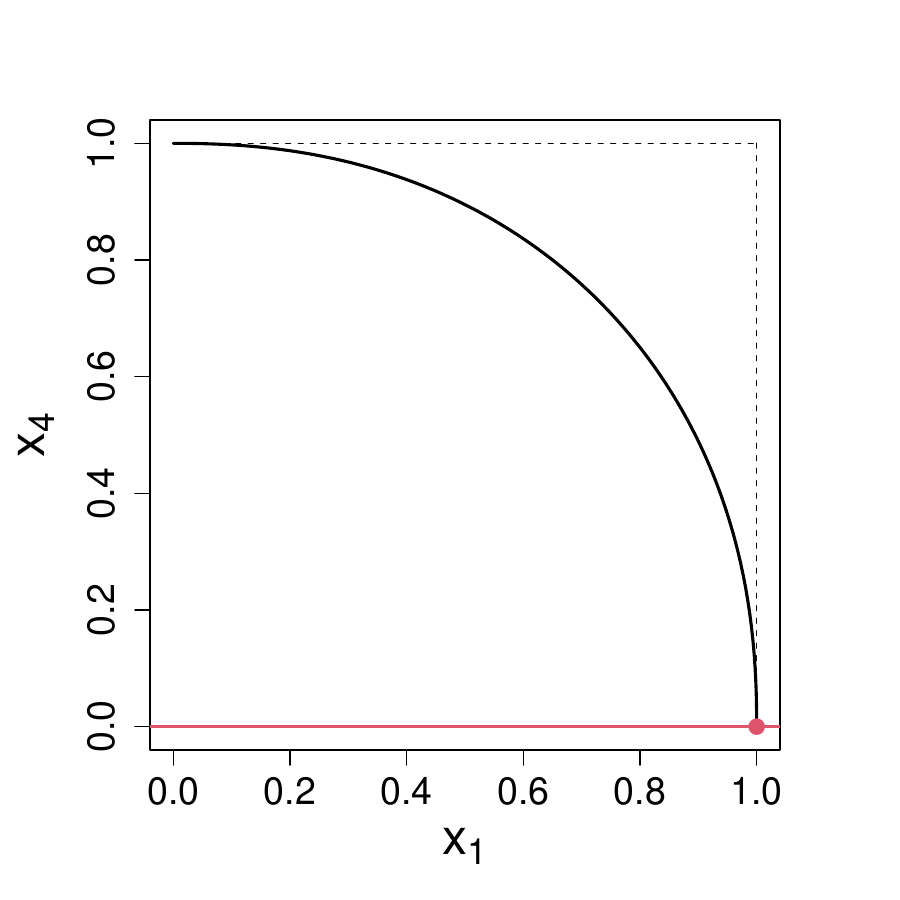}    \includegraphics[width=0.24\linewidth]{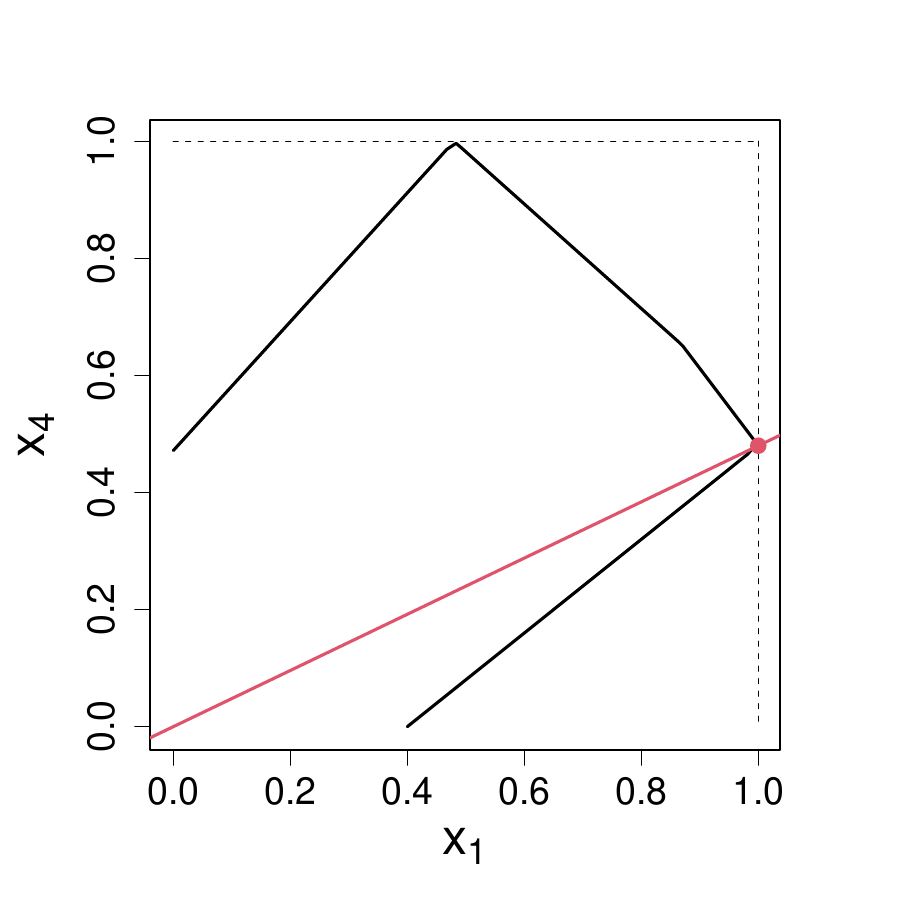}\\
      \includegraphics[width=0.24\linewidth]{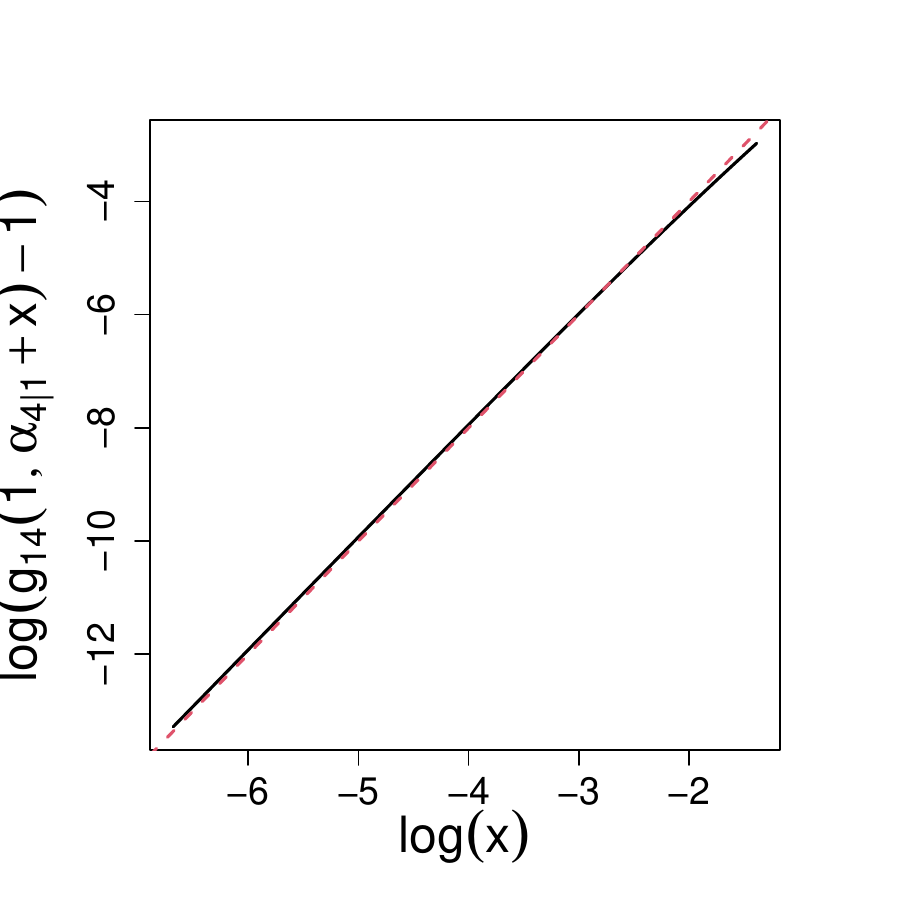}
      \includegraphics[width=0.24\linewidth]{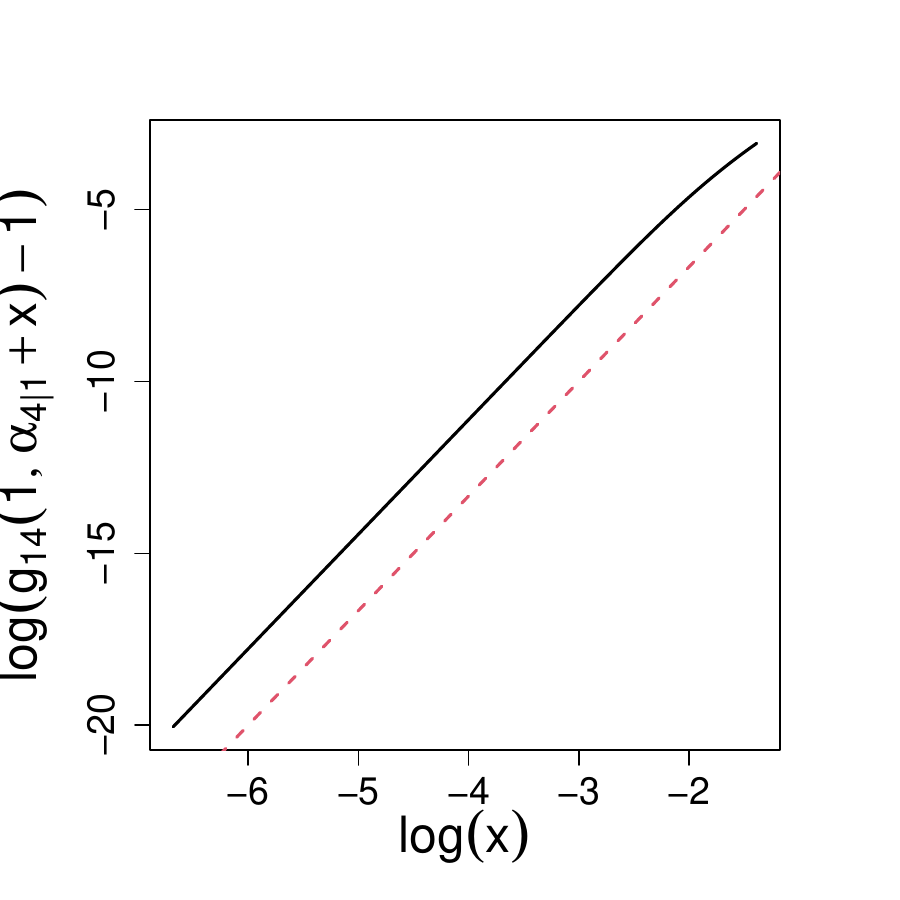}
      \includegraphics[width=0.24\linewidth]{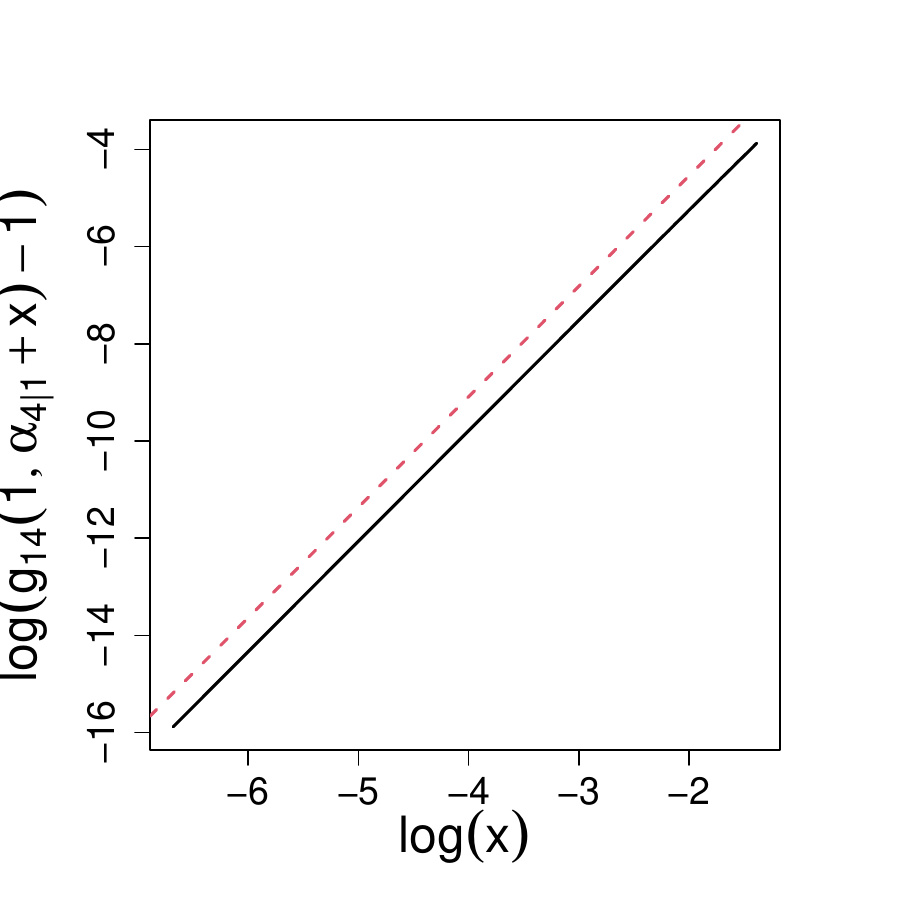}    \includegraphics[width=0.24\linewidth]{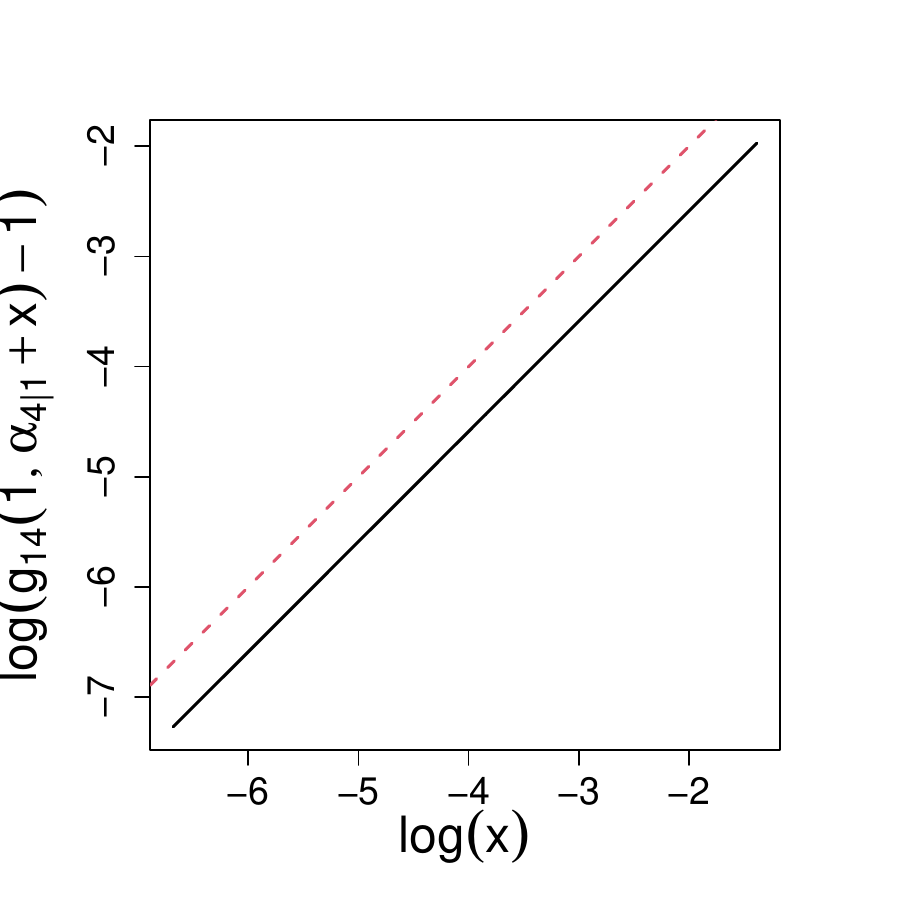}
      \caption{Top row: Illustrations of the unit level set of $g_{\{1,4\}}(x_1,x_4)$. The red dot illustrates the point $(1,\alpha_{4|1})$. Bottom: Illustration of $\log(g_{\{1,4\}}(1,\alpha_{4|1}+x)-1)$ (black solid line) and line with slope $1/(1-\beta_{4|1})$ (red dashed line). The intercept of the red dashed line is always 0; only the slope is compared to the black line, for small $x$. Left--right: Examples (a), (b), (c), (d). Relevant parameters in (b) are: $\theta_{34}=0.3$, $\beta_{4|3}=0.7$; in (c): $\theta_{12}=0.3$, $\theta_{34}=0.2$, $\beta_{4|1}=0.7 \times 0.8 = 0.56$.}
      \label{fig:betaillustrations}
    \end{figure}
    
  \end{example}
  
\subsection{Laplace margins}
\label{sec:deplap}
With geometric extremes, it is natural to work in Laplace margins when there is negative association between at least one pair of variables. Similar results apply as in the exponential case, but we also find relevance in what happens when some variables are small, and therefore require a modest amount of new definitions and notation.

\subsubsection{Conditional extremes assumptions and notation}
\label{sec:cexassnot2}

In this section, we denote the vector $\bm{\alpha}_{\Vset \sm \{i\} \mid i}$ defined in Assumption~\ref{ass:ce_conditions} as $\bm{\alpha}_{\Vset \sm \{i\} \mid i}^+ \in [-1,1]^{d-1}$, and the limit distribution as $K^{+}_{\Vset \sm \{i\}}$. We introduce an analogous Assumption~\ref{ass:ce_conditions_sgn}, and Proposition~\ref{prop:joint_convergence_sgn}, which include conditioning upon each variable to be negatively large. In the following, Assumptions, Propositions and Remarks have an alternative title based on the analogous statement from Section~\ref{sec:depexp}.

\begin{assumption}[Assumption \ref{ass:ce_conditions}$\pm$]
  \label{ass:ce_conditions_sgn}
  Let $\sgn = +$ or $\sgn=-$. For all $i \in \Vset$,
  \begin{equation}
    \frac{\partial^{d-1}}{\partial \bm z} \PP\left(\frac{\bm X_{-i} -
        \bm a^{\sgn}_{\Vset \sm \{i\}\mid i}(\level)}{\bm b^{\sgn}_{\Vset \sm \{i\}\mid i }(\level)}
      \leq \bm z \mid X_{i} = {\sgn} \times \level\right)
    \to \frac{\partial^{d-1}}{\partial \bm z}\,K^{\sgn}_{\Vset \sm \{i\} \mid i}(\bm
    z) =: k^{\sgn}_{\Vset \sm \{i\}\mid i}(\bm z), \quad \text{as $\level\to \infty$.}
    \label{eq:ce_densconvpm}
  \end{equation}

  Additionally,
  \begin{itemize}[wide=0\parindent]
  \item[$(i)$] for $j\in\Vset \sm \{i\}$, the domain of the marginal distribution $K^{\sgn}_{j\mid i}$
    of $K^{\sgn}_{\Vset \sm \{i\} \mid i}$ includes $(0,\infty)$; 
  \item[$(ii)$]
    $\bm \alpha^{\sgn}_{\Vset \sm \{i\} \mid i}:=\lim_{\level \to
      \infty} \bm a^{\sgn}_{\Vset \sm \{i\} \mid i}(t)/t$ exists in
    $[-1,1]^{d-1}$.
  \end{itemize}
\end{assumption}
As in equation~\eqref{eq:ce_dist_conv},  $K^{\sgn}_{\Vset \sm \{i\} \mid i}$ is a distribution
function on $\mathbb{R}^{d-1}$ satisfying
$\lim_{\level\to\infty}K^{\sgn}_{\Vset\sm \{i\} \mid i}(\bm t_j) = 1$
for all $j\in \Vset\sm \{i\}$.

\subsubsection{\texorpdfstring{$\alpha$}{a}-coefficients}
\label{sec:alphacoefflaplace}

\begin{proposition}[Proposition \ref{prop:joint_convergence}$\pm$]
  \label{prop:joint_convergence_sgn}
  Suppose that for $\bm X=(X_i\,:\,i\in \Vset)$, the
  convergence \eqref{eq:gauge_densconv} and Assumption
  \ref{ass:ce_conditions_sgn}
  holds. Let $\sgnmm = +$ or $\sgnmm = -$.
  Then,
  \begin{itemize}[wide=0\parindent]
  \item[$(i)$] for all $i\in \Vset$,
    $g_{\{i\} \cup \{\Vset \sm \{i\}\}}(\sgnmm \times 1, \bm \alpha^{\sgnmm}_{\Vset
      \sm \{i\} \mid i}) = 1$;

  \item[$(ii)$] for all $i\in \Vset$, if there are multiple
    vectors $\bm \alpha^{\sgnmm}$ satisfying
    $g_{\{i\} \cup \{\Vset \sm \{i\}\}}(\sgnmm \times 1, \bm \alpha^{\sgnmm}) = 1$,
    then the coordinate-wise maximum such vector $\bm \alpha^{\star\sgnmm}$
    also satisfies
    $g_{\{i\} \cup \{\Vset \sm \{i\}\}}(\sgnmm \times 1, \bm \alpha^{\star\sgnmm})
    = 1$ and 
    $\bm \alpha^{\sgnmm}_{\Vset \sm \{i\} \mid i} = \bm
    \alpha^{\star\sgnmm}$.
  \end{itemize}
\end{proposition}

\begin{remark}[Remark~\ref{rmk:alphamin}$\pm$]
\label{rmk:alphaminpm}
Since $g(\bm x) \geq \max(|x_i|: i \in \Vset)$,
Proposition~\ref{prop:joint_convergence_sgn} implies that
$\bm \alpha^{\sgn}_{\Vset \sm \{i\}}$ is a global minimizer,
that is,
$g_{\{i\}\cup \{\Vset\sm \{i\}\}}(\sgn \times 1, \bm
\alpha^{\sgn}_{\Vset\sm\{i\}\mid i}) \leq
g_{\{i\}\cup \{\Vset\sm \{i\}\}}(\sgn \times 1, \bm x)$
for all $\bm x\in\RR^{d-1}$.
\end{remark}

The proof of Proposition~\ref{prop:joint_convergence_sgn} is in Appendix~\ref{sec:proofs}. This allows us to introduce Proposition~\ref{prop:block_graph_alphas_lap}, also proven in Appendix~\ref{sec:proofs}, as the Laplace margin analogue of Proposition~\ref{prop:block_graph_alphas}.

\begin{proposition}[Proposition~\ref{prop:block_graph_alphas}$\pm$]
  \label{prop:block_graph_alphas_lap}
  Suppose that $\bm X= (X_i:i \in \Vset)$ follows a block geometric extremal graphical model with graph $\GG=(\Vset, \Eset)$, and that the assumptions in
  Proposition \ref{prop:joint_convergence_sgn} hold.\ For any two
  distinct vertices $i$ and $j$ in $\Vset$, let $\{i=v_0, v_1,\ldots,v_{m_{i j}}=j\}$ be the vertices along the unique shortest path of length $m_{ij} \geq 1$ from $i$ to $j$. Then for $m_{ij}=2$, $\alpha_{j|i}^+ = |\alpha_{v_1|i}^+| \alpha_{j|v_1}^{\sgnmm(\alpha_{v_1|i}^+)}$. For $m_{ij} \geq 3$, 
  \begin{equation}
    \alpha_{j\mid i}^+ = \left|\alpha_{v_1|i}^+\right| \prod_{k=2}^{m_{ij}-1} \left|\alpha_{v_k|v_{k-1}}^{\sgnmm(\alpha^+_{v_{k-1}|i})}\right| \times \alpha_{j|v_{m_{ij}-1}}^{\sgnmm(\alpha_{v_{m_{ij}-1}|i}^+)}.
    \label{eq:product_alphas_lap}
  \end{equation}
\end{proposition}
We note that the relevant signs in equation~\eqref{eq:product_alphas_lap} can be deduced through the recursion
\begin{align*}
    \sgn(\alpha_{v_k|i}^+) = \sgn\left(\alpha_{v_k|v_{k-1}}^{\sgn(\alpha_{v_{k-1}|i}^{+})}\right),
\end{align*}
and that equation~\eqref{eq:product_alphas_lap} reduces to equation~\eqref{eq:product_alphas} in the event that all signs are positive. We illustrate Proposition~\ref{prop:block_graph_alphas_lap} with two examples.
\begin{example}\normalfont 
 Let $\Vset = \{1,2,3,4,5,6\}$ and $\Eset = \{(1,2),(2,3),(3,4),(3,5),(3,6),(4,5),(4,6)\}$. This is a block graph with cliques $\mathcal{C} =\{\{1,2\},\{2,3\},\{3,4,5,6\}\}$ and separators $\mathcal{D} = \{2,3\}$. We consider a Gaussian distribution defined by such a conditional independence graph, with Laplace margins. Such a distribution has limit set with gauge function
 \begin{align*}
 g(\bm{x}) = g_{\{1,2\}}(x_1,x_2) + g_{\{2,3\}}(x_2,x_3) + g_{\{3,4,5,6\}}(x_3,x_4,x_5,x_6) - |x_2| - |x_3|,
 \end{align*}
 where
 $g_{A}(\bm{x}_A) =
 \left(\sgn(\bm{x}_A)|\bm{x}_{A}|^{1/2}\right)^\top(\Sigma_{AA})^{-1}
 \left(\sgn(\bm{x}_A)|\bm{x}_{A}|^{1/2}\right)$, and for
 $A \subset \Vset$, $\Sigma_{AA}$ is the Gaussian correlation matrix
 for the subset of components indexed by $A$. We consider
 $\alpha^+_{6|1}$. The shortest path from node $1$ to $6$ is
 $\{v_0=1, v_1=2, v_2=3, v_3 = 6\}$, with path length $m_{16} = 3$. By
 the result of Proposition~\ref{prop:block_graph_alphas_lap}, we
 therefore have
 $\alpha_{6|1}^+ =
 |\alpha_{2|1}^+|\times|\alpha_{3|2}^{\sgn(\alpha_{2|1}^+)}|\times\alpha_{6|3}^{\sgn(\alpha_{3|1}^+)}$. A
 Gaussian distribution with this conditional independence graph can be
 constructed as follows:
\begin{align*}
    Z_1 &\sim N(0,1)\\
    Z_2|Z_1 &\sim \rho_{12} Z_1 + (1-\rho_{12}^2)^{1/2} \epsilon_2, \qquad &\epsilon_2 \sim N(0,1) \bigCI Z_1\\
    (Z_3,Z_4,Z_5,Z_6)|Z_2  &\sim \rho_{23} Z_2 + (1-\rho_{23}^2)^{1/2} (\epsilon_3,\epsilon_4,\epsilon_5,\epsilon_6), \qquad &(\epsilon_3,\epsilon_4,\epsilon_5,\epsilon_6) \sim N_4(\bm{0},\Sigma_{3:6,3:6}) \bigCI Z_2,
\end{align*}
 where $\Sigma_{3:6,3:6}$ is a $4 \times 4$ correlation matrix with off-diagonal entries $\rho_{kl}$, $k<l$, $k,l \in \{3,4,5,6\}$. We apply the probability integral transformation to get from $\bm{Z}$ to $\bm{X}$ with Laplace margins. For Gaussian distributions with Laplace margins, we have $\alpha_{i|j}^+ = \alpha_{j|i}^+ = \sgn(\rho_{ij})\rho_{ij}^2$, and $\alpha_{i|j}^- = \alpha_{j|i}^- = -\sgn(\rho_{ij})\rho_{ij}^2$ In Figure~\ref{fig:GaussLaplaceExample} we illustrate this with $\rho_{12} = -0.9$, $\rho_{23} = 0.8$, $\rho_{36} = 0.7$, giving $\alpha_{2|1}^+ = -0.81$, $\alpha_{3|2}^{-}=-0.64$, $\alpha_{3|2}^+ =0.64$, $\sgn(\alpha_{3|1}^+) = \sgn(\alpha_{3|2}^{\sgn(\alpha_{2|1}^+)}) = \sgn(\alpha_{3|2}^-) = -$, $\alpha_{6|3}^- = -0.49$, so that
 \begin{align*}
 \alpha_{6|1}^+ &= |\alpha_{2|1}^+|\times|\alpha_{3|2}^{\sgn(\alpha_{2|1}^+)}|\times\alpha_{6|3}^{\sgn(\alpha_{3|1}^+)}\\
 & = 0.81 \times 0.64 \times -0.49 = -0.254.
 \end{align*}

\begin{figure}
    \centering
    \includegraphics[width=0.24\textwidth]{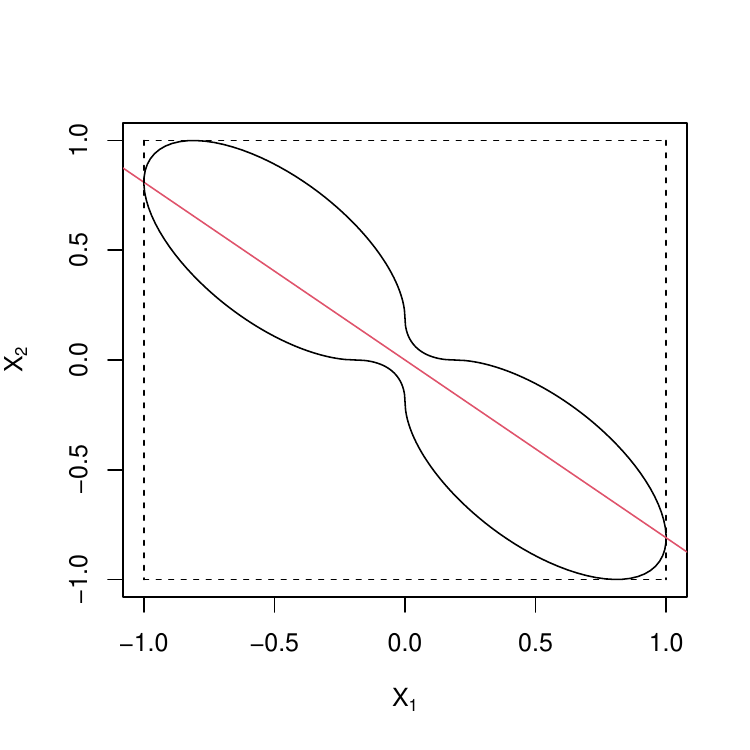}
    \includegraphics[width=0.24\textwidth]{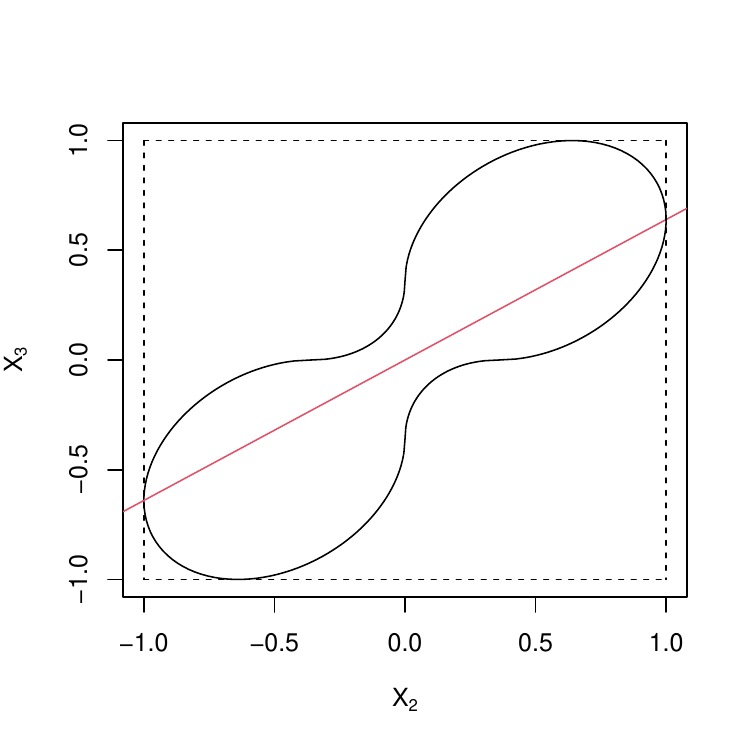}
    \includegraphics[width=0.24\textwidth]{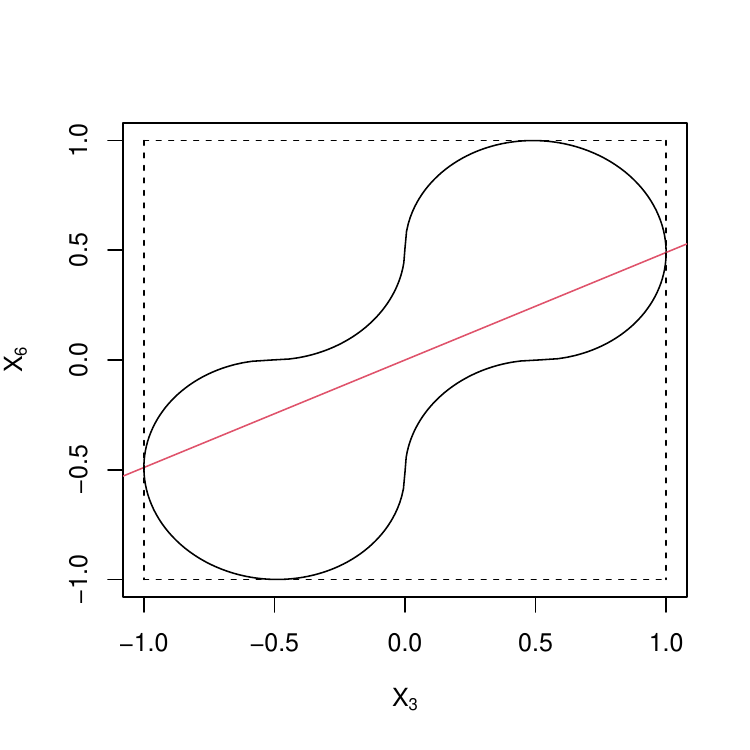}
    \includegraphics[width=0.24\textwidth]{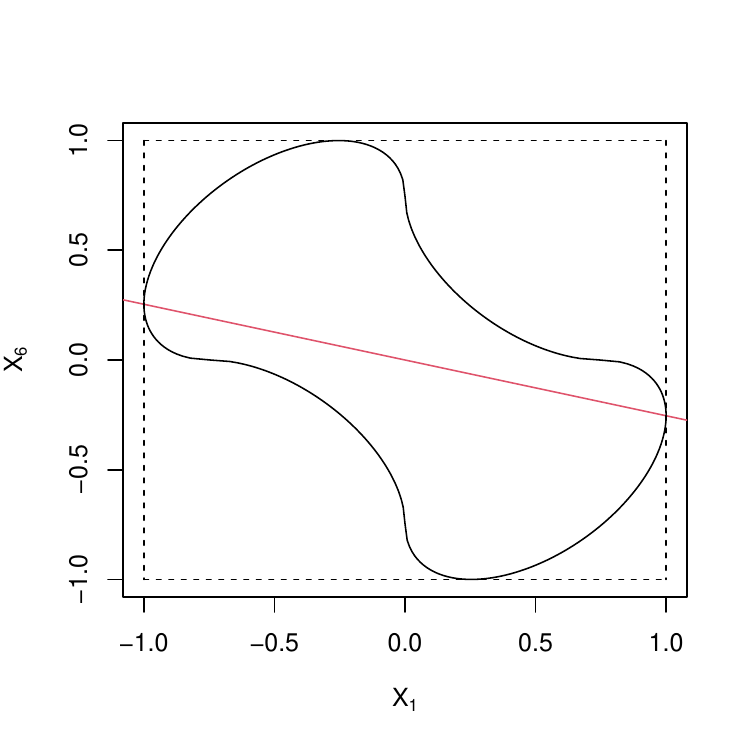}
    \caption{Left to right: marginal 2-dimensional limit sets for the pairs $(X_1,X_2)$, $(X_2,X_3)$, $(X_3,X_6)$, $(X_1,X_6)$. For a pair $(k,l)$, the red lines go through the origin and have slope $\alpha_{k|l}^+$.}
    \label{fig:GaussLaplaceExample}
\end{figure}
In this case of a Gaussian graphical model for $\bm{Z}$, we note this result could also be derived directly. Suppose we marginalize over $Z_2$ to give a block graph with cliques $\mathcal{C} =\{\{1,3\},\{3,4,5,6\}\}$ and separator $\mathcal{D} = \{3\}$. Then $Z_6 \bigCI Z_1 | Z_3$, so the partial correlation $\rho_{16|3}$ is zero, implying $\rho_{16}=\rho_{13}\rho_{36}$. Now if we instead restore $Z_2$ and consider the zero partial correlation between $(Z_1,Z_3)|Z_2$ we similarly get $\rho_{13}=\rho_{12}\rho_{23}$. Putting these together, gives $\rho_{16} = \rho_{36}\rho_{23}\rho_{12} = -0.504$, while $\alpha_{6|1}^+ = \sgn(\rho_{16})\rho_{16}^2=-0.254$. 
\end{example}

\begin{example}\normalfont
Consider a simple chain graphical model with $\Vset=\{1,2,3\}$, $\mathcal{C}=\Eset =\{(1,2),(2,3)\}$, and $\mathcal{D} =\{2\}$. The gauge function is
\begin{align*}
    g(x_1,x_2,x_3) = g_{\{1,2\}}(x_1,x_2) + g_{\{2,3\}}(x_2,x_3) - |x_2|.
\end{align*}
We take $g_{\{1,2\}}$ as a Gaussian gauge with Laplace margins and parameter $\rho_{12}=-0.9$, yielding $\alpha_{2|1}^+ = -0.9^2$. For $g_{\{1,3\}}$, we take a gauge function with more than one solution to $g_{\{2,3\}}(-1,\alpha) = 1$; specifically $g_{\{2,3\}}(-1,-1) = g_{\{2,3\}}(-1,-0.5)=1$, see the second panel of Figure~\ref{fig:LaplaceTree}. By our definition that involves the maximum such solution, we have $\alpha_{3|2}^{-} = -0.5$. The two specified marginal gauge functions, along with the implied marginal $g_{\{1,3\}}$, and joint function $g$, are displayed in Figure~\ref{fig:LaplaceTree}. Combining these and using Proposition~\ref{prop:block_graph_alphas_lap}, we have
\begin{align*}
    \alpha_{3|1}^+ &= |\alpha_{2|1}^{+}| \times \alpha_{3|2}^{\sgn(\alpha_{2|1}^+)} = 0.81 \times -0.5 = -0.405.
\end{align*}
    
\begin{figure}
    \centering
    \includegraphics[width=0.24\textwidth]{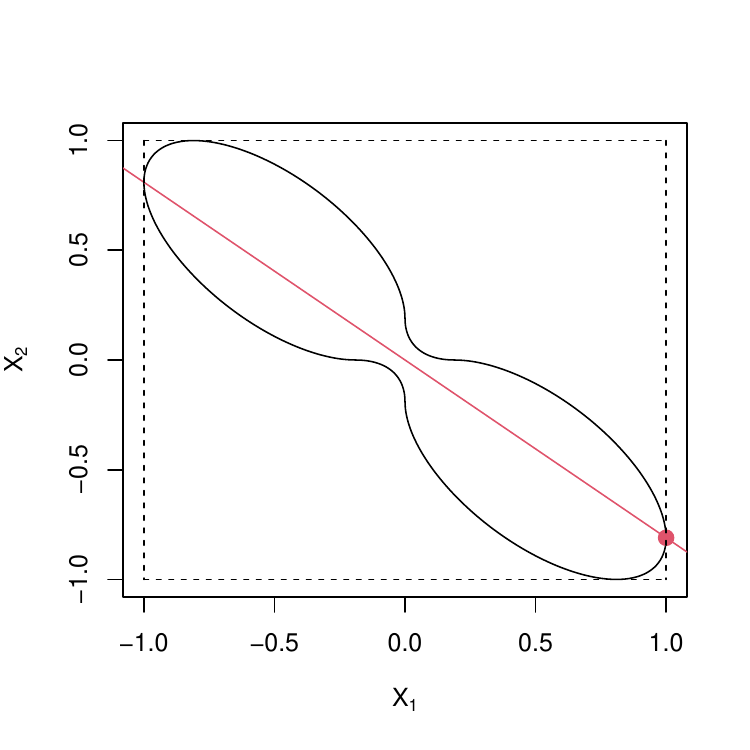}
    \includegraphics[width=0.24\textwidth]{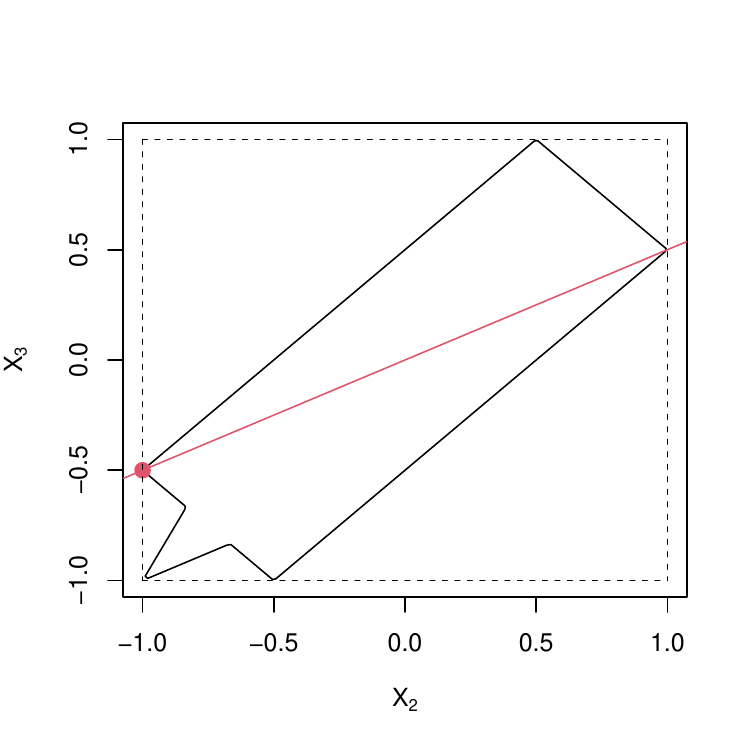}
    \includegraphics[width=0.24\textwidth]{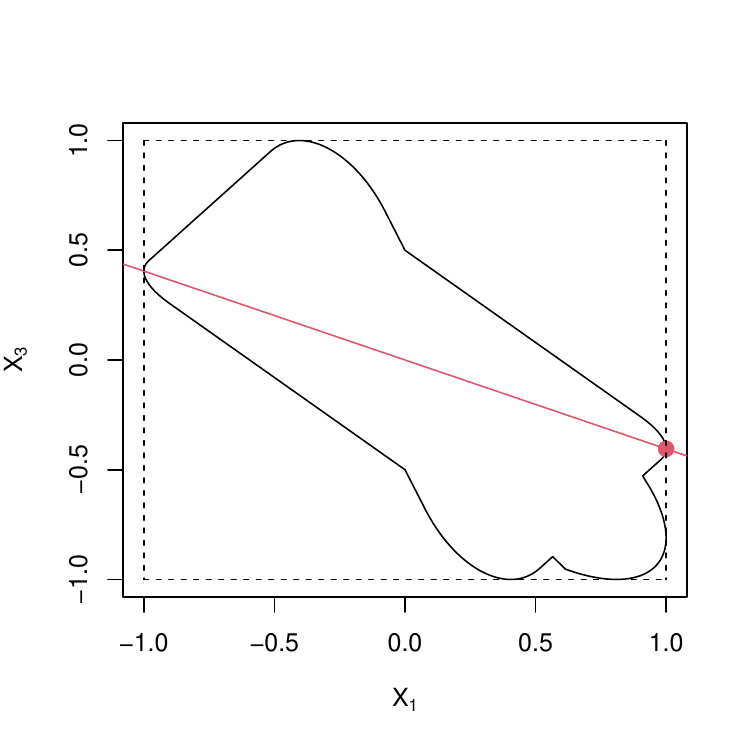}
    \includegraphics[width=0.24\textwidth]{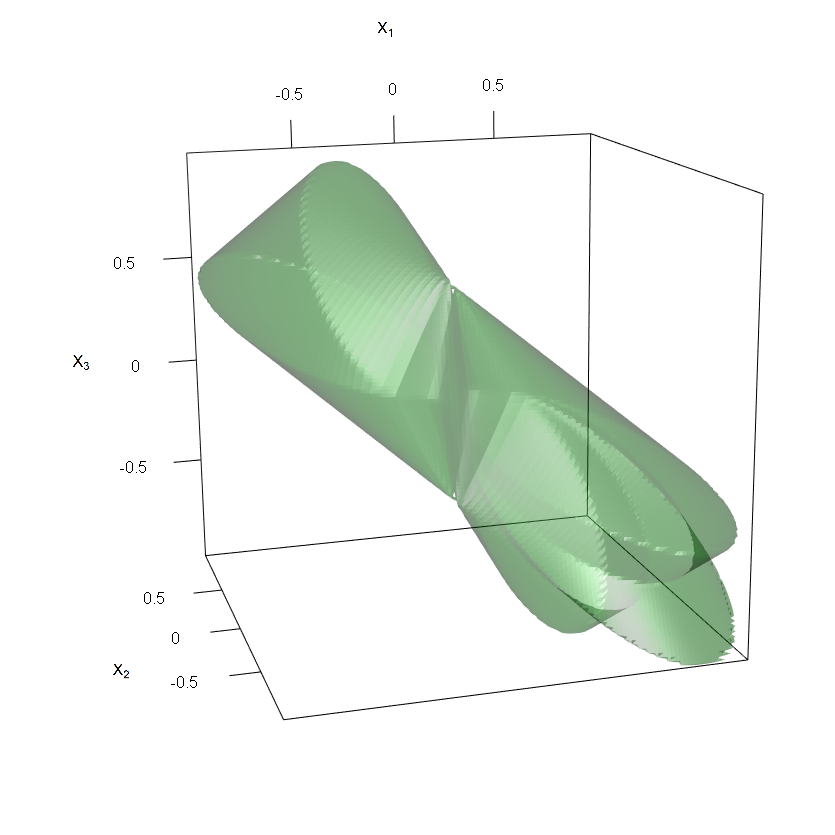}
    \caption{Left to right: marginal 2-dimensional limit sets for the pairs $(X_1,X_2)$, $(X_2,X_3)$, $(X_1,X_3)$; 3-d limit set for $(X_1,X_2,X_3)$. In the first three figures, the highlighted points are $(1,\alpha_{2|1}^+)$, $(-1,\alpha_{3|2}^{-})$, and $(1,\alpha_{3|1}^+)$, respectively, with lines that intersect these points and the origin.}
    \label{fig:LaplaceTree}
\end{figure}
\end{example}

\subsubsection{\texorpdfstring{$\beta$}{b}-coefficients}
\label{sec:betacoefflaplace}

We do not explicitly consider $\beta$-coefficients in the Laplace margin case, but we anticipate that the results follow from blending the arguments of Sections~\ref{sec:beta_coefficients} and~\ref{sec:alphacoefflaplace}. Namely, we need to consider both upper and lower tails, as well as $\alpha$-values in the respective tails.

\section{Results related to joint extremes}
\label{sec:jointex}

When we have a geometric extremal graphical model for which $\alpha_{j|i}=1$ for all $i \in \Vset$ and $j \in \Vset\sm\{i\}$, Propositions~\ref{prop:block_graph_alphas} and~\ref{prop:block_graph_alphas_lap} do not offer particularly interesting results. Such cases are strongly linked to full asymptotic dependence, where the extremal graphical models of \citet{EngelkeHitz20} are relevant. Nonetheless, it is possible to construct geometric extremal graphical models possessing such asymptotic dependence, and consider results pertaining to them.

We firstly consider the notion of joint extremes in the geometric setting, and outline how joint extremes within cliques link to joint extremes across multiple cliques. We then present connections between our definitions of joint extremes and the $\alpha$-coefficients of Section~\ref{sec:Dependencecoeffs}. We finish the section with an important special case of a tree geometric extremal graphical model whose components comprise a common example of a bivariate gauge function for asymptotically dependent variables.

\subsection{Joint extremes in the geometric framework}

If we have full asymptotic dependence, as defined in~equation~\eqref{eq:chipos} of Section~\ref{sec:intro}, and convergence on to a limit set $G=\{\bm{x} \in\mathcal{S}^d:g(\bm{x}) \leq 1\}$, then $g(\bm{1}) = 1$. The converse does not hold true in general. However, when working in the limit set framework, it is simplest to consider joint extremes in $d$ variables as being defined by the property $g(\bm{1}) = 1$, rather than equation~\eqref{eq:chipos}. Based on this observation, and following \citet{CampbellWadsworth24}, we consider a simple criterion in terms of the gauge function for defining groups of joint extremes. 

\begin{definition}\label{def:jtex}
    Take $A \subseteq \Vset$. We say that the variables indexed by $A$ are jointly extreme, while those indexed by $\Vset \setminus A$ are smaller order, if $g(\bm{z}^A) = 1$, where $z_j^A = 1$ for $j \in A$, and $z_{j}^A = \gamma_j$, $j \not\in A$ for $\gamma_j \in [0,1)$ or $\gamma_j \in (-1,1)$ for exponential and Laplace margins, respectively. The full collection of such sets $A$ is denoted $\mathcal{A}$ and termed the \emph{geometric extreme directions}. 
\end{definition}

\begin{remark}
The name \emph{geometric extreme directions} follows the terminology of \emph{extreme directions} from \citet{Mourahibetal24}. The latter is phrased in terms of the support of the so-called spectral measure that arises in classical extreme value theory; see also \citet{Goixetal17}. For a given joint distribution, geometric extreme directions will often coincide with extreme directions, but this need not always be the case.
\end{remark}

The vectors $\{\bm{z}^A: A \in \mathcal{A}\}$ represent points of intersection
between the boundary of the limit set $G$ as described by the unit
level set of $g$, and its bounding box $[0,1]^d$ or $[-1,1]^d$ in
exponential or Laplace margins, respectively. As the coordinatewise
supremum of $G$ is $(1,\ldots,1)$, each index $j \in \{1,\ldots,d\}$ is
represented in at least one $A \in \mathcal{A}$. Note that for a given
$A$, there may be more than one vector $\bm{z}^A$ such that
$g(\bm{z}^A)=1$, and there may be uncountably many such vectors if the
boundary of $G$ coincides with the bounding box on a region of
non-null measure. Furthermore, while the set $\mathcal{A}$ of geometric extreme
directions summarizes the extremal dependence structure, the vectors
$\{\bm{z}^A:A \in \mathcal{A}\}$ provide further detail as the actual directions
experiencing most extremes. Finally, we also note that we exclude
$\gamma_j=-1$ from Definition~\ref{def:jtex} in the Laplace margin
case. The reason for this is that we can replace $-1$ values with $+1$
through negating that variable and switching its lower and upper
tail. This is therefore a version of ``joint extremes'' and excluded
for simplicity.

\begin{example}\normalfont
\label{ex:LogGauss123}
    We illustrate Definition~\ref{def:jtex} with a simple example. Let $\Vset=\{1,2,3\}$ and $\Eset=\{(1,2),(2,3)\}$ represent a chain geometric extremal graphical model, with exponential margins, $g_{\{1,2\}}(x_1,x_2)= (x_1+x_2)/\theta + (1-2/\theta)\min(x_1,x_2)$, $\theta \in (0,1]$ and $g_{\{2,3\}}(x_2,x_3)=(x_2+x_3-2\rho(x_2 x_3)^{1/2})/(1-\rho^2)$, $\rho\in [0,1)$. The limit set defined by $g(x_1,x_2,x_3)=g_{\{1,2\}}(x_1,x_2)+g_{\{2,3\}}(x_2,x_3)-x_2$ is illustrated in Figure~\ref{fig:LogGauss123}. In this example $\mathcal{A}=\{\{1,2\},\{3\}\}$, with $\bm{z}^{\{1,2\}} = (1,1,\rho^2)$ and $\bm{z}^{\{3\}} =(\rho^2,\rho^2,1)$. The interpretation is that joint extremes occur simultaneously in variables $(X_1,X_2)$, and extremes in variable $X_3$ occur only while $(X_1,X_2)$ are of smaller order.
\end{example}
\begin{figure}
    \centering
 \includegraphics[width=0.5\linewidth]{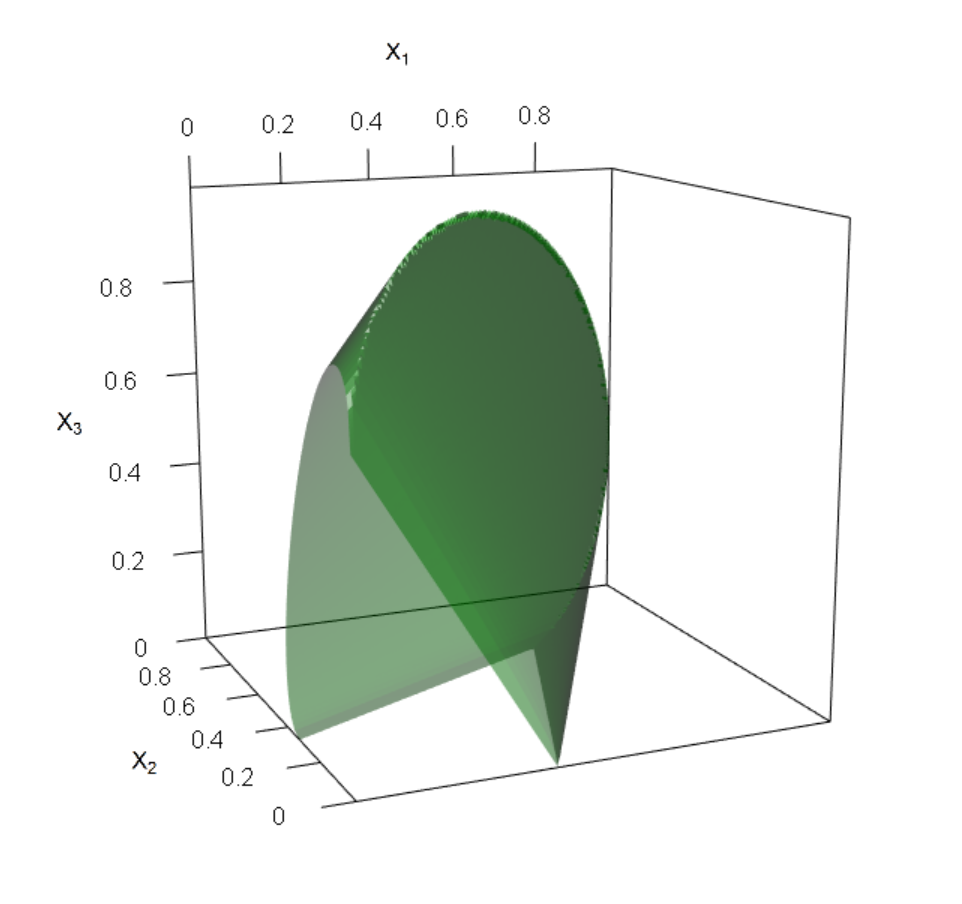}
    \caption{Illustration of the unit level set $g(x_1,x_2,x_3)=1$ for Example~\ref{ex:LogGauss123}, with $\theta=0.4,\rho=0.6$.}
    \label{fig:LogGauss123}
\end{figure}

The following simple proposition highlights that if we have a block geometric extremal graphical model for which all cliques exhibit joint extremes, in the sense $g_{C}(\bm{1}_C) = 1$, then this is equivalent to all variables in $\Vset$ exhibiting joint extremes, in the sense $g(\bm{1}) = 1$. Equivalently, $\Vset \in \mathcal{A}$.

\begin{proposition} \label{prop:blockjtex}
Consider a block geometric extremal graphical model, defined through the gauge function in~\eqref{eq:bgg}. If $g_C(\bm{1}_C) = 1$ for all $C \in \mathcal{C}$, then $g(\bm{1}) = 1$.  Conversely, $g(\bm{1}) = 1$ implies that $g_C(\bm{1}_C) = 1$ for all $C \in \mathcal{C}$.
\end{proposition}

\begin{proof}
For the first direction we have
    \begin{align*}
    g(\bm{1}) = \sum_{C \in \mathcal{C}} g_C(\bm{1}_C) - \sum_{D \in \mathcal{D}} 1
     = |\mathcal{C}|\times 1 - |\mathcal{D}|\times 1
    =1,
\end{align*}
since $|\mathcal{D}| = |\mathcal{C}|-1$. In the other direction,
\begin{align*}
    1=g(\bm{1}) = \sum_{C \in \mathcal{C}} g_C(\bm{1}_C) - \sum_{D \in \mathcal{D}} 1= \sum_{C \in \mathcal{C}} g_C(\bm{1}_C) -(|\mathcal{C}|-1)\times 1,
\end{align*}
implying $\sum_{C \in \mathcal{C}} g_C(\bm{1}_C) = |\mathcal{C}|$. But since $g_C(\bm{x}_C) \geq \max_{j \in C}\{|x_j|\}$, this implies $g_C(\bm{1}_C) = 1$ for all $C \in \mathcal{C}$.
\end{proof}

Our next result in this section pertains to co-extreme behaviour across cliques. In particular, based on Definition~\ref{def:jtex}, we show that groups of variables that cross cliques can only be jointly extreme if the separator variables are included. Proposition~\ref{prop:jtex2clique} outlines this for two cliques and is included for its short and insightful proof. Proposition~\ref{prop:jtexbg} gives the general case.

\begin{proposition} \label{prop:jtex2clique}
    Consider a block graph with two cliques, $C_1,C_2$ and one separator set, $D$, so that $g(\bm{x}) = g_{C_1}(\bm{x}_{C_1})+g_{C_2}(\bm{x}_{C_2}) - |x_{D}|$. Let $B_1 \subseteq C_1 \setminus D$, $B_2 \subseteq C_2 \setminus D$, and define $B= B_1 \cup B_2$ with $\bm{z}^B$ as in Definition~\ref{def:jtex}. Then $g(\bm{z}^B)>1$, i.e., $B$ is not an extreme geometric direction.
\end{proposition}
\begin{proof}
    The gauge function evaluated at $\bm{z}^B$ is
    \begin{align*}
        g(\bm{z}^B) = g_{B_1,C_1\setminus \{B_1 \cup D\},D}(\bm{1}_{B_1},\gamma_{C_1\setminus \{B_1 \cup D\}},\gamma_D) + g_{B_2,C_2\setminus \{B_2 \cup D\},D}(\bm{1}_{B_2},\gamma_{C_1\setminus \{B_1 \cup D\}},\gamma_D) - |\gamma_D|.
    \end{align*} 
    We have $g_{B_1,C_1\setminus \{B_1 \cup D\},D}(\bm{1}_{B_1},\gamma_{C_1\setminus \{B_1 \cup D\}},\gamma_D) \geq 1$, $g_{B_2,C_2\setminus \{B_2 \cup D\},D}(\bm{1}_{B_2},\gamma_{C_1\setminus \{B_1 \cup D\}},\gamma_D) \geq 1$ and $|\gamma_D|<1$, meaning that $g(\bm{z}^B)>1$.
\end{proof}

\begin{proposition}
\label{prop:jtexbg}
    Consider a block graph with clique set $\mathcal{C}=\{C_1,\ldots,C_N\}$ and separator set, $\mathcal{D}=\{D_2,\ldots,D_N\}$. Define $A \subset \Vset$ with $\bm{z}^A$ as in Definition~\ref{def:jtex}. Suppose that at least one element of $A$ lies in $C_1$ and at least one element of $A$ lies in $C_M$ with $1<M \leq N$. Then a minimum requirement for $g(\bm{z}^A)=1$ is that $\cup_{i=2}^M D_i \subset A$, i.e., we cannot have joint extremes across cliques without including the separator variables between those cliques.
\end{proposition}

The proof of Proposition~\ref{prop:jtexbg} is in Appendix~\ref{app:jtex}.

\subsection{Relation between \texorpdfstring{$\alpha$}{a}-coefficients and geometric extreme directions}

There is a strong relationship between $\alpha$-coefficients discussed in Sections~\ref{sec:alpha_coeffs} and~\ref{sec:alphacoefflaplace} and the geometric extreme directions. Recall that the $\alpha$-coefficients satisfy $g_{\{i\}\cup \{\Vset\sm \{i\}\}}(1, \bm
\alpha_{\Vset\sm\{i\}\mid i}) =1$. This implies that there is a set $A\subseteq\{1,\ldots,d\}$ that is given by
\begin{align}
A=\{i\}\cup\{j \in \Vset\sm\{i\}: \alpha_{j|i}=1\}. \label{eq:Aalpha}
\end{align}
The corresponding $\bm{z}^A$ takes value 1 in the $i$th coordinate, and $\bm
\alpha_{\Vset\sm\{i\}\mid i}$ in the other coordinates. However, the $\alpha$-coefficients alone may not give complete information on the geometric extreme directions. This is because there may be multiple such vectors $\widetilde{\bm{\alpha}}_{\Vset\sm\{i\}\mid i}$ satisfying $g_{\{i\}\cup \{\Vset\sm \{i\}\}}(1, \widetilde{\bm{\alpha}}_{\Vset\sm\{i\}\mid i}) =1$, and $\bm
\alpha_{\Vset\sm\{i\}\mid i}$ is the coordinatewise maximum of all of these. For example, suppose that $\bm
\alpha_{\Vset\sm\{i\}\mid i} = (1,\ldots,1)$, so that $A = \{1,\ldots,d\} \in \mathcal{A}$. Then there may exist other geometric extreme directions, but we cannot determine them from the collection of vectors $\bm{\alpha}_{\Vset\sm\{i\}\mid i}, i \in \Vset$. On the other hand, if there is a single vector $\bm
\alpha_{\Vset\sm\{i\}\mid i}$ satisfying $g_{\{i\}\cup \{\Vset\sm \{i\}\}}(1, \bm{\alpha}_{\Vset\sm\{i\}\mid i}) =1$ for each $i \in \Vset$ then these completely determine the geometric extreme directions via relation~\eqref{eq:Aalpha}. In Example~\ref{ex:LogGauss123}, these vectors would be $\bm\alpha_{\{2,3\}|1} = (1,\rho^2)$, $\bm\alpha_{\{1,3\}|2} = (1,\rho^2)$ and $\bm\alpha_{\{1,2\}|3} = (\rho^2,\rho^2)$. The first two of these yield $A=\{1,2\}$, and the third gives $A=\{3\}$, so $\mathcal{A}=\{\{1,2\},\{3\}\}$, as already observed.

\subsection{A tree graphical construction for joint extremes}

In several parametric bivariate asymptotically dependent examples for which there is a limit set, the following form of gauge function arises in exponential margins:
\begin{align}
    g(x_1,x_2) = \frac{x_1}{\theta} + \frac{x_2}{\gamma} + \left(1-\frac{1}{\theta}-\frac{1}{\gamma}\right)\min(x_1,x_2), \qquad \theta,\gamma \in (0,1). \label{eq:ADBVgauge}
\end{align}
We specify to the exponential margin case here since asymptotic dependence is linked to the behaviour of samples in the positive quadrant, and behaviour in other quadrants could be anything. We consider the effect of constructing a tree graphical model with bivariate component gauges of this form. The following lemma provides a building block for the main proposition.

\begin{lemma}\label{lem:bvad}
    Consider a simple chain graph with $\Vset=\{1,2,3\}$ and $\Eset = \{(1,2),(2,3)\}$, so that $g(x_1,x_2,x_3) = g_{\{1,2\}}(x_{1},x_{2}) + g_{\{2,3\}}(x_{1},x_{2}) - x_{2}$. If the bivariate component gauge functions are $g_{\{i,j\}}(x_i,x_j) =\frac{x_1}{\theta_{ij}} + \frac{x_2}{\gamma_{ij}} + \left(1-\frac{1}{\theta_{ij}}-\frac{1}{\gamma_{ij}}\right)\min(x_i,x_j) $, $(i,j) = (1,2), (2,3)$, then the marginal gauge function $g_{\{1,3\}}$ is
    \begin{align*}
          g_{\{1,3\}}(x_{1},x_{3}) = \frac{x_{1}}{\max(\theta_{1 2},\theta_{23})}+ \frac{x_{3}}{\max(\gamma_{12},\gamma_{23})} + \left(1-\frac{1}{\max(\theta_{12},\theta_{23})}-\frac{1}{\max(\gamma_{12},\gamma_{23})}\right)\min(x_{1},x_{3})  .
    \end{align*}
\end{lemma}

 Our main proposition concerning tree extremal graphical models with such gauge functions follows. The proofs of Lemma~\ref{lem:bvad} and Proposition~\ref{prop:treead} are in Appendix~\ref{app:jtex}.

\begin{proposition}
\label{prop:treead}
    Consider a tree extremal graphical model, with gauge function
    \begin{align*}
      g(\bm{x}) = \sum_{(i,j) \in \mathcal{E}} g_{\{i,j\}}(x_i,x_j) - x_i - x_j
+ \sum_{k \in \Vset} x_k,      
    \end{align*}
where all $g_{\{i,j\}}$ have the form in Lemma~\ref{lem:bvad}. For any $k<l \in \Vset$, let $\mbox{pa}(k,l) \subset \Eset$ denote the edges along the shortest path on the graph between $k$ and $l$. Then the marginal gauge function $g_{\{k,l\}}(x_k,x_l)$ equals
    \begin{IEEEeqnarray*}{rCl}
      \frac{x_k}{\max_{(i,j) \in \mbox{pa}(k,l)}(\theta_{ij})}&+& \frac{x_l}{\max_{(i,j) \in \mbox{pa}(k,l)}(\gamma_{ij})} + \\
      &&\left(1-\frac{1}{\max_{(i,j) \in \mbox{pa}(k,l)}(\theta_{ij})}-\frac{1}{\max_{(i,j) \in \mbox{pa}(k,l)}(\gamma_{ij})}\right)\min(x_k,x_l).
    \end{IEEEeqnarray*}
\end{proposition}
We remark that the gauge function in equation~\eqref{eq:ADBVgauge} represents weaker dependence as the coefficients $\theta,\gamma$ increase. Therefore the result of Proposition~\ref{prop:treead} highlights the intuitive result that the strength of dependence between pairs is non-increasing with distance along the graph.

\section{Discussion and future directions}
\label{sec:discussion}

The definition of geometric extremal graphical models opens up the possibility of high dimensional statistical modelling via the frameworks of \citet{WadsworthCampbell24} or \citet{Papastathopoulosetal24}. When adopting a block graph structure, any valid choices of gauge functions on cliques will necessarily lead to a valid higher dimensional joint gauge; in particular, tree gauges permit high dimensional specification solely in terms of bivariate gauge functions, for which a wide variety of choices is available. This solves a key challenge in the geometric framework of specifying sufficiently flexible gauges in higher dimensions.

The theoretical results presented in this work form a useful foundation for understanding the properties of these models, as well as offering opportunities for structure learning, through estimation of various coefficients. We have shown that, in a variety of ways, dependence coefficients decay with distance along the graph: through a product form for $\alpha$-coefficients associated with the conditional extreme value model, and through a mixed product and maximum form for the $\beta$-coefficients of that model. When joint extremes occur, we still observe weakening dependence as we move further along the graph as exemplified in Proposition~\ref{prop:treead}. Future work could generalize such results to general block graphs with different forms of gauges satisfying $g_{C}(\bm{1})=1$ by defining new dependence coefficients for such gauges. We have also shown that joint extremes can occur across cliques only when the separator variables are included, which provides an important interpretation for fitted models.



Our work has centred on decomposable graphs, with a particular focus
on block graphs. The nature of the separator sets as singletons
reduces complexity for analysis, and is also beneficial for
constructing simpler statistical models since one does not have to
consider compatibility criteria between models for
$g_C, C\in \mathcal{C}$ and $g_D, D \in \mathcal{D}$ to ensure a valid
$d-$dimensional gauge function $g$. Nonetheless, we anticipate that
future work will consider the role of more complex graphs.

\section*{Acknowledgements} JW gratefully acknowledges funding from EPSRC grant EP/X010449/1.

\section*{Disclaimer} The authors used the ELM platform for generative
AI; University of Edinburgh, Edina
(\url{https://elm.edina.ac.uk/elm/elm}, model: ChatGPT-5), to help
structure, edit and validate the exposition of Lemma 4. All
mathematical content, derivations, and conclusions were developed and
verified by the authors.


\bibliographystyle{apalike} 
\bibliography{ETGMbib.bib}
\newpage

\appendix
\begin{center}
    \Large \textbf{Supplementary Material}
\end{center}
\section{Proofs associated with \texorpdfstring{$\alpha$}{a}-coefficients}
\label{sec:proofs}

\begin{proof}[Proof of Proposition~\ref{prop:joint_convergence}]
  \begin{enumerate}[leftmargin=0pt, label=(\roman*)]
  \item The proof is a direct extension of Proposition~5~(i) of \citet{NoldeWadsworth22}, included here for completeness. The assumed conditional extremes convergence~\eqref{eq:ce_densconv} on the density scale can be expressed as
  \begin{align}
    e^{t} f_{\{i\} \cup \Vset \sm \{i\}}(t, \bm{a}_{|i}(t) + \bm{b}_{|i}(t)\bm{z}_{|i}) \prod_{j \in  \Vset \sm \{i\}} \left[b_{j|i}(t)\right] \to  k_{\Vset \sm \{i\}\mid i}(\bm z_{|i}) = e^{-h_{\Vset \sm \{i\}\mid i}(\bm z_{|i})}, \label{eq:ce_densconv2}
  \end{align}
  which translates on the log scale to
    \begin{align}
    -\log f_{\{i\} \cup \Vset \sm \{i\}}(t, \bm{a}_{|i}(t) + \bm{b}_{|i}(t)\bm{z}_{|i}) -t -\sum_{j \in  \Vset \sm \{i\}} \log b_{j|i}(t) \to  h_{\Vset \sm \{i\}\mid i}(\bm z_{|i}), \label{eq:ce_log_densconv}
  \end{align}
  where $h_{\Vset \sm \{i\}\mid i}(\bm z_{|i})<\infty$ for all $\bm z_{|i} \in (0,\infty)^{d-1}$, since the support of the limit distribution includes $(0,\infty)$ in each margin. Let $\bm{x}_t = (1, \bm{a}_{|i}(t)/t + \bm{b}_{|i}(t)\bm{z}_{|i}/t) \to \bm{x}$, $t \to \infty$. Using assumption~\eqref{eq:gauge_densconv}, we have
  \begin{align}
      -\log f_{\{i\} \cup \Vset \sm \{i\}} (t, \bm{a}_{|i}(t) + \bm{b}_{|i}(t)\bm{z}_{|i}) = t g_{\{i\} \cup \Vset \sm \{i\}}(\bm{x}_t)[1+o(1)] = t g_{\{i\} \cup \Vset \sm \{i\}}(\bm{x})[1+o(1)]. \label{eq:ce_gauge_conv}
  \end{align}
Combining~\eqref{eq:ce_log_densconv} and~\eqref{eq:ce_gauge_conv} gives
\begin{align}
   g_{\{i\} \cup \Vset \sm \{i\}}(\bm{x}_t)[1+o(1)] = 1 +  h_{\Vset \sm \{i\}\mid i}(\bm z_{|i})/t + \sum_{j \in  \Vset \sm \{i\}} \log b_{j|i}(t) / t +o(1/t). \label{eq:ce_densgauge}
\end{align}
Suppose that for the first $j \in \Vset \sm \{i\}$, $b_{j|i}(t)/t \to \gamma_j>0$. Then $\bm{x}_t \to \bm{x} = (1, \alpha_{j|i} + \gamma_j z_{j|i}, \bm{\alpha}_{\setminus j | i})$, while taking $t \to \infty$ in~\eqref{eq:ce_densgauge} yields $g_{\{i\} \cup \Vset \sm \{i\}}(1, \alpha_{j|i} + \gamma_j z_{j|i}, \bm{\alpha}_{\setminus j | i}) = 1$ for any $z_{j|i}$. But since $g(\bm{x}) \geq \|\bm{x}\|_{\infty}$, this implies $z_{j|i} \leq (1-\alpha_{j|i})/\gamma_j$. No such upper bound applies, so we conclude $\gamma_j=0$, i.e., $b_{j|i}(t) = o(t)$. The same argument can be repeated for any index $j$. Therefore, taking limits in~\eqref{eq:ce_densgauge} yields $g_{\{i\} \cup \Vset \sm \{i\}}(1, \bm{\alpha}_{|i}) = 1$.
\item Suppose that there are $m \geq 2$ vectors $\bm{\alpha}^1 \neq \cdots \neq \bm{\alpha}^m$ satisfying $ g_{\{i\} \cup \{\Vset \sm \{i\}\}}(1, \bm \alpha^1) = \cdots = g_{\{i\} \cup \{\Vset \sm \{i\}\}}(1, \bm \alpha^m)  = 1$. Consider the marginal gauge function 
    \begin{align*}
        1 =g_{\{i,j\}}(1, \alpha_{j|i}) &= g_{\{i,j\}}(1, \alpha_{j|i}^1) \\
        & = g_{\{i\} \cup \{\Vset \sm \{i\}\}}(1, \bm \alpha^1)\\ 
        & = \ldots \\
        &= g_{\{i,j\}}(1, \alpha_{j|i}^m) \\
        &= g_{\{i\} \cup \{\Vset \sm \{i\}\}}(1, \bm \alpha^m)
    \end{align*}
    Using the convergence to types argument as in Proposition~5~(iii) of \citet{NoldeWadsworth22}, we have $\alpha_{j|i} = \max(\alpha_{j|i}^1,\ldots,\alpha_{j|i}^m)$. This argument holds true for any index $j \in \Vset \sm \{i\}$, so that $\bm{\alpha}_{\Vset\sm\{i\}|i} = \max(\bm{\alpha}^1,\ldots,\bm{\alpha}^m)$, where the maximum operation is applied componentwise. Note that this implies $\bm \alpha_{\Vset\sm\{i\}|i} = \bm{\alpha}^k$ for some $k\in \{1,\ldots, m\}$.
  \end{enumerate}
  \end{proof}

\begin{proof}[Proof of Proposition~\ref{prop:joint_convergence_sgn}]

The proof is very similar to that of Proposition~\ref{prop:joint_convergence}, beginning by replacing convergence~\eqref{eq:ce_densconv2} by the equivalent of~\eqref{eq:ce_densconvpm}, namely
  \begin{align}
    2e^{t} f_{\{i\} \cup \Vset \sm \{i\}}(\sgn\times t, \bm{a}^{\sgn}_{|i}(t) + \bm{b}^{\sgn}_{|i}(t)\bm{z}_{|i}) \prod_{j \in  \Vset \sm \{i\}} \left[b^{\sgn}_{j|i}(t)\right] \to  k^{\sgn}_{\Vset \sm \{i\}\mid i}(\bm z_{|i}) = e^{-h^{\sgn}_{\Vset \sm \{i\}\mid i}(\bm z_{|i})}. \label{eq:ce_densconv_lap}
  \end{align}
Following from equation~\eqref{eq:ce_gauge_conv}, we proceed in the same manner as before but with $\bm{x}_t = (\sgn \times 1,\bm{a}^{\sgn}_{|i}(t)/t + \bm{b}^{\sgn}_{|i}(t)\bm{z}_{|i}/t)) \to \bm{x} = (\sgn \times 1,\bm{\alpha}^{\sgn}_{|i})$, leading to the conclusion that $g_{\{i\} \cup \Vset \sm \{i\}}(\sgn \times 1, \bm{\alpha}^{\sgn}_{|i}) = 1$. The proof of part~(ii) is also entirely analogous to part~(ii) of Proposition~\ref{prop:joint_convergence}.
    
\end{proof}


\begin{proof}[Proof of Proposition~\ref{prop:block_graph_alphas}]

We begin by noting that as a block graph is a decomposable graph, the set of cliques can be ordered as
$\calC = \{{C}_1, \dots, C_N\}$ such that, for all $i=2, \dots, N$,
\begin{equation} D_i := C_i \cap ( \cup_{j=1}^{i-1} C_j ) \subset C_k
  \quad \text{for some } \, k < i;
  \label{eq:ri}
\end{equation}
see \citet[][Appendix A]{EngelkeHitz20}, or \citet[][Chapter 2]{Lauritzen96}.
This condition is called the running intersection property and the
elements of $\calD = (D_2,\dots,D_N)$ are termed the separators of the
graph.\ We will always assume that the set of cliques and the set of separators have been ordered so that \eqref{eq:ri} holds true.

For any $i$ and $j$, the set of cliques can be ordered so that $i\subset C_1$ and
$j\subset C_{m_{i j}}$ where $1\leq m_{i j}\leq N$
denotes the length of the shortest path between $i$ and $j$ in $\GG$.\ This shortest path can be expressed as $\{i, D_2, D_3, \ldots, D_{m_{ij}-1}, D_{m_{ij}},j\}$, where $D_k \subset C_{k}, C_{k-1}$ for $k=2,\ldots, m_{ij}$.\\
\\
We require
$\alpha_{j\mid i}=\max\{\widetilde{\alpha}_{j\mid
  i}\in(0,1]\,:\,g_{\{i,j\}}(1,\widetilde{\alpha}_{j\mid i})=1\}$.\
Using the result of Proposition~\ref{prop:joint_convergence}, that the
vector $\bm \alpha_{\Vset \sm \{i\} \mid i}$ is a global
minimizer,
  \begin{align*}
  1=  g_{\{i,j\}}(1,\alpha_{j\mid i})&= \min_{x_{s}\geq
      0\,:\, s\notin \{i,j\}} g_{i, \Vset \sm \{i,j\}, j}(1, \bm
    x_{\Vset \sm \{i, j\}},
    \alpha_{j|i})\\
    & = g_{\{i\} \cup \{\Vset \sm \{i\}\}}(1, \bm \alpha_{\Vset
      \sm \{i\} \mid i}) \\
    &={g_{\{i\}\cup
          \{C_1\sm \{i\}\}}(1,\bm{\alpha}_{C_1\sm \{i\}|i})} + \sum_{k=2}^{N}
      {g_{D_k\cup\{C_k \sm D_k\}}(\alpha_{D_k|i}, \bm \alpha_{C_k\sm D_k|i}) -
        \alpha_{D_k|i}}.
  \end{align*}
Since we require
  $g_{\{i,j\}}(1, \alpha_{j\mid i})=1$, we must have that
  \begin{align}
    g_{\{i\}\cup \{C_1\sm \{i\}\}}(1,\bm{\alpha}_{C_1\sm \{i\}|i}) = 1 \label{eq:giC1}
    \end{align}
    and
    \begin{align}   {g_{D_k\cup\{C_k \sm D_k\}}(\alpha_{D_k|i}, \bm \alpha_{C_k\sm D_k|i}) -
        \alpha_{D_k|i}} = 0, \qquad k=2,\ldots,N. \label{eq:gDkCkfull}\end{align} 
        Suppose firstly that $\alpha_{D_k|i}>0$ for all $k=2,\ldots,m_{ij}$. Factorising $\alpha_{D_k|i}$ out of equations~\eqref{eq:gDkCkfull} gives
              \begin{align} g_{D_k\cup\{C_k \sm D_k\}}(1, \bm \alpha_{C_k\sm
      D_k|i}/\alpha_{D_k|i}) = 1, \qquad k=2,\ldots,m_{ij}.\label{eq:gDkCk}
  \end{align}
  Recognizing that this defines the vector $\bm{\alpha}_{C_k\sm
      D_k|D_k}$, we have
      \begin{align}
         \bm{\alpha}_{C_k\sm
      D_k|D_k} & = \bm \alpha_{C_k\sm
      D_k|i}/\alpha_{D_k|i},\qquad k=2,\ldots,m_{ij}.\label{eq:alpharelation}
      \end{align}
      Now $j \subset C_{m_{ij}}$, so~\eqref{eq:alpharelation} yields
      $\alpha_{j|i} = \alpha_{D_{m_{ij}}|i} \alpha_{j|D_{m_{ij}}}$. If $m_{ij}=2$, then $\alpha_{j|i}=\alpha_{D_2|i}\alpha_{j|D_2}$. Otherwise, for $k\geq 3$, the system
      of equations~\eqref{eq:alpharelation} provides a recurrence
      relation for $\alpha_{D_{m_{ij}}|i}$. Specifically, since
      $D_k \subset C_k, C_{k-1}$,
      \begin{align}
          \alpha_{D_k|i} = \alpha_{D_{k-1}|i} \alpha_{D_k|D_{k-1}}, \qquad k=3,\ldots, m_{ij}, \label{eq:alpharecursion}
      \end{align}
      and so
  \begin{align}
    \alpha_{j|i} =  \alpha_{D_2|i} \left[\prod_{k=3}^{m_{ij}} \alpha_{D_{k}|D_{k-1}}\right] \alpha_{j|D_{m_{ij}}}. \label{eq:alphaprod}
\end{align}
Now consider the case where $\alpha_{D_k|i}=0$ for at least one $k=2,\ldots,m_{ij}$. If $\alpha_{D_k|i} = 0$ and
\begin{align}
    {g_{D_k\cup\{C_k \sm D_k\}}(\alpha_{D_k|i}, \bm \alpha_{C_k\sm D_k|i}) -
        \alpha_{D_k|i}} = 0, \label{eq:gCkDkalpha0}
\end{align}
then $g_{D_k\cup\{C_k \sm D_k\}}(0, \bm \alpha_{C_k\sm D_k|i}) = 0$,
which implies $\bm \alpha_{C_k\sm D_k|i} = \bm{0}$, since
$g_{D_k\cup\{C_k \sm D_k\}}(\bm{x}) \geq \|\bm{x}\|_{\infty}$. Since
$D_{k+1} \subset C_k$, we therefore have $\alpha_{D_{k+1}|i} = 0$
also. By iteration therefore,
$\alpha_{D_k|i} = \alpha_{D_{k+1}|i} = \cdots = \alpha_{D_{m_{ij}}|i} =
0$. Using equation~\eqref{eq:gCkDkalpha0} with $k=m_{ij}$ and
$\alpha_{D_{m_{ij}}|i}= 0$ we get $\alpha_{j|i}=0$.

Suppose that $k^\star = \min_{k \in 2,\ldots,m_{ij} }\{k: \alpha_{D_{k}|i}=0\}$. By taking $j=D_{k^\star-1}$, we can see that $\alpha_{D_{k^\star-1}|i}$ can be expressed as in the product~\eqref{eq:alphaprod}, and equation~\eqref{eq:alpharecursion} gives $\alpha_{D_k^\star|i} = \alpha_{D_k^\star|D_{k^\star}-1}\alpha_{D_{k^\star-1}|i}$, which implies $\alpha_{D_{k^\star}|D_{k^\star-1}} = 0$. Therefore $\alpha_{j|i}=0$ can be expressed through product~\eqref{eq:alphaprod} also. 
\end{proof}

\begin{proof}[Proof of Proposition~\ref{prop:block_graph_alphas_lap}]
  Again, we follow a similar argument to Proposition~\ref{prop:block_graph_alphas}, with $\alpha^+_{j\mid i}=\max\{\widetilde{\alpha}_{j\mid
    i}\in[-1,1]\,:\,g_{\{i,j\}}(1,\widetilde{\alpha}_{j\mid i})=1\}$, and
  \begin{align*}
  1=  g_{\{i,j\}}(1,\alpha^+_{j\mid i})&= \min_{x_{s} \in \mathbb{R}^{d-2}\,:\, s\notin \{i,j\}} g_{i, \Vset \sm \{i,j\}, j}(1, \bm
    x_{\Vset \sm \{i, j\}},
    \alpha^+_{j|i})\\
    & = g_{\{i\} \cup \{\Vset \sm \{i\}\}}(1, \bm \alpha^+_{\Vset
      \sm \{i\} \mid i}) \\
    &={g_{\{i\}\cup
          \{C_1\sm \{i\}\}}(1,\bm{\alpha}^+_{C_1\sm \{i\}|i})} + \sum_{k=2}^{N}
      {g_{D_k\cup\{C_k \sm D_k\}}(\alpha^+_{D_k|i}, \bm \alpha^+_{C_k\sm D_k|i}) -
        |\alpha^+_{D_k|i}|},
  \end{align*}
where  $g_{\{i\}\cup
          \{C_1\sm \{i\}\}}(1,\bm{\alpha}^+_{C_1\sm \{i\}|i}) = 1$ and $\sum_{k=2}^{N}
      g_{D_k\cup\{C_k \sm D_k\}}(\alpha^+_{D_k|i}, \bm \alpha^+_{C_k\sm D_k|i}) -
        |\alpha^+_{D_k|i}| = 0$.

As before, begin with the case $|\alpha^+_{D_k|i}|>0$, $k=2,\ldots,m_{ij}$, so that
\begin{align*}
|\alpha^+_{D_k|i}| \left[g_{D_k\cup\{C_k \sm D_k\}}\left(\sgn(\alpha^+_{D_k|i}) \times 1, \bm \alpha^+_{C_k\sm D_k|i}/|\alpha^+_{D_k|i}|\right) -
        1\right] = 0, \qquad k=2,\ldots,m_{ij},
\end{align*}
giving 
\begin{align}
\bm{\alpha}^{\sgn(\alpha^+_{D_k|i})}_{C_k\sm D_k|D_k} = \bm{\alpha}^+_{C_k\sm D_k|i}/|\alpha^+_{D_k|i}|. \label{eq:alphareclaplace}
\end{align}
For $m_{ij}=2$, this yields $\alpha_{j|i}^+ = |\alpha_{D_2|i}^+| \alpha_{j|D_2}^{\sgn(\alpha_{D_2|i}^+)}$. For $m_{ij} \geq 3$, equation~\eqref{eq:alphareclaplace} leads to the relation $\alpha^+_{D_k|i} = \alpha^{\sgn(\alpha^+_{D_{k-1}|i})}_{D_k|D_{k-1}}|\alpha^+_{D_{k-1}|i}|$, $k=3,\ldots, m_{ij}$. Overall we thus have
\begin{align}
    \alpha_{j\mid i}^+ = \left|\alpha_{D_2|i}^+\right| \prod_{k=3}^{m_{ij}} \left|\alpha_{D_k|D_{k-1}}^{\sgn(\alpha^+_{D_{k-1}|i})}\right| \times \alpha_{j|D_{m_{ij}}}^{\sgn(\alpha_{D_{m_{ij}}|i}^+)}. \label{eq:alphaprodlap}
\end{align}
When $|\alpha^+_{D_k|i}|=0$ for some $k$, an analogous argument to that in Proposition~\ref{prop:block_graph_alphas} leads to the conclusion that $\alpha^+_{j|i} = 0$. Similarly we can show that if $\alpha^+_{D_{k^\star}|i}=0$ and $|\alpha^+_{D_{k^\star-1}|i}| \neq 0$, then $\alpha^{\sgn(\alpha^+_{D_{k-1}|i})}_{D_k|D_{k-1}} = 0$, so that the result $\alpha_{j|i}^+=0$ can be expressed in the form of product~\eqref{eq:alphaprodlap}.

\end{proof}


\section{{Proofs associated with \texorpdfstring{$\beta$}{b}-coefficients}}
\label{app:proofbetas}
\subsection{Marginalization properties of gauge functions}
\label{sec:gauge_marginalization}
\begin{lemma}
  \label{lem:blockgraphtopath}
  Let $\GG = (\Vset,\Eset)$ be a block geometric extremal graphical model. For any two indices $(i,j) \in \Vset$, the marginal gauge function $g_{\{i,j\}}(x_i,x_j)$ may be obtained solely through consideration of the chain graph with nodes $\Vset'=\{i, D_2, D_3, \ldots, D_{m_{ij}-1}, D_{m_{ij}},j\}$ lying on the unique shortest path between $i$ and $j$.
\end{lemma}

\begin{proof}[Proof of Lemma~\ref{lem:blockgraphtopath}]
  As in the proof of Proposition~\ref{prop:block_graph_alphas},
  we order the maximal cliques
  $C_1,C_2,\dots,C_N$ so that $i\subset C_1$ and
  $j\subset C_{m_{i j}}$ where $1\leq m_{i j}\leq N$ denotes the length of the shortest path between $i$ and $j$ in $\GG$.\ Recall that this shortest path can be expressed as $\{i, D_2, D_3, \ldots, D_{m_{ij}-1}, D_{m_{ij}},j\}=:\Vset'$, where $D_k \subset C_{k}, C_{k-1}$ for $k=2,\ldots, m_{ij}$. We have
  \begin{align*}
    g_{\{i,j\}}(x_i,x_j) &= \min_{x_s \geq 0, s \neq i,j} g_{C_1}(\bm x_{C_1}) + \sum_{k=2}^{m_{ij}} [g_{C_k}(\bm{x}_{C_k})- x_{D_k}] +\sum_{k=m_{ij}+1}^N [g_{C_k}(\bm{x}_{C_k})- x_{D_k}]\\
                         &= \min_{x_s:s=D_2,\ldots,D_{m_{ij}}} \min_{x_s: s\in \Vset \sm \Vset'} g_{C_1}(\bm{x}_{C_1}) + \sum_{k=2}^{m_{ij}} [g_{C_k}(\bm{x}_{C_k})- x_{D_k}] +\sum_{k=m_{ij}+1}^N [g_{C_k}(\bm{x}_{C_k})- x_{D_k}]
  \end{align*}
  Partition the index set $\Vset\sm\Vset'$ into two disjoint sets. Let $(\Vset\sm\Vset')_1$ represent the indices lying only in a single clique, and $(\Vset\sm\Vset')_2$ represent indices in two cliques (i.e., the separators $D_{m_{ij}+1}, \ldots, D_N$, or $D_{m_{ij}+2}, \ldots, D_N$, if $j=D_{m_{ij}+1}$). If $j=D_{m_{ij}+1}$ is a separator, then
  
  
    \begin{IEEEeqnarray}{rCl}
    g_{\{i,j\}}(x_i,x_j) &=& \min_{x_s:s=D_2,\ldots,D_{m_{ij}}} \min_{x_s: s\in (\Vset \sm \Vset')_2} \min_{x_s: s\in (\Vset \sm \Vset')_1}  \left\{g_{C_1}(\bm{x}_{C_1}) + \sum_{k=2}^{m_{ij}} [g_{C_k}(\bm{x}_{C_k})- x_{D_k}] \right. \notag \\ & & \hspace{8cm} \left. +\sum_{k=m_{ij}+1}^N [g_{C_k}(\bm{x}_{C_k})- x_{D_k}]\right\} \notag \\
                        & = &\min_{x_s:s=D_2,\ldots,D_{m_{ij}}} \min_{x_s: s\in (\Vset \sm \Vset')_2} \left\{g_{\{i,D_2\}}(x_i,x_{D_2}) + \sum_{k=2}^{m_{ij}} [g_{\{D_{k},D_{k+1}\}}(x_{D_k},x_{D_{k+1}})- x_{D_k}] \right. \label{eq:lem2l2}\\ & & \hspace{6.2cm} \left. +\sum_{k=m_{ij}+1}^{N-1} [g_{\{D_{k},D_{k+1}\}}(x_{D_k},x_{D_{k+1}})- x_{D_k}]\right\} \notag \\
                         &= &\min_{x_s:s=D_2,\ldots,D_{m_{ij}}} g_{\{i,D_2\}}(x_i,x_{D_2}) + \sum_{k=2}^{m_{ij}} [g_{\{D_{k},D_{k+1}\}}(x_{D_k},x_{D_{k+1}})- x_{D_k}]. \notag
    \end{IEEEeqnarray}
If $j$ is not a separator, the same ideas hold, but the notation is more involved. In particular, the middle line~\eqref{eq:lem2l2} becomes
\begin{IEEEeqnarray*}{r}
\min_{x_s:s=D_2,\ldots,D_{m_{ij}}} \min_{x_s: s\in (\Vset \sm \Vset')_2} \left\{g_{\{i,D_2\}}(x_i,x_{D_2}) + \sum_{k=2}^{m_{ij}-1} [g_{\{D_{k},D_{k+1}\}}(x_{D_k},x_{D_{k+1}})- x_{D_k}] \right. \\\left. + g_{\{D_{m_{ij}},j,D_{m_{ij}+1}\}}(x_{D_{m_{ij}}},x_j,x_{D_{m_{ij}+1}}) - x_{D_{m_{ij}}}+\sum_{k=m_{ij}+1}^{N-1} [g_{\{D_{k},D_{k+1}\}}(x_{D_k},x_{D_{k+1}})- x_{D_k}]\right\},
\end{IEEEeqnarray*}
which upon taking the inner minimum gives
\begin{align*}
    \min_{x_s:s=D_2,\ldots,D_{m_{ij}}} g_{{i,D_2}}(x_i,x_{D_2}) + \sum_{k=2}^{m_{ij}-1} [g_{\{D_{k},D_{k+1}\}}(x_{D_k},x_{D_{k+1}})- x_{D_k}] + g_{\{D_{m_{ij}},j\}}(x_{D_{m_{ij}}},x_j)- x_{D_{m_{ij}}}.
\end{align*}
This completes the proof.
\end{proof}


\begin{lemma}
  \label{lem:graph_marginalization}
  Let $\Vset=\{i_1,\ldots,i_m\}$ and $\Eset =\{(i_1,i_2),(i_2,i_3),\ldots,(i_{m-1},i_m)\}$ so that $\GG = (\Vset,\Eset)$ is a chain geometric extremal graphical model. Then marginalizing over any index $i_k \in \Vset$ also leads to a chain geometric extremal graphical model $\GG_{-k} = (\Vset_{-k},\Eset_{-k})$ with $\Vset_{-k} = \{i_1,\ldots,i_{k-1},i_{k+1},\ldots,i_m\}$ and $\Eset_{-k}=\{(i_1,i_2),\ldots,(i_{k-1},i_{k+1}),\ldots,(i_{m-1},i_m)\}$.
\end{lemma}

\begin{proof}[Proof of Lemma~\ref{lem:graph_marginalization}]
  We have
  \begin{align*}
    g_{\Vset_{-k}}(\bm{x}_{\Vset_{-k}})&= \min_{x_k} \left\{\sum_{l=1}^{m-1} g_{\{i_l,i_{l+1}\}}(x_{i_l},x_{i_{l+1}}) - \sum_{l=2}^{m-1}x_{i_l}\right\}\\
  \end{align*}
  If $k=1$, then the only term containing $i_k$ is $g_{i_1,i_2}(x_{i_1},x_{i_2})$. Minimizing this gives $x_{i_2}$ and so 
  \begin{align*}
    g_{\Vset_{-1}}(\bm{x}_{\Vset_{-1}})&= \sum_{l=2}^{m-1} g_{\{i_l,i_{l+1}\}}(x_{i_l},x_{i_{l+1}}) - \sum_{l=3}^{m-1}x_{i_l}.
  \end{align*}
  Similarly, if $k=m$ then 
  \begin{align*}
    g_{\Vset_{-m}}(\bm{x}_{\Vset_{-m}})&= \sum_{l=1}^{m-2} g_{\{i_l,i_{l+1}\}}(x_{i_l},x_{i_{l+1}}) - \sum_{l=2}^{m-2}x_{i_l}.
  \end{align*}
  Otherwise, the terms containing $x_k$ are $g_{\{i_{k-1},i_{k}\}}(x_{i_{k-1}},x_{i_k}) + g_{\{i_{k},i_{k+1}\}}(x_{i_{k}},x_{i_{k+1}}) -x_{i_k}$, which when minimized gives $g_{\{i_{k-1},i_{k+1}\}}(x_{k-1},x_{k+1})$, so 
  \begin{align*}
    g_{\Vset_{-k}}(\bm{x}_{\Vset_{-k}})&= \sum_{(a,b) \in \Eset_{-k}} g_{\{a,b\}}(x_{a},x_{b}) - \sum_{l=2,l\neq k}^{m-1}x_{i_l}.
  \end{align*}
  Each of these has the form of the chain geometric extremal graphical model $\GG_{-k}$.
\end{proof}

\subsection{Proof of Proposition~\ref{prop:block_graph_betas}}
\label{sec:beta_recurrence_proof}
The proof of Proposition~\ref{prop:block_graph_betas} is split into
two parts. The crux of the proof is in the technical
Lemma~\ref{lem:two-edge-prop} in
Section~\ref{sec:beta_induction_step}, where we show that our
objective function is regularly varying with the claimed index. Given
this technical lemma, the remainder of the proof proceeds by
induction, showing that the regular variation assumptions made on the
edges in Assumption~\ref{ass:edge-RV}, lead to regular variation of
pairwise gauges not directly separated by an edge.

\color{black} We prove Proposition~\ref{prop:block_graph_betas} by
induction on the shortest path length $m_{ij}$ between $i$ and $j$. By
path reduction and marginalization (Lemmas~\ref{lem:blockgraphtopath}
and \ref{lem:graph_marginalization}), we work on the unique shortest
path $i=v_0,\dots,v_{m_{ij}}=j$, and by
Proposition~\ref{prop:block_graph_alphas} we have
$\alpha_{v_t\mid v_0}=\alpha_{v_{t-1}\mid v_0}\alpha_{v_t\mid v_{t-1}}$. 

\begin{proof}[Proof of Proposition~\ref{prop:block_graph_betas}]
  Fix $i,j$ and let $i=v_0,\ldots,v_m=j$ be their unique shortest path,
  where $m=m_{ij}$. For $t=1,\ldots,m$ set
  $f_t(x)=g_{\{v_0,v_t\}}(1,\alpha_{v_t\mid v_0}+x)-1$. Define
  $\sigma^{(1)}=1/(1-\beta_{v_1\mid v_0})$. For each $t=1,\ldots,m-1$ define
  $\sigma^{(t+1)}=1/(1-\beta_{v_{t+1}\mid v_0})$, where
  $\beta_{v_{t+1}\mid v_0}$ is the value given by the recurrence
  \eqref{eq:beta_recurrence} applied to $\beta_{v_t\mid v_0}$ and
  $\beta_{v_{t+1}\mid v_t}$ with parameters $\alpha_{v_t\mid v_0}$ and
  $\alpha_{v_{t+1}\mid v_t}$, using
  $\alpha_{v_{t+1}\mid v_0}=\alpha_{v_t\mid v_0}\alpha_{v_{t+1}\mid v_t}$.

  We prove by induction on $t$ that
  $f_t\in \RV^{0^+}_{\,\sigma^{(t)}}$. Since
  $\{v_0,v_1\}\in \Eset$, Assumption~\ref{ass:edge-RV} gives
  $f_1\in \RV^{0^+}_{\,\sigma^{(1)}}$. Assume
  $f_t\in \RV^{0^+}_{\,\sigma^{(t)}}$ for some $t<m$. Set
  $(i,\pi,j)=(v_0,v_t,v_{t+1})$ and define
  \[
    A_t^+(y)=g_{\{i,\pi\}}(1,\alpha_{\pi\mid i}+y)-1=f_t(y), \qquad
    B_t^+(z)=g_{\{\pi,j\}}(1,\alpha_{j\mid \pi}+z)-1,
  \]
  and
  \[
    f_{t+1}(x)=g_{\{i,j\}}(1,\alpha_{j\mid i}+x)-1.
  \]
  By the induction hypothesis
  $A_t^+\in \RV^{0^+}_{\,\sigma^{(t)}}$. By
  Assumption~\ref{ass:edge-RV}
  $B_t^+\in \RV^{0^+}_{\,1/(1-\beta_{v_{t+1}\mid v_t})}$. By
  Proposition~\ref{prop:block_graph_alphas} we have
  $\alpha_{j\mid i}=\alpha_{\pi\mid i}\alpha_{j\mid \pi}$. Applying Lemma \ref{lem:two-edge-prop}
  to the triple $(i,\pi,j)$ with $A^+=A_t^+$, $B^+=B_t^+$ and
  $f=f_{t+1}$ yields
  \[
    f_{t+1}\in \RV^{0^+}_{\,\tilde\sigma},
  \]
  where $\tilde\sigma$ is the index prescribed by the recurrence
  \eqref{eq:beta_recurrence} from $\sigma^{(t)}$ and
  $1/(1-\beta_{v_{t+1}\mid v_t})$ with parameters $\alpha_{\pi\mid i}$ and
  $\alpha_{j\mid \pi}$. By the definition above we have
  $\tilde\sigma=\sigma^{(t+1)}$. This closes the induction.
\end{proof}

\begin{remark}
  \color{black} The chain-min composition (cf. \eqref{eq:chain_gauge}
  and \eqref{eq:min_trivariate_gauge}) that defines
  $g_{\{v_0,v_{t+1}\}}$ from $g_{\{v_0,v_t\}}$ and
  $g_{\{v_t,v_{t+1}\}}$ preserves the basic gauge structure. It is
  $1$-homogeneous, satisfies the bound $g_{\{v_0,v_{t+1}\}}\ge \max$,
  and is continuous (by Berge’s maximum
  theorem). 
  At each step we apply
  Lemma~\ref{lem:two-edge-prop} to $(i,\pi,j)=(v_0,v_t,v_{t+1})$ with
  $A^+=A_t^+=f_t$, $B^+=B_t^+$ (from the edge $g_{\{v_t,v_{t+1}\}}$),
  and $f=f_{t+1}$; Assumption~\ref{ass:edge-RV} is invoked only for
  that edge and for the chain–min identity.  \color{black}
\end{remark}



\subsection{Induction step in \texorpdfstring{$\beta$}{b}-recurrence}
\label{sec:beta_induction_step}
Fix distinct nodes $(i,\pi,j)$ and define, for
$y\geq-\alpha_{\pi\mid i}$ and $z\geq-\alpha_{j\mid \pi}$,
\[
  A^+(y):=g_{\{i,\pi\}}(1,\alpha_{\pi\mid i}+y)-1\quad \text{and}\quad
  B^+(z):=g_{\{\pi,j\}}(1,\alpha_{j\mid \pi}+z)-1,
\]
and, for $x\ge 0$, $f(x):=g_{\{i,j\}}(1,\alpha_{j\mid i}+x)-1$.  Recall from
Proposition~\ref{prop:block_graph_alphas} that
$\alpha_{j\mid i}= \alpha_{\pi\mid i} \alpha_{j\mid \pi}$.

Set $\varepsilon_0(x):=x/\alpha_{j\mid \pi}$ for $\alpha_{j\mid \pi} > 0$,
\[
  \ell(\varepsilon,x):=\frac{x-\alpha_{j\mid \pi}\varepsilon}{\alpha_{\pi\mid i}+\varepsilon},
  \qquad \text{objective}(\varepsilon ,x) = A^+(\varepsilon)+(\alpha_{\pi\mid i}+\varepsilon) B^+\big(\ell(\varepsilon,x)\big),
\]
and note $f(x)=\min_{\varepsilon\geq-\alpha_{\pi\mid i}}\{\text{objective}(\varepsilon, x)\}$.
  \color{black}
\begin{lemma}[Induction step for $\beta$-recurrence]\label{lem:two-edge-prop} Suppose 
  $A^+(y) \in \RV^{0^+}_{\sigma_a}$ and
  $B^+(z) \in \RV^{0^+}_{\sigma_b}$, with
  $\sigma_a=1/(1-\beta_{\pi\mid i})$,
  $\sigma_b=1/(1-\beta_{j\mid \pi})$, and
  $\alpha_{j\mid i}=\alpha_{\pi\mid i}\alpha_{j\mid \pi}$. Define
  $f(x):=g_{\{i,j\}}(1,\alpha_{j\mid i}+x)-1$, $x\ge
  0$. 
  Then $f\in \RV^{0^+}_{\sigma_{j\mid i}}$ with
  $\sigma_{j\mid i}=1/(1-\beta_{j\mid i})$, where $\beta_{j\mid i}$ is given by the
  recurrence \eqref{eq:beta_recurrence}.
\end{lemma}
\color{black}  
\begin{remark}
  For each $x>0$ we minimize the chain objective
  $\varepsilon\mapsto A^+(\varepsilon)+(\alpha_{\pi\mid i}+\varepsilon)\,B^+(\ell(\varepsilon,x))$. At the boundary
  $\varepsilon=-\alpha_{\pi\mid i}$, or when taking limits toward it, we avoid the
  quotient form $\ell(\varepsilon,x)$ and work directly with the gauge via
  $(\alpha_{\pi\mid i}+\varepsilon)\,B^+(\ell(\varepsilon,x)) = g_{\{\pi,j\}}(\alpha_{\pi\mid i}+\varepsilon,\ \alpha_{\pi\mid i}\alpha_{j\mid
    \pi}+x)- (\alpha_{\pi\mid i}+\varepsilon)$, so the objective extends continuously to the
  boundary and all comparisons there are made at the gauge
  level. \color{black} The function $A^+(\varepsilon)$ is minimized at
  $\varepsilon=0$, while $B^+(\ell(\varepsilon,x))$ is minimized by
  $\varepsilon_0(x)$.  We call the A-branch the choice
  $\varepsilon=\varepsilon_0(x)$ with cost $A^+(\varepsilon_0(x))$, the B-branch the choice
  $\varepsilon=0$ with cost
  $\alpha_{\pi\mid i} B^+(x/\alpha_{\pi\mid i})$.  By dominant branch we mean the one with
  asymptotically smaller cost as $x\to 0^+$; equivalently, if
  $R(x):=A^+(\varepsilon_0(x))/\big(\alpha_{\pi\mid i} B^+(x/\alpha_{\pi\mid i})\big)$, then
  $A$ dominates if $R(x)\to 0$, $B$ dominates if $R(x)\to \infty$, and the tie
  case is $R(x)\to c\in(0,\infty)$.\color{black}

\color{black}The common structure that is adopted in the proof of
Lemma~\ref{lem:two-edge-prop} is the following\color{black}
\begin{enumerate}
  
\item Localize $\varepsilon$ to a small interval where both arguments of
  $A^+$ and $B^+$ are in a fixed compact multiplier range. Far tails
  are excluded by fixed positive constants coming from continuity, the
  lower bound $g_{a,b}(x,y)\ge\max(x,y)$, and $1$–homogeneity, and, when
  needed, by Lemma~\ref{lem:Potter0} (Potter bounds) to rule out
  intermediate regions where the the multiplier
  $\varepsilon/\varepsilon_0(x)$ could drift.
  
\item On the localized interval, by Lemma~\ref{lem:UCT}, we obtain
  uniform two‑sided ratio bounds for $A^+$ and $B^+$ on fixed compact
  multiplier ranges.
  
\item Evaluate the objective at the band centre, where one function
  vanishes in Cases A–C, bound the other function uniformly on the
  band, and conclude that $f(x)$ is asymptotically equivalent to the
  dominant branch or, when only an index is claimed, that $f$ has the
  desired regular variation
  index.  
  
\end{enumerate}

What differs across cases is the location and nature of the centre. In
Case A ($\alpha_{\pi\mid i}>0$, $\alpha_{j\mid \pi}>0$) and Case C
($\alpha_{\pi\mid i}=0$, $\alpha_{j\mid \pi}>0$) there is a moving centre
$\varepsilon_0(x)=x/\alpha_{j\mid \pi}$ that kills the $B^+\circ\ell$ function. We localize to a
band around $\varepsilon_0(x)$, use Lemma~\ref{lem:UCT} on
$[1-\eta,1+\eta]$ for $A^+$, and invoke Potter only to exclude the
intermediate right region $\varepsilon\ge(1+\eta)\varepsilon_0(x)$.

In Case B ($\alpha_{\pi\mid i}>0$, $\alpha_{j\mid \pi}=0$) the centre is fixed at
$\varepsilon=0$, which kills the $A^+$ function, so all multipliers remain in a
fixed compact around $1$ and Lemma~\ref{lem:UCT} suffices. In Case C
with $\sigma_a>1$ a crude $x$–free bound excludes the entire left
block. With $\sigma_a=1$ we do not force a strict comparison and prove only
the index $f\in\RV^{0^+}_1$ via a scaling sandwich
$f(\lambda x)\sim \lambda f(x)$, as $x\to0^+$.


\color{black} Case D ($\alpha_{\pi\mid i}=\alpha_{j\mid \pi}=0$) is different, especially
when $\sigma_a,\sigma_b>1$, as here there is no useful centre. Although
$A^+(0)=0$, at $\varepsilon=0$ the second term is
$g_{\{\pi,j\}}(0,x)\ge x$, and for $0\le \varepsilon\le x/2$ one has
$\varepsilon\,B^+(x/\varepsilon)\ge x/2$. This means the boundary is suboptimal, because a
different choice for $\varepsilon$ gives regular variation at $0^+$ with index
greater than one. Instead we choose the balanced scale
$\varepsilon=x^\gamma u$ with $\gamma=\sigma_b/(\sigma_a+\sigma_b-1)$, which makes the two contributions
comparable. Lemma~\ref{lem:Potter0} localizes $u$ to a fixed window,
and on this window Lemma~\ref{lem:UCT} ``flattens'' the slowly varying
factors, reducing the problem to minimizing a simple proxy in
$u$. This gives $f(x)=x^\rho L(x)\,(1+o(1))$ as $x\to0^+$, with $L$ slowly
varying at $0^+$ and $\rho=\sigma_a\sigma_b/(\sigma_a+\sigma_b-1) > 1$. When
$\min\{\sigma_a, \sigma_b\}=1$, the proxy loses strict curvature and the
boundary $\varepsilon=0$ can be competitive, so we treat those cases separately.
\color{black}

\end{remark}

\begin{proof}[Proof of Lemma \ref{lem:two-edge-prop}]

  

  \noindent \textbf{Case A}: $\alpha_{\pi\mid i}>0$ and
  $\alpha_{j\mid \pi}>0$.
  Evaluating the objective function at $\varepsilon=0$ and
  $\varepsilon=\varepsilon_0(x)$ gives the bounds
  \begin{equation}
    f(x)\le \alpha_{\pi\mid i}B^+\Big(\frac{x}{\alpha_{\pi\mid i}}\Big)\quad \text{and} \quad
    f(x)\le A^+\Big(\frac{x}{\alpha_{j\mid \pi}}\Big).
    \label{eq:f_upper_bounds}
  \end{equation}
  Hence,
  \begin{equation}
    f(x)\leq M(x),\quad \text{where} \quad M(x) =
    \min\left\{A^+\left(\frac{x}{\alpha_{j\mid \pi}}\right), \alpha_{\pi\mid
        i}B^+\left(\frac{x}{\alpha_{\pi\mid i}}\right)\right\}.
    \label{eq:global_upper_bound}
  \end{equation}
 
  Fix any $\varepsilon'\in (0,\alpha_{\pi\mid i}/2)$. Since $A^+$ is continuous and
  $0$ is the rightmost $\varepsilon$-minimizer of $A^+(\varepsilon)$ with
  $A^+(0)=0$, we have
  $m_{R}:=\inf\{A^+(\varepsilon): \varepsilon\ge \varepsilon'\}>0$. Hence, $A^+(\varepsilon)\geq m_R$ for all $\varepsilon\ge \varepsilon'$.

  By $1$-homogeneity,
  \[
    (\alpha_{\pi\mid i}+\varepsilon) B^+\big(\ell(\varepsilon,x)\big)
    = g_{\{\pi,j\}}(\alpha_{\pi\mid i}+\varepsilon , \alpha_{j\mid \pi}\alpha_{\pi\mid i} + x)-(\alpha_{\pi\mid i}+\varepsilon).
  \]
  From $g_{\{\pi,j\}}(x,y)\ge \max(x,y)$, we have
  \[
    \inf\Big\{\lim_{\varepsilon\rightarrow-\alpha_{\pi\mid i}^+}(\alpha_{\pi\mid i}+\varepsilon) B^+\big(\ell(\varepsilon,x)\big) : x
    \geq 0\Big\} = \inf\{g_{\{\pi,j\}}(0,\alpha_{\pi\mid i} \alpha_{j\mid \pi}+x) : x\geq 0\} \geq
    \alpha_{\pi\mid i} \alpha_{j\mid \pi} > 0.
  \]
  Also,
  $\ell(\varepsilon,x)\ge (\alpha_{j\mid \pi}\varepsilon')/(\alpha_{\pi\mid i}-\varepsilon') > 0$ for any
  $\varepsilon \in (-\alpha_{\pi\mid i}+\zeta, -\varepsilon']$ with
  $\zeta \in (0, \varepsilon' - (\alpha_{\pi\mid i}/2))$, we have
  \[
    m_{L}(\zeta):=\inf\{(\alpha_{\pi\mid i} + \varepsilon) B^+(\ell(\varepsilon,x)) : \varepsilon \in(-\alpha_{\pi\mid i}+\zeta, -\varepsilon'], x\in[0,1]\} > 0.
  \]
  Since the map
  $(\varepsilon, x)\mapsto(\alpha_{\pi\mid i} + \varepsilon)B^+(\ell(\varepsilon, x))$ is continuous on
  $[-\alpha_{\pi\mid i}, \varepsilon']\times [0,1]$,
  $\varepsilon_0(x)$ is the rightmost $\varepsilon$-minimizer of
  $B^+(\ell(\varepsilon,x))$ with $B^+(\ell(\varepsilon_0(x),x))=0$, and
  $g_{\{\pi,j\}}(x,y)\ge \max(x,y)$, we deduce
  \[
    m_{L} := \inf\{(\alpha_{\pi\mid i} + \varepsilon)B^+(\ell(\varepsilon,x)): \varepsilon \in[-\alpha_{\pi\mid i}
    ,-\varepsilon'], x\in[0,1]\}>0.
  \]

  Since $A^+(x/\alpha_{j\mid \pi})\to 0$ and
  $\alpha_{\pi\mid i}B^+(x/\alpha_{\pi\mid i})\to 0$ as $x\to 0^+$, minimizers necessarily lie
  in $(-\varepsilon',\varepsilon')$ for all sufficiently small $x$, where
  $\alpha_{\pi\mid i}+\varepsilon \in(\alpha_{\pi\mid i}/2, \alpha_{\pi\mid i}+\varepsilon')$.

  \noindent $(i)$ $\sigma_a>\sigma_b$.  Fix $\eta\in(0,1)$ and take $x$ small so that
  $(1+\eta)\varepsilon_0(x)<\varepsilon'$. Let
  $\varepsilon\ge (1+\eta)\varepsilon_0(x)$ and fix
  $\delta\in(0,\sigma_a)$ and $c>1$, so that
  $c^{-1}(1+\eta)^{\sigma_a-\delta}>1$. By Lemma~\ref{lem:Potter0} there exists
  $\varepsilon_0(\delta,c)>0$ such that for all $\varepsilon_1,\varepsilon_2\in(0,\varepsilon_0(\delta,c))$,
  \[
    c^{-1}\min\left\{\left(\frac{\varepsilon_1}{\varepsilon_2}\right)^{\sigma_a+\delta},\ \left(\frac{\varepsilon_1}{\varepsilon_2}\right)^{\sigma_a-\delta}\right\}
    \leq \frac{A^+(\varepsilon_1)}{A^+(\varepsilon_2)}\leq
    c\max\left\{\left(\frac{\varepsilon_1}{\varepsilon_2}\right)^{\sigma_a+\delta},\ \left(\frac{\varepsilon_1}{\varepsilon_2}\right)^{\sigma_a-\delta}\right\}.
  \]
  Choose $x$ small so that
  $(1+\eta)\varepsilon_0(x)<\min\{\varepsilon',\varepsilon_0(\delta,c)\}$. Then, because
  $A^+(\varepsilon)=A^+((\varepsilon/\varepsilon_0(x)) \varepsilon_0(x))$ and
  $\varepsilon/\varepsilon_0(x)\ge 1+\eta$ for any $\varepsilon$ with
  $(1+\eta)\varepsilon_0(x)\le \varepsilon<\min\{\varepsilon',\varepsilon_0(\delta,c)\}$, Lemma~\ref{lem:Potter0} gives
  \[
    A^+(\varepsilon)\ge\ c^{-1}\Big(\frac{\varepsilon}{\varepsilon_0(x)}\Big)^{\sigma_a-\delta} A^+(\varepsilon_0(x)) \ \ge\
    c^{-1}(1+\eta)^{\sigma_a-\delta} A^+(\varepsilon_0(x)) > A^+(\varepsilon_0(x)),
  \]
  for all
  $(1+\eta)\varepsilon_0(x)\le \varepsilon<\min\{\varepsilon',\varepsilon_0(\delta,c)\}$. For $\varepsilon\ge \min\{\varepsilon',\varepsilon_0(\delta,c)\}$, set
  \[
    m_{A^+}^*:=\inf\{A^+(\varepsilon) \,:\, \varepsilon\ge \min\{\varepsilon',\varepsilon_0(\delta,c)\}\} > 0.
  \]
  Since $A^+(\varepsilon_0(x))\to 0$, for $x$ small we have $m_{A^+}^*>A^+(\varepsilon_0(x))$. Therefore, for all $\varepsilon\ge (1+\eta)\varepsilon_0(x)$,
  \[
    \text{objective}(\varepsilon,x)\ \ge\ A^+(\varepsilon)\ >\ A^+(\varepsilon_0(x))\ =\ \text{objective}(\varepsilon_0(x),x),
  \]
  so these $\varepsilon$ are suboptimal relative to $\varepsilon_0(x)$.

  \color{black} If
  $-\varepsilon'\le \varepsilon\le (1-\eta)\varepsilon_0(x)$, then
  \begin{equation}
    (0,1)\ni\eta/\{1+(\varepsilon'/\alpha_{\pi\mid i})\}=: \kappa(\eta, \varepsilon') \leq \frac{\ell(\varepsilon,x)}{x/\alpha_{\pi\mid i}} \le
    \lambda(\varepsilon'):=\alpha_{\pi\mid i}/(\alpha_{\pi\mid i}-\varepsilon') \in (1,2).
    \label{eq:ell_kappa_lambda_bounds}
  \end{equation}
  Thus, from Lemma~\ref{lem:UCT} for $F=B^+$ on
  $[\kappa(\eta,\varepsilon'), \lambda(\varepsilon')]$ and the inequality
  $\alpha_{\pi\mid i}+\varepsilon\ge \alpha_{\pi\mid i}/2$ on
  $[-\varepsilon',\varepsilon']$, there exists $x_0 > 0$ and
  $c_{B}(\eta, \varepsilon')>0$ such that for all $0<x<x_0$ and all
  $\varepsilon\in[-\varepsilon',(1-\eta)\varepsilon_0(x)]$,
  \[
    (\alpha_{\pi\mid i}+\varepsilon) B^+\left(\ell(\varepsilon,x)\right) = (\alpha_{\pi\mid i}+\varepsilon)
    B^+\left(\frac{\ell(\varepsilon,x)}{x/\alpha_{\pi\mid i}}x/\alpha_{\pi\mid i}\right) \ge
    c_{B}(\eta, \varepsilon') \alpha_{\pi\mid i} B^+(x/\alpha_{\pi\mid i}),
  \]

  Since $\sigma_a > \sigma_b$, $A^+ \in \RV_{\sigma_a}^{0^+}$ and
  $B^+ \in \RV_{\sigma_b}^{0^+}$, then
  \[
    \lim_{x\to 0^{+}}\frac{\alpha_{\pi\mid i}B^+(x/\alpha_{\pi\mid i})}{A^+(\varepsilon_0(x))} = \infty,
  \]
  so there is $x_1\in(0,x_0]$ such that for all $0<x<x_1$,
  $c_B(\eta,\varepsilon') \alpha_{\pi\mid i} B^+(x/\alpha_{\pi\mid i}) > A^+(\varepsilon_0(x))$.  Therefore, for
  all such $x$ and all $\varepsilon\in(-\varepsilon',(1-\eta)\varepsilon_0(x)]$, 
  \[
    \text{objective}(\varepsilon ,x) \ge c_B(\eta,\varepsilon') \alpha_{\pi\mid i}B^+\Big(\frac{x}{\alpha_{\pi\mid
        i}} \Big) > A^+(\varepsilon_0(x)) =\text{objective}(\varepsilon_0(x) ,x),
  \]
  so these $\varepsilon$ are suboptimal relative to $\varepsilon_0(x)$. Thus, minimizers lie
  in $[(1-\eta)\varepsilon_0(x),(1+\eta)\varepsilon_0(x)]$. 
  On this interval, Lemma~\ref{lem:UCT} for $F=A^+$ on $[1-\eta,1+\eta]$ yields,
  for $x$ small,
  \[
    (1-\omega_A(\eta)) A^+(\varepsilon_0(x))\leq A^+(\varepsilon)\leq (1+\omega_A(\eta)) A^+(\varepsilon_0(x)),
  \]
  with $\omega_A(\eta)\to 0$ as $\eta\to 0$. Evaluating at the minimizer in the
  interval and at $\varepsilon=\varepsilon_0(x)$ (where $B^+(0)=0$),
  \[
    (1-\omega_A(\eta)) A^+\Big(\frac{x}{\alpha_{j\mid \pi}}\Big)\leq f(x)\leq A^+\Big(\frac{x}{\alpha_{j\mid \pi}}\Big).
  \]
  Hence, $f(x)/A^+(x/\alpha_{j\mid \pi})\to 1$. Thus, $f\in\RV^0_{\sigma_a}$ and $\beta_{j\mid i}=\beta_{\pi\mid i}$.

  \noindent $(ii)$ $\sigma_b>\sigma_a$ Fix $\eta\in(0,1)$,
  $\delta\in(0,\sigma_a)$ and $c>1$. By Lemma~\ref{lem:Potter0} there is
  $\varepsilon_0(\delta,c)>0$ such that for all $\varepsilon_1, \varepsilon_2 \in (0, \varepsilon_0(\delta, c))$,
  \[
    c^{-1} \min\left\{\left(\frac{\varepsilon_1}{\varepsilon_2}\right)^{\sigma_a - \delta},
      \left(\frac{\varepsilon_1}{\varepsilon_2}\right)^{\sigma_a + \delta}\right\}\leq
    \frac{A^+(\varepsilon_1)}{A^+(\varepsilon_2)} \leq c
    \max\left\{\left(\frac{\varepsilon_1}{\varepsilon_2}\right)^{\sigma_a - \delta}, \left(\frac{\varepsilon_1}{\varepsilon_2}\right)^{\sigma_a + \delta}\right\}.
  \]

  \color{black} Choose $x$ small so that
  $(1+\eta)\varepsilon_0(x) < \min \{\varepsilon_0(\delta,c), \varepsilon'\}$ and therefore,
  $\varepsilon_0(x) < \min \{\varepsilon_0(\delta,c), \varepsilon'\}$ and
  $\eta \varepsilon_0(x) < \min \{\varepsilon_0(\delta,c), \varepsilon'\}$. Then, because
  $A^+(\varepsilon)=A^+((\varepsilon/\varepsilon_0(x))\varepsilon_0(x))$ where $\varepsilon/\varepsilon_0(x)\ge\eta$ with $\eta\in(0,1)$, we have
  \begin{IEEEeqnarray*}{rCl}
    A^+(\varepsilon) &\ge& c^{-1}\min\left\{\left(\frac{\varepsilon}{\varepsilon_0(x)}\right)^{\sigma_a-\delta},
      \left(\frac{\varepsilon}{\varepsilon_0(x)}\right)^{\sigma_a+\delta}\right\} A^+(\varepsilon_0(x)) \ge
    c^{-1}\,\eta^{\sigma_a+\delta}A^+(\varepsilon_0(x))\\
    &=:& c_A(\eta,\delta,c)A^+(\varepsilon_0(x)),
  \end{IEEEeqnarray*}
  for all $\eta\,\varepsilon_0(x)\le \varepsilon<\min\{\varepsilon',\varepsilon_0(\delta,c)\}$.

  
  Furthermore, because $\min\{\varepsilon',\varepsilon_0(\delta, c)\}>0$ and
  $0$ is the rightmost $\varepsilon$-minimizer of $A^+(\varepsilon)$, it follows that
  $A^+(\varepsilon)\ge m_{A^+}:=\inf\{A^+(\varepsilon)\,:\,\varepsilon\ge \min\{\varepsilon',\varepsilon_0(\delta, c)\}>0$ for all
  $\varepsilon\ge \min\{\varepsilon',\varepsilon_0(\delta, c)\}$. Hence, for all $\varepsilon\ge \eta \varepsilon_0(x)$,
  \[
    \text{objective}(\varepsilon,x)\ \ge A^+(\varepsilon)\geq \min\{c_A(\eta,\delta,c) A^+(\varepsilon_0(x)),\
    m_{A^+}\}.
  \]
  Since $\sigma_a < \sigma_b$, $A^+ \in \RV_{\sigma_a}^{0^+}$ and
  $B^+ \in \RV_{\sigma_b}^{0^+}$, then
  \[
    \lim_{x\to 0^{+}}\frac{\alpha_{\pi\mid i}B^+(x/\alpha_{\pi\mid i})}{A^+(\varepsilon_0(x))} = 0,
  \]
  Therefore, for $x$ small,
  \[
    \min\{c_A(\eta,\delta,c) A^+(\varepsilon_0(x)),\ m_{A^+}\}\ >\ \alpha_{\pi\mid i}B^+\Big(\frac{x}{\alpha_{\pi\mid i}}\Big)\ =\ \text{objective}(0,x),
  \]
  and every $\varepsilon\ge \eta \varepsilon_0(x)$ is strictly suboptimal relative to $\varepsilon=0$.

  If $-\varepsilon'<\varepsilon<0$, then $\alpha_{\pi\mid i}+\varepsilon\ge \alpha_{\pi\mid i}/2$ and
  \[
    1 \leq \frac{\ell(\varepsilon,x)}{x/\alpha_{\pi\mid i}} \leq \lambda(\varepsilon')\in(1,2),
  \]
  where $\lambda(\varepsilon')$ is defined in expression
  \eqref{eq:ell_kappa_lambda_bounds}. Thus, by Lemma~\ref{lem:UCT} for
  $F=B^+$ on $[1,\lambda(\varepsilon')]$, and the inequality
  $\alpha_{\pi\mid i}+\varepsilon\ge \alpha_{\pi\mid i}/2$ on
  $[-\varepsilon',0)$, there exists $x_0 > 0$ and
  $c_{B}(\varepsilon')>0$ such that for all $0<x<x_0$ and all $\varepsilon\in[-\varepsilon',0)$,
  \[
    (\alpha_{\pi\mid i}+\varepsilon) B^+(\ell(\varepsilon,x)) = (\alpha_{\pi\mid i}+\varepsilon)
    B^+\Big(\frac{\ell(\varepsilon,x)}{x/\alpha_{\pi\mid i}}x/\alpha_{\pi\mid i}\Big) \ge 
    c_B(\varepsilon') \alpha_{\pi\mid i} B^+\Big(\frac{x}{\alpha_{\pi\mid i}}\Big),
  \]
  and such $\varepsilon$ are suboptimal relative to $\varepsilon=0$. Therefore, the
  minimizer lies in $[0,\eta \varepsilon_0(x)]$.

  On this interval,
  \[
    \kappa(\eta,\varepsilon') \leq \frac{\ell(\varepsilon,x)}{x/\alpha_{\pi\mid i}} \leq 1,
  \]
  where $\kappa(\eta, \varepsilon')$ is defined in expression
  \eqref{eq:ell_kappa_lambda_bounds} and
  $\alpha_{\pi\mid i}\le \alpha_{\pi\mid i}+\varepsilon\le \alpha_{\pi\mid i}+\eta \varepsilon_0(x)$. Lemma~\ref{lem:UCT} for
  $F=B^+$ on $[\kappa(\eta,\varepsilon'),1]$ then yields, for $x$ small,
  \[
    (1-\omega_B(\eta)) \alpha_{\pi\mid i} B^+\Big(\frac{x}{\alpha_{\pi\mid i}}\Big)\leq f(x)\leq \alpha_{\pi\mid i} B^+\Big(\frac{x}{\alpha_{\pi\mid i}}\Big),
  \]
  with $\omega_B(\eta)\to 0$ as $\eta\to 0$. Hence,
  $f(x)/(\alpha_{\pi\mid i}B^+(x/\alpha_{\pi\mid i}))\to 1$. Thus
  $f\in \RV^{0^+}_{\sigma_b}$ and $\beta_{j\mid i}=\beta_{j\mid \pi}$.

  \noindent $(iii)$ $\sigma_a=\sigma_b=:\sigma$. \color{black} Fix
  $\lambda>0$. Since $f(x)\le A^+(x/\alpha_{j\mid \pi})\to 0^+$ and
  $f(x)\le \alpha_{\pi\mid i}B^+(x/\alpha_{\pi\mid i})\to 0^+$ as
  $x\to 0^+$, any minimizer
  $\varepsilon^\star(x)\in\arg\min_{\varepsilon>-\alpha_{\pi\mid i}}\text{\rm objective}(\varepsilon,x)$ satisfies
  $A^+(\varepsilon^\star(x))\le f(x)\to 0^+$. By continuity of $A^+$ at
  $0$ and $0$ being its rightmost minimizer,
  $\varepsilon^\star(x)\to 0$ as $x\to 0^+$. Consequently
  \[
    t^\star(x):=\ell(\varepsilon^\star(x),x)=\frac{x-\alpha_{j\mid \pi}\,\varepsilon^\star(x)}{\alpha_{\pi\mid i}+\varepsilon^\star(x)}\ \to\
    0, \quad \text{as $x\to 0^+$}.
  \]
  For all $\varepsilon>-\alpha_{\pi\mid i}$ and $x>0$,
  \begin{equation}\label{eq:ell-scale}
    \ell(\lambda\,\varepsilon,\lambda x)=\frac{\lambda\,(x-\alpha_{j\mid \pi}\,\varepsilon)}{\alpha_{\pi\mid i}+\lambda\,\varepsilon}
    =\lambda\,\frac{\alpha_{\pi\mid i}+\varepsilon}{\alpha_{\pi\mid i}+\lambda\,\varepsilon}\,\ell(\varepsilon,x).
  \end{equation}
  At $\varepsilon=\varepsilon^\star(x)$, write
  \[
    \ell(\lambda\varepsilon^\star(x),\lambda x)=\lambda\,\theta(x)t^\star(x),\qquad \theta(x):=\frac{\alpha_{\pi\mid
        i}+\varepsilon^\star(x)}{\alpha_{\pi\mid i}+\lambda \varepsilon^\star(x)}
  \]
  and note that $\theta(x)\to 1$, as $x\to 0^+$.

  
  \color{black} Fix $\zeta\in(0,1)$ and note that
  $\varepsilon^\star(x),t^\star(x)\to 0^+$ as $x\to0^+$. By Lemma~\ref{lem:UCT} on the
  compact multiplier sets \color{black}$\{\lambda\}$ and
  $\{\lambda \theta(x)\}$, both lying in a fixed compact around
  $\lambda$ for $x$
  small, \color{black} 
  there
  exists $x_0=x_0(\zeta)>0$ such that, for all $0<x<x_0$, the bounds below
  hold with the same $\zeta$,
  \begin{equation}\label{eq:UCT-upper-A}
    A^+(\lambda \varepsilon^\star(x))\leq (1+\zeta)^{1/3} \lambda^{\sigma} A^+(\varepsilon^\star(x)),
  \end{equation}
  \begin{equation}\label{eq:UCT-upper-B}
    B^+(\lambda \theta(x) t^\star(x))\leq (1+\zeta)^{1/3} \big(\lambda \theta(x)\big)^{\sigma} B^+(t^\star(x)),
  \end{equation}
  and, since $\varepsilon^\star(x)\to 0$,
  \begin{equation}\label{eq:den-upper}
    \frac{\alpha_{\pi\mid i}+\lambda \varepsilon^\star(x)}{\alpha_{\pi\mid i}+\varepsilon^\star(x)}\leq (1+\zeta)^{1/3},\qquad
    \theta(x)\ \in\ \big[(1+\zeta)^{-1/3},\ (1+\zeta)^{1/3}\big].
  \end{equation}
  After shrinking $\zeta$ to some $\zeta'\in(0,\zeta)$ such that
  $(1+\zeta')^{(1+\sigma)/3}\le 1+\zeta$, and since
  $\theta(x)\in[(1+\zeta')^{-1/3},(1+\zeta')^{1/3}]$, we have
  $(\lambda \theta(x))^{\sigma}\le (1+\zeta')^{\sigma/3}\lambda^{\sigma}$. Combining this with the prefactor
  bound $(1+\zeta')^{1/3}$ from \eqref{eq:UCT-upper-B} yields
\begin{equation}\label{eq:UCT-upper-B-compressed}
  B^+(\lambda \theta(x) t^\star(x))\ \le\ (1+\zeta')^{(1+\sigma)/3}\,\lambda^{\sigma}\,B^+(t^\star(x))\ \le\ (1+\zeta)\,\lambda^{\sigma}\,B^+(t^\star(x)).
\end{equation}
  \color{black} Now let $\varepsilon^\star(\lambda x)$ be a minimizer for
  $f(\lambda x)$ and set
  $s^\star(x):=\ell(\varepsilon^\star(\lambda x),\lambda x) \to 0^+$ as
  $x\to 0^+$.  Using \eqref{eq:ell-scale} we can write
  \[
    \ell(\varepsilon^\star(\lambda x),x)=\frac{s^\star(x)}{\lambda\,\vartheta(x)},\qquad
    \vartheta(x):=\frac{\alpha_{\pi\mid i}+\varepsilon^\star(\lambda x)}{\alpha_{\pi\mid i}+\lambda\,\varepsilon^\star(\lambda x)}\ \to\ 1\quad(x\to 0^+).
  \]
  By Lemma~\ref{lem:UCT}, applied at $(\lambda x,\varepsilon^\star(\lambda x))$ to the compact multiplier sets $\{1/\lambda\}$ for $A^+$ and $\{1/(\lambda\,\vartheta(x))\}$ for $B^+$ (both lying in a fixed compact around $1/\lambda$ for $x$ sufficiently small), there exists $x_1=x_1(\zeta)>0$ such that, for all $0<x<x_1$,
  \color{black}
  \begin{equation}\label{eq:UCT-lower-A}
    A^+(\varepsilon^\star(\lambda x))\leq \frac{1}{1-\zeta} \lambda^{-\sigma} A^+(\lambda \varepsilon^\star(\lambda x)),
  \end{equation}
  \begin{equation}\label{eq:UCT-lower-B}
    B^+\Big(\frac{s^\star(x)}{\lambda \vartheta(x)}\Big)\leq \frac{1}{1-\zeta} \big(\lambda \vartheta(x)\big)^{-\sigma} B^+\big(s^\star(x)\big),
  \end{equation}
  and, by $\varepsilon^\star(\lambda x)\to 0$,
  \begin{equation}\label{eq:den-lower}
    \frac{\alpha_{\pi\mid i}+\varepsilon^\star(\lambda x)}{\alpha_{\pi\mid i}+\lambda \varepsilon^\star(\lambda x)}\leq (1+\zeta)^{1/2},\qquad
    \vartheta(x)\ \in\ \big[(1+\zeta)^{-1/2},\ (1+\zeta)^{1/2}\big].
  \end{equation}
  From \eqref{eq:UCT-lower-B} and \eqref{eq:den-lower},
  $(\lambda \vartheta(x))^{-\sigma}\le (1+\zeta)^{1/2} \lambda^{-\sigma}$ for $0< x < x_1$.

  Let $x^\star = \min\{x_0, x_1\}$. Using \eqref{eq:ell-scale} at
  $\varepsilon=\varepsilon^\star(x)$, then \eqref{eq:UCT-upper-A},
  \eqref{eq:UCT-upper-B-compressed}, and \eqref{eq:den-upper}, for all
  $0<x<x^\star$,
  \[
    \begin{aligned}
      f(\lambda x)
      &\le\ \text{\rm objective}(\lambda\,\varepsilon^\star(x),\lambda x)\\
      &=\ A^+(\lambda\,\varepsilon^\star(x))+(\alpha_{\pi\mid i}+\lambda\,\varepsilon^\star(x))\,B^+\big(\lambda\,\theta(x)\,t^\star(x)\big)\\
      &\le\ (1+\zeta)^{1/3}\,\lambda^{\sigma}\,A^+(\varepsilon^\star(x))\ +\ (1+\zeta)^{1/3}\,(1+\zeta)^{2/3}\,\lambda^{\sigma}\,(\alpha_{\pi\mid i}+\varepsilon^\star(x))\,B^+(t^\star(x))\\
      &\leq\ (1+\zeta)\,\lambda^{\sigma}\,\big[A^+(\varepsilon^\star(x))+(\alpha_{\pi\mid i}+\varepsilon^\star(x))\,B^+(t^\star(x))\big]\\
      &=\ (1+\zeta)\,\lambda^{\sigma}\,f(x),
    \end{aligned}
  \]
  so that
  \begin{equation}\label{eq:upper-bound}
    f(\lambda x)\ \le\ (1+\zeta)\,\lambda^{\sigma}\,f(x).
  \end{equation}
  Next, we apply \eqref{eq:upper-bound} with $(x,\lambda)$ replaced by
  $(\lambda x,\,1/\lambda)$. For sufficiently small $x$, this gives
  \begin{equation}\label{eq:lower-from-upper}
    f(\lambda x) \ge \frac{\lambda^{\sigma}}{1+\zeta}\,f(x).
  \end{equation}
  Combining \eqref{eq:upper-bound} and \eqref{eq:lower-from-upper}, for
  all sufficiently small $x>0$,
  \[
    \frac{\lambda^{\sigma}}{1+\zeta}\ \le\ \frac{f(\lambda x)}{f(x)}\ \le\ (1+\zeta)\,\lambda^{\sigma}.
  \]
  Letting $x\to 0^+$ and then $\zeta\to 0^+$ yields
  $f(\lambda x)/f(x)\to \lambda^{\sigma}$ for each fixed $\lambda>0$,
  i.e., $f\in\RV^{0^+}_{\sigma}$.

  \color{black}

  \noindent\textbf{Case B:} $\alpha_{\pi\mid i}>0$ and $\alpha_{j\mid \pi}=0$.
  Evaluating the objective function at $\varepsilon=0$ gives
  $f(x)\le \alpha_{\pi\mid i}\,B^+(x/\alpha_{\pi\mid i})$. Also
  $A^+(\varepsilon)=g_{\{i,\pi\}}(1,\alpha_{\pi\mid i}+\varepsilon)-1\ge 0$ for all
  $\varepsilon>-\alpha_{\pi\mid i}$. Fix
  $\varepsilon'\in(0,\alpha_{\pi\mid i}/2)$. Since $A^+$ is continuous and
  $A^+(0)=0$ is the rightmost minimizer,
  $m_R:=\inf\{A^+(\varepsilon):\varepsilon\ge \varepsilon'\}>0$, hence any
  $\varepsilon\ge \varepsilon'$ is suboptimal for $x$ small because
  $f(x)\le \alpha_{\pi\mid i}B^+(x/\alpha_{\pi\mid i})\to 0$.

  Next, consider $\varepsilon\le -\varepsilon'$. Fix
  $\delta\in(0,\alpha_{\pi\mid i}-\varepsilon')$. For
  $\varepsilon\in[-\alpha_{\pi\mid i}+\delta,-\varepsilon']$ the multiplier
  $u(\varepsilon):=\alpha_{\pi\mid i}/(\alpha_{\pi\mid i}+\varepsilon)$ ranges over the fixed compact
  $[u_{\min},u_{\max}]$, with
  $u_{\min}:=\alpha_{\pi\mid i}/(\alpha_{\pi\mid i}-\varepsilon')>1$ and
  $u_{\max}:=\alpha_{\pi\mid i}/\delta$. By Lemma~\ref{lem:UCT} for
  $F=B^+$ on $[u_{\min},u_{\max}]$, there exist
  $c_{B,1}(\delta,\varepsilon')\in(0,1)$ and $x_0>0$ such that, for all
  $0<x<x_0$ and all $\varepsilon\in[-\alpha_{\pi\mid i}+\delta,-\varepsilon']$,
  \[
    B^+(x/(\alpha_{\pi\mid i}+\varepsilon))=B^+(u(\varepsilon)\,x/\alpha_{\pi\mid i})\ge c_{B,1}(\delta,\varepsilon')\,B^+(x/\alpha_{\pi\mid i}).
  \]
  Hence, using $A^+(\varepsilon)\ge 0$,
  \[
    \text{objective}(\varepsilon,x)\ge (\alpha_{\pi\mid i}+\varepsilon)\,B^+(x/(\alpha_{\pi\mid i}+\varepsilon))
    \ge c_{B,1}(\delta,\varepsilon')\,\alpha_{\pi\mid i}\,B^+(x/\alpha_{\pi\mid i}).
  \]

  For $\varepsilon\in(-\alpha_{\pi\mid i},-\alpha_{\pi\mid i}+\delta)$ the multiplier
  $u(\varepsilon)=\alpha_{\pi\mid i}/(\alpha_{\pi\mid i}+\varepsilon)\ge \alpha_{\pi\mid i}/\delta$ is unbounded as
  $\varepsilon\downarrow -\alpha_{\pi\mid i}$. By Potter bounds at $0^+$ for
  $B^+$ with index $\sigma_b\ge 1$, for any $\eta\in(0,1)$ there exist
  $C(\eta)>1$ and $x_1>0$ such that, for all $0<x<x_1$ and all
  $\varepsilon$ with $\alpha_{\pi\mid i}+\varepsilon\in(0,\delta]$,
  \[
    B^+(u(\varepsilon)\,x/\alpha_{\pi\mid i})\ge C(\eta)^{-1}\,u(\varepsilon)^{\sigma_b-\eta}\,B^+(x/\alpha_{\pi\mid i}).
  \]
  Therefore
  \begin{IEEEeqnarray*}{rcl}
    \text{objective}(\varepsilon,x)\ge (\alpha_{\pi\mid i}+\varepsilon)\,B^+(x/(\alpha_{\pi\mid i}+\varepsilon))
    &=& \alpha_{\pi\mid i}\,u(\varepsilon)^{-1}\,B^+(u(\varepsilon)\,x/\alpha_{\pi\mid i}) \\
    &\ge& C(\eta)^{-1}\,\alpha_{\pi\mid i}\,u(\varepsilon)^{\sigma_b-\eta-1}\,B^+(x/\alpha_{\pi\mid i}).
  \end{IEEEeqnarray*}

  In what follows, we split cases to $\sigma_b>1$ and $\sigma_b = 1$.

  If $\sigma_b>1$, choose $\eta\in(0,\sigma_b-1)$. Then
  $\sigma_b-\eta-1>0$, so the minimum over the strip occurs at the smallest
  multiplier $u(\varepsilon)=\alpha_{\pi\mid i}/\delta$. Thus
  \[
    \text{objective}(\varepsilon,x)\ge c_{B,2}(\delta,\eta)\,\alpha_{\pi\mid i}\,B^+(x/\alpha_{\pi\mid i}),\qquad c_{B,2}(\delta,\eta):=C(\eta)^{-1}\,(\alpha_{\pi\mid i}/\delta)^{\sigma_b-\eta-1}.
  \]
  Combining, there exist $c_B(\varepsilon',\delta,\eta)\in(0,1)$ and
  $x^\ast>0$ such that, for all $0<x<x^\ast$ and all $\varepsilon\le -\varepsilon'$,
  \[
    \text{objective}(\varepsilon,x)\ge c_B(\varepsilon',\delta,\eta)\,\alpha_{\pi\mid i}\,B^+(x/\alpha_{\pi\mid i}).
  \]
  For $\varepsilon\in(-\varepsilon',\varepsilon')$ we have
  $\alpha_{\pi\mid i}+\varepsilon\in[\alpha_{\pi\mid i}/2,\alpha_{\pi\mid i}+\varepsilon']$, and
  $u(\varepsilon)\in[\,\alpha_{\pi\mid i}/(\alpha_{\pi\mid i}+\varepsilon'),\,1\,]$. By Lemma~\ref{lem:UCT} for
  $F=B^+$ on that compact and $A^+(\varepsilon)\ge 0$, for $x$ small
  \[
    (1-\omega_B(\varepsilon'))\,\alpha_{\pi\mid i}\,B^+(x/\alpha_{\pi\mid i})\le f(x)\le \alpha_{\pi\mid i}\,B^+(x/\alpha_{\pi\mid i}),
  \]
  hence $f(x)/(\alpha_{\pi\mid i}B^+(x/\alpha_{\pi\mid i}))\to 1$, i.e., $f\in\RV^{0^+}_{\sigma_b}$ and $\beta_{j\mid i}=\beta_{j\mid \pi}$.

  If $\sigma_b=1$, we only claim the index and avoid any localization at the
  minimizer. Fix $\lambda>0$ and $\delta>0$. By Lemma~\ref{lem:UCT} for
  $B^+$ at fixed multipliers $\lambda$ and $1/\lambda$, there exist
  $x_3,x_4>0$ such that, for all $0<x<\min\{x_3,x_4\}$ and all small
  $t$,
  \[
    B^+(\lambda t)\le (1+\delta)\,\lambda\,B^+(t),\qquad B^+(t/\lambda)\le (1+\delta)\,(1/\lambda)\,B^+(t).
  \]
  For the upper bound, note that for any $\varepsilon\ge -\alpha_{\pi\mid i}/2$,
  \[
    \text{objective}(\varepsilon,\lambda x)\le A^+(\varepsilon)+(1+\delta)\,\lambda\,(\alpha_{\pi\mid i}+\varepsilon)\,B^+(x/(\alpha_{\pi\mid i}+\varepsilon))\le (1+\delta)\,\lambda\,\text{objective}(\varepsilon,x),
  \]
  so taking inf over $\varepsilon\ge -\alpha_{\pi\mid i}/2$ yields $f(\lambda x)\le (1+\delta)\,\lambda\,\inf_{\varepsilon\ge -\alpha_{\pi\mid i}/2}\text{objective}(\varepsilon,x)\le (1+\delta)\,\lambda\,f(x)$.

  For the lower bound, note that for any $\varepsilon\ge -\alpha_{\pi\mid i}/2$,
  \[
    \text{objective}(\varepsilon,x)\le A^+(\varepsilon)+(1+\delta)\,(1/\lambda)\,(\alpha_{\pi\mid i}+\varepsilon)\,B^+(\lambda x/(\alpha_{\pi\mid i}+\varepsilon))\le (1+\delta)\,(1/\lambda)\,\text{objective}(\varepsilon,\lambda x),
  \]
  so taking inf over $\varepsilon\ge -\alpha_{\pi\mid i}/2$ gives
  $f(x)\le (1+\delta)\,(1/\lambda)\,f(\lambda x)$, i.e.,
  $f(\lambda x)\ge \lambda\,(1+\delta)^{-1}\,f(x)$. Combining,
  \[
    \lambda\,(1+\delta)^{-1}\le \frac{f(\lambda x)}{f(x)}\le (1+\delta)\,\lambda\qquad(0<x<\min\{x_3,x_4\}),
  \]
  and letting $x\to 0^+$ then $\delta\downarrow 0$ yields
  $f(\lambda x)/f(x)\to \lambda$, i.e., $f\in\RV^{0^+}_1$.

  \noindent \textbf{Case C:} $\alpha_{\pi\mid i}=0$ and $\alpha_{j\mid \pi}>0$. Here
  \[
    f(x)=\min_{\varepsilon\ge 0}\{A^+(\varepsilon)+\varepsilon(g_{\{\pi,j\}}(1,x/\varepsilon)-1)\}, \qquad
    \varepsilon_0(x):=x/\alpha_{j\mid \pi}.
  \]
  Evaluating the objective at $\varepsilon=\varepsilon_0(x)$ gives
  $f(x)\le A^+(x/\alpha_{j\mid \pi})$ and $\ell(\varepsilon_0(x),x)=0$.

  Fix $\varepsilon'\in(0,1)$ and $\eta\in(0,1)$. The exclusion of the entire right
  interval $\{\varepsilon\,:\,\varepsilon\ge (1+\eta)\varepsilon_0(x)\}$ follows exactly as in Case A
  $(i)$. Below the fixed cutoff $\varepsilon'$ we apply Lemma~\ref{lem:Potter0} to
  obtain a uniform constant which is greater than unity, comparing
  $A^+(\varepsilon)$ with $A^+(\varepsilon_0(x))$ for all
  $(1+\eta)\varepsilon_0(x)\le \varepsilon<\varepsilon'$, and for
  $\varepsilon\ge \varepsilon'$ we use the fixed tail constant
  $m_{A^+}:=\inf\{A^+(\varepsilon)>0\,:\, \varepsilon\ge \varepsilon'\}$ while
  $A^+(\varepsilon_0(x))\to 0$. Thus, for $x$ small,
  \[
    \text{objective}(\varepsilon,x) > \text{objective}(\varepsilon_0(x),x)=A^+(\varepsilon_0(x)), \qquad
    \text{ for all $\varepsilon\ge (1+\eta)\varepsilon_0(x)$}.
  \]
  On the interval
  $\varepsilon\in[(1-\eta)\varepsilon_0(x),(1+\eta)\varepsilon_0(x)]$, as in Case A $(i)$ we use
  Lemma~\ref{lem:UCT} for $F=A^+$ on $[1-\eta,1+\eta]$ and obtain
  \begin{equation}\label{eq:caseC-band}
    (1-\omega_A(\eta))\,A^+(\varepsilon_0(x))\leq A^+(\varepsilon)\leq (1+\omega_A(\eta))\,A^+(\varepsilon_0(x)),
  \end{equation}
  with $\omega_A(\eta)\to 0$ as $\eta\to 0$. Since
  $B^+(\ell(\varepsilon,x))$ vanishes at $\varepsilon=\varepsilon_0(x)$, evaluating the objective there
  gives the matching upper bound $f(x)\le A^+(x/\alpha_{j\mid \pi})$.

  For the left interval $\varepsilon < (1-\eta) \varepsilon_0(x)$ we split by
  $\sigma_a$. Suppose $\sigma_a>1$. Evaluating the objective function at
  $\varepsilon=\varepsilon_0(x)$ gives
  $f(x)\le A^+(x/\alpha_{j\mid \pi})$. Fix
  $\eta\in(1-\alpha_{j\mid \pi},\,1)$. Then for every
  $\varepsilon\in[0,(1-\eta)\varepsilon_0(x)]$ we have
  $x/\varepsilon\ge x/\{(1-\eta)\varepsilon_0(x)\} = \alpha_{j\mid \pi}/(1-\eta) > 1$ and hence,
  $\varepsilon(g_{\{\pi,j\}}(1,x/\varepsilon)-1) \ge \varepsilon((x/\varepsilon)-1) = x-\varepsilon \ge x-(1-\eta)\,\varepsilon_0(x) =
  (1-(1-\eta)/\alpha_{j\mid \pi})\,x =:\ c(\eta)\,x$, with
  $c(\eta)>0$. Also, $A^+(\varepsilon)\geq0$, and therefore
  \[
    \text{objective}(\varepsilon,x) = A^+(\varepsilon) + \varepsilon(g_{\{\pi,j\}}(1,x/\varepsilon)-1) \geq c(\eta) x, \quad
    \text{for all $\varepsilon\in[0,(1-\eta)\varepsilon_0(x)]$}.
  \]
  This excludes uniformly the whole left interval
  $[0,(1-\eta)\varepsilon_0(x)]$. In particular, because $\sigma_a>1$, then
  $\text{objective}(\varepsilon_0(x),x) = A^+(x/\alpha_{j\mid \pi})=o(x)$, so every
  $\varepsilon$ in this interval is strictly suboptimal relative to $\varepsilon_0(x)$.

  Combining the left-interval exclusion with the right-interval
  exclusion and \eqref{eq:caseC-band} yields
  \[
    (1-\omega_A(\eta))\,A^+\Big(\frac{x}{\alpha_{j\mid \pi}}\Big)\leq f(x)\leq
    A^+\Big(\frac{x}{\alpha_{j\mid \pi}}\Big),
  \]
  for all sufficiently small $x>0$. Hence,
  $f(x)\sim A^+(x/\alpha_{j\mid \pi})$ as $x\to 0^+$, i.e.,
  $f\in \RV^{0^+}_{\sigma_a}$ and $\beta_{j\mid i}=\beta_{\pi\mid i}$.

  Next, suppose $\sigma_a=1$. The right–interval exclusion and
  \eqref{eq:caseC-band} above hold verbatim for $\sigma_a=1$. To conclude
  $f\in \RV^{0^+}_1$ we use the following scaling argument.

  Fix $\lambda>0$ and let $\varepsilon^\star(x)$ be a minimizer for
  $f(x)$ satisfying $\varepsilon^\star(x)\to0$ as $x\to0^{+}$. Since $A^+\in \RV^{0^+}_1$, by Lemma~\ref{lem:UCT}
  at fixed $\lambda > 0$, we have that for any $\delta>0$, there exists
  $x_0>0$ such that
  $A^+(\lambda\,\varepsilon^\star(x)) \le (1+\delta) \, \lambda\, A^+(\varepsilon^\star(x))$, for all
  $0<x<x_0$. Thus, for any $\delta > 0$, we get
  \begin{IEEEeqnarray*}{rCl}
    f(\lambda x) \leq \text{objective}(\lambda \varepsilon^\star(x), \lambda x)
    &=&A^+(\lambda\,\varepsilon^\star(x))+\lambda\,\varepsilon^\star(x)(g_{\{\pi,j\}}(1,x/\varepsilon^\star(x))-1) \leq (1+\delta)\,
    \lambda\,f(x)
  \end{IEEEeqnarray*}
  so $\limsup_{x\to 0^+} f(\lambda x)/f(x)\le \lambda$.

  Conversely, let $\varepsilon^\star(\lambda x)$ be a minimizer for
  $f(\lambda x)$. Since $A^+\in\RV^{0^+}_1$, by Lemma~\ref{lem:UCT} at
  fixed $1/\lambda>0$, for any $\delta>0$ there exists $x_1>0$ such that
  $A^+(\varepsilon^\star(\lambda x)/\lambda)\leq (1+\delta)\,A^+(\varepsilon^\star(\lambda x))/\lambda$ for all $0<x<x_1$. Hence,
  \begin{IEEEeqnarray*}{rCl}
    f(x) &=&\text{objective}(\varepsilon^\star(\lambda x)/\lambda,x)
    = A^+(\varepsilon^\star(\lambda x)/\lambda)\ +\ (\varepsilon^\star(\lambda x)/\lambda)\Big(g_{\{\pi,j\}}(1,x/(\varepsilon^\star(\lambda x)/\lambda))-1\Big)\\
    &=& A^+(\varepsilon^\star(\lambda x)/\lambda)\ +\ (1/\lambda)\,\varepsilon^\star(\lambda x)\Big(g_{\{\pi,j\}}(1,(\lambda x)/\varepsilon^\star(\lambda x))-1\Big)\\
    &\le& (1+\delta)\,A^+(\varepsilon^\star(\lambda x))/\lambda\ +\ (1/\lambda)\,\varepsilon^\star(\lambda x)\Big(g_{\{\pi,j\}}(1,(\lambda x)/\varepsilon^\star(\lambda x))-1\Big)\\
    &\le& (1+\delta)\,f(\lambda x)/\lambda,
  \end{IEEEeqnarray*}
  for all $0<x<x_1$, so
  $\liminf_{x\to 0^+} f(\lambda x)/f(x) \ge \lambda$. Thus,
  $\lim_{x\to 0^+}f(\lambda x)/f(x)=\lambda$ and $\beta_{j\mid i}=\beta_{\pi\mid i}$.

  \color{black}
  \noindent \textbf{Case D:} $\alpha_{\pi\mid i}=0$ and $\alpha_{j\mid \pi}=0$.\newline
  \noindent \emph{Subcase D.1: $\sigma_a>1$ and $\sigma_b>1$}. We have
  $f(x)=\min_{\varepsilon\ge 0}\{A^+(\varepsilon)+\varepsilon\,B^+(x/\varepsilon)\}$ where
\[
  A^+(\varepsilon)=\varepsilon^{\sigma_a}L_a(\varepsilon),\quad B^+(\varepsilon)=\varepsilon^{\sigma_b}L_b(\varepsilon),
\]
with $L_a,L_b$ slowly varying at $0^+$. Set
\[
  \gamma:=\frac{\sigma_b}{\sigma_a+\sigma_b-1},\qquad \rho:=\frac{\sigma_a\sigma_b}{\sigma_a+\sigma_b-1},\qquad
  r(x):=\frac{L_b(x^{1-\gamma})}{L_a(x^\gamma)}.
\]
Change variables $\varepsilon=x^\gamma u$. Then
\[
  A^+(x^\gamma u)=x^\rho\,u^{\sigma_a}L_a(x^\gamma u),\qquad
  x^\gamma u\,B^+\big(x/(x^\gamma u)\big)=x^\rho\,u^{1-\sigma_b}L_b(x^{1-\gamma}/u),
\]
so
\[
  f(x)=x^\rho\,\inf_{u>0}F_x(u),\text{ with }
  F_x(u):=u^{\sigma_a}L_a(x^\gamma u)+u^{1-\sigma_b}L_b(x^{1-\gamma}/u).
\]

Fix $\varepsilon\in(0,1)$ and define the balance point
$\bar u(x):=r(x)^{1/(\sigma_a+\sigma_b-1)}$. Write $u=\bar u(x)\,v$ with
$v>0$. Then
\[
  F_x(\bar u\,v)=\bar u(x)^{\sigma_a}\,v^{\sigma_a}\,L_a\!\big(x^\gamma \bar u(x)\,v\big)+\bar u(x)^{\,1-\sigma_b}\,v^{1-\sigma_b}\,L_b\!\big(x^{1-\gamma}/(\bar u(x)\,v)\big).
\]
Fix $C\ge 1$. By Lemma~\ref{lem:UCT}, there exists $x_0=x_0(C,\varepsilon)>0$ such that, for all $0<x<x_0$ and all $v\in[1/C,C]$,
\[
  (1-\varepsilon)\,L_a\!\big(x^\gamma \bar u(x)\big)\ \le\ L_a\!\big(x^\gamma \bar u(x)\,v\big)\ \le\ (1+\varepsilon)\,L_a\!\big(x^\gamma \bar u(x)\big),
\]
\[
  (1-\varepsilon)\,L_b\!\big(x^{1-\gamma}/\bar u(x)\big)\ \le\ L_b\!\big(x^{1-\gamma}/(\bar u(x)\,v)\big)\ \le\ (1+\varepsilon)\,L_b\!\big(x^{1-\gamma}/\bar u(x)\big).
\]
Hence, uniformly for $v\in[1/C,C]$ and $0<x<x_0$,
\[
(1-\varepsilon)\,G_x(v)\ \le\ F_x\big(\bar u(x)\,v\big)\ \le\ (1+\varepsilon)\,G_x(v),
\]
where
\color{black}
\begin{IEEEeqnarray*}{rCl}
  G_x(v) &:=& \bar u(x)^{\sigma_a}\,v^{\sigma_a}\,L_a\!\big(x^\gamma \bar u(x)\big)\ +\ \bar u(x)^{\,1-\sigma_b}\,v^{1-\sigma_b}\,L_b\!\big(x^{1-\gamma}/\bar u(x)\big)\\\\
  &=& \bar u(x)^{\sigma_a}\,L_a\!\big(x^\gamma \bar u(x)\big)\,[v^{\sigma_a}+\tilde r(x)\,v^{1-\sigma_b}],
\end{IEEEeqnarray*}
with
\[
  \tilde r(x):=\bar u(x)^{1-\sigma_b-\sigma_a}\,\frac{L_b\!\big(x^{1-\gamma}/\bar
    u(x)\big)}{L_a\!\big(x^\gamma \bar u(x)\big)}.
\]
\color{black}
For the tails $v\ge C$ and $v\le 1/C$ we apply Lemma~\ref{lem:Potter0} to
the regularly varying functions $A^+(\cdot)$ and $B^+(\cdot)$. Fix
$\delta\in(0,\min\{\sigma_a-1,\sigma_b-1,1\})$ and let
$c:=c(C,\delta)\ge 1$ and $x_1=x_1(C,\delta)>0$ be as given by
Lemma~\ref{lem:Potter0}. For all $0<x<x_1$ and all $v\ge C$,
\[
  \frac{A^+(x^\gamma \bar u\,v)}{A^+(x^\gamma \bar u)}\ \ge\ c^{-1}\,v^{\sigma_a-\delta},\qquad
  \frac{B^+\big(x^{1-\gamma}/(\bar u\,v)\big)}{B^+\big(x^{1-\gamma}/\bar u\big)}\ \le\ c\,v^{-(\sigma_b-\delta)}.
\]
\color{black}
Using $A^+(t)=t^{\sigma_a}L_a(t)$ and $B^+(t)=t^{\sigma_b}L_b(t)$, we can rewrite
\[
  F_x(\bar u\,v)\;=\;x^{-\gamma\sigma_a}\,A^+(x^\gamma \bar u\,v)\;+\;x^{-(1-\gamma)\sigma_b}\,(\bar u v)\,B^+\!\big(x^{1-\gamma}/(\bar u v)\big).
\]
Therefore, by Lemma~\ref{lem:Potter0}, for all $v\ge C$,
\[
  A^+(x^\gamma \bar u\,v)\ \ge\ c^{-1}\,v^{\sigma_a-\delta}\,A^+(x^\gamma \bar u),
\]
and hence
\begin{equation}
  F_x(\bar u\,v)\ \ge\ x^{-\gamma\sigma_a}\,c^{-1}\,v^{\sigma_a-\delta}\,A^+(x^\gamma \bar u)
  \;=\;c^{-1}\,\bar u^{\sigma_a}\,v^{\sigma_a-\delta}\,L_a\!\big(x^\gamma \bar u\big).
  \label{eq:tail_bound1}
\end{equation}
Similarly, for all $v\le 1/C$,
\[
  B^+\!\big(x^{1-\gamma}/(\bar u v)\big)\ \ge\ c^{-1}\,v^{-(\sigma_b-\delta)}\,B^+\!\big(x^{1-\gamma}/\bar u\big),
\]
so
\begin{equation}
  F_x(\bar u\,v)\ \ge\ x^{-(1-\gamma)\sigma_b}\,(\bar u v)\,c^{-1}\,v^{-(\sigma_b-\delta)}\,B^+\!\big(x^{1-\gamma}/\bar u\big)
  \;=\;c^{-1}\,\bar u^{1-\sigma_b}\,v^{\,1-\sigma_b+\delta}\,L_b\!\big(x^{1-\gamma}/\bar u\big).
  \label{eq:tail_bound2}
\end{equation}
\color{black}

Let $\phi(r):=\inf_{v>0}\big[v^{\sigma_a}+r\,v^{1-\sigma_b}\big]$ and
$\phi_{[C]}(r):=\inf_{v\in[1/C,C]}\big[v^{\sigma_a}+r\,v^{1-\sigma_b}\big]$. Note
that $v\mapsto v^{\sigma_a}+r\,v^{\,1-\sigma_b}$ has a unique minimizer on $(0,\infty)$ at
\[
  \color{black}v_\star(r):=\big((\sigma_b-1)r/\sigma_a\big)^{1/(\sigma_a+\sigma_b-1)},
  \color{black}
\]
\color{black} Since $\tilde r$ is slowly varying at $0^+$, there exist
$x_\ast>0$ and constants
$0<\tilde r_{\min}\le \tilde r(x)\le \tilde r_{\max}<\infty$ for all
$0<x<x_\ast$. The map $v\mapsto v^{\sigma_a}+\tilde r\,v^{1-\sigma_b}$ has a unique
minimizer $v_\star(\tilde r)$, which is continuous and increasing in
$\tilde r$. Hence
$v_\star(\tilde r(x))\in[\,v_\star(\tilde r_{\min}),\,v_\star(\tilde r_{\max})\,]$
for $0<x<x_\ast$. Choosing
\[
  C\ \ge\ \max\{\,v_\star(\tilde r_{\max}),\ 1/v_\star(\tilde r_{\min})\,\}
\]
ensures $v_\star(\tilde r(x))\in[1/C,\,C]$ for all sufficiently small $x$.


With $C$ as above so that $v_\star(\tilde r(x))\in[1/C,\,C]$ for all
sufficiently small $x$, we obtain, for $0<x<\min\{x_0,x_1\}$,
\[
  \inf_{v\in[1/C,C]} F_x\big(\bar u(x)\,v\big)\ \asymp\ G_x\big(v_\star(\tilde
  r(x))\big), 
\]
with two-sided comparison factors bounded by $(1\pm\varepsilon)$, uniformly in $v\in[1/C,C]$.
Here, we write $U(x) \asymp V(x)$ as $x\to0^+$ to mean there exist constants
$0< c_1\leq c_2<\infty$ and $x_1>0$, independent of $x$, such that
$c_1 V(x) \leq U(x) \leq c_2 V(x)$ for all $0<x<x_1$. Moreover, by the tail
bounds \eqref{eq:tail_bound1} and \eqref{eq:tail_bound2}, and since $v\mapsto v^{\sigma_a-\delta}$ and
$v\mapsto v^{1-\sigma_b+\delta}$ are monotone on $(1,\infty)$ and
$(0,1)$ respectively, we can choose $C$ sufficiently large (depending
only on $\varepsilon,\delta,\sigma_a,\sigma_b$) so that both tail infima
\[
  c^{-1}C^{\sigma_a-\delta}\,\bar u(x)^{\sigma_a}L_a(x^\gamma \bar u)\qquad\text{and}\qquad
  c^{-1}C^{\sigma_b-1-\delta}\,\bar u(x)^{1-\sigma_b}L_b(x^{1-\gamma}/\bar u)
\]
exceed $(1+\varepsilon)\,G_x\big(v_\star(\tilde r(x))\big)$. Therefore, for all
$0<x<\min\{x_0,x_1\}$,
\[
  \inf_{v\notin[1/C,C]}F_x(\bar u\,v)\ \ge\ \inf_{v\in[1/C,C]}F_x(\bar u\,v),
\]
and the global infimum of $F_x(u)$ over $u>0$ equals its infimum over $u\in[\bar u(x)/C,\ \bar u(x)\,C]$. Consequently,
\color{black}
\[
  f(x)=x^\rho\,\inf_{v\in[1/C,C]}F_x\big(\bar u(x)\,v\big)\ \asymp\ x^\rho\,\bar u(x)^{\sigma_a}\,L_a\!\big(x^\gamma \bar u(x)\big)\,\phi_{[C]}\big(\tilde r(x)\big).
\]
\color{black}

\color{black} It remains to show that 
$L(x):=\bar u(x)^{\sigma_a}\,L_a\!\big(x^\gamma \bar
u(x)\big)\,\phi_{[C]}\big(\tilde r(x)\big)$ is slowly varying at
$0^+$. For each fixed $\lambda>0$ define $y(x):=x^\gamma \bar u(x)$. Since
$r(x)=L_b(x^{1-\gamma})/L_a(x^\gamma)$ and $L_a,L_b$ are slowly varying at
$0^+$, we have $r(\lambda x)/r(x)\to 1$, hence
$\bar u(\lambda x)/\bar u(x)=\big(r(\lambda x)/r(x)\big)^{1/(\sigma_a+\sigma_b-1)}\to
1$. Consequently,
\[
  \frac{y(\lambda x)}{y(x)}=\frac{(\lambda x)^\gamma\,\bar u(\lambda x)}{x^\gamma\,\bar u(x)}\ \to\ \lambda^\gamma\qquad(x\to 0^+).
\]
Fix $\lambda>0$ and choose $\kappa>1$ so that, for all sufficiently small
$x>0$,
$y(\lambda x)/y(x) \in [\lambda^\gamma/ \kappa, \kappa \lambda^\gamma]\subset(0,\infty)$.  By the uniform convergence
theorem for slow variation at $0^+$, it follows that
$L_a(y(\lambda x))/L_a(y(x))\ \to 1$ as $x\to 0^+$. Moreover,
$r(\lambda x)/r(x)\to 1$ and $r\mapsto\phi_{[C]}(r)$ is Lipschitz on
$(0,\infty)$, hence \color{black}
\[
  \frac{\phi_{[C]}(\tilde r(\lambda x))}{\phi_{[C]}(\tilde r(x))}\ \to\ 1, \qquad(x\to
  0^+).
\]
\color{black}
Combining the two displays yields, for each fixed $\lambda>0$,
\[
  \frac{L(\lambda x)}{L(x)}\ =\ \frac{L_a\!\big(y(\lambda x)\big)}{L_a\!\big(y(x)\big)}\cdot \frac{\phi_{[C]}(r(\lambda x))}{\phi_{[C]}(r(x))}\ \to\ 1\qquad(x\to 0^+),
\]
i.e., $L$ is slowly varying at $0^+$. Therefore,
\[
  \frac{f(\lambda x)}{f(x)}=\lambda^\rho\,\frac{L(\lambda x)}{L(x)}\cdot(1+o(1))\ \to\ \lambda^\rho\qquad(x\to 0^+),
\]
which shows $f\in\RV^{0^+}_\rho$ with $\rho=\sigma_a\sigma_b/(\sigma_a+\sigma_b-1)$ and $\beta_{j\mid i}=\beta_{\pi\mid i}\,\beta_{j\mid \pi}$.

\noindent \emph{Subcase D.2: $\min(\sigma_a,\sigma_b)=1$.}  We show that
$f\in\RV^{0^+}_1$ by a scaling sandwich based only on
Lemma~\ref{lem:UCT}, without any tail truncation as in \emph{Subcase
  D.1}.  Without loss of generality assume $\sigma_b=1$ (the case
$\sigma_a=1$ is treated similarly by applying Lemma~\ref{lem:UCT} to
$A^+$ in the argument below). Fix $\lambda>0$ and $\zeta\in(0,1)$. By
Lemma~\ref{lem:UCT} applied to $B^+\in\RV^{0^+}_1$ at the fixed
multipliers $\lambda$ and $1/\lambda$, there exists $x_0>0$ such that, for all
$0<x<x_0$ and all small $t>0$,
\[
B^+(\lambda t)\ \le\ (1+\zeta)\,\lambda\,B^+(t),\qquad B^+(t/\lambda)\ \le\ (1+\zeta)\,(1/\lambda)\,B^+(t).
\]

Let
$\varepsilon^\star(x)\in\arg\min_{\varepsilon>0}\text{\rm objective}(\varepsilon,x)$. Then, for all
$0<x<x_0$,
\[
\begin{aligned}
f(\lambda x)
&\le\ \text{\rm objective}(\varepsilon^\star(x),\lambda x)\\
&=\ A^+(\varepsilon^\star(x))+\varepsilon^\star(x)\,B^+(\lambda x/\varepsilon^\star(x))\\
&\le\ A^+(\varepsilon^\star(x))+(1+\zeta)\,\lambda\,\varepsilon^\star(x)\,B^+(x/\varepsilon^\star(x))\\
&\le\ (1+\zeta)\,\lambda\,f(x).
\end{aligned}
\]
If $\lambda\ge 1$, let
$\varepsilon^\star(\lambda x)\in\arg\min_{\varepsilon>0}\text{\rm objective}(\varepsilon,\lambda x)$ and set
$\tilde\varepsilon(x):=\varepsilon^\star(\lambda x)/\lambda$. Then, for all $0<x<x_0$,
\[
\begin{aligned}
f(x)
&\le\ \text{\rm objective}(\tilde\varepsilon(x),x)\\
&=\ A^+(\varepsilon^\star(\lambda x)/\lambda)+\big(\varepsilon^\star(\lambda x)/\lambda\big)\,B^+\Big(\frac{x}{\varepsilon^\star(\lambda x)/\lambda}\Big)\\
&=\ A^+(\varepsilon^\star(\lambda x)/\lambda)+(1/\lambda)\,\varepsilon^\star(\lambda x)\,B^+(\lambda x/\varepsilon^\star(\lambda x)).
\end{aligned}
\]
Using $A^+(\cdot)\ge 0$ and the inequality $B^+(t/\lambda)\le (1+\zeta)\,(1/\lambda)\,B^+(t)$ with $t=\lambda x/\varepsilon^\star(\lambda x)$ gives
\[
f(x)\ \le\ (1/\lambda)\,f(\lambda x)\qquad(\lambda\ge 1).
\]
For $\lambda\in(0,1)$ apply the previous inequality with $1/\lambda>1$ and $x$ replaced by $\lambda x$ to obtain
\[
f(\lambda x)\ \ge\ \lambda\,f(x).
\]
Combining the two bounds, for all $0<x<x_0$,
\[
\frac{\lambda}{1+\zeta}\ \le\ \frac{f(\lambda x)}{f(x)}\ \le\ (1+\zeta)\,\lambda.
\]
Letting $x\to 0^+$ and then $\zeta\to 0^+$ yields
$f(\lambda x)/f(x)\to \lambda$ for each fixed $\lambda>0$, i.e., $f\in\RV^{0^+}_1$.
\end{proof}

\subsection{Auxiliary Lemmas}
\color{black}We record two fundamental tools from regular variation
used throughout, the uniform convergence theorem and Potter
bounds. Standard statements appear in \citet[Theorems 1.5.2 and
1.5.6,][]{BinghamGoldieTeugels89}. For completeness—and because they
are used repeatedly in our proof of the beta recursion—we state their
$0^+$ versions, obtained from the $\infty$-results by the inversion
$x\mapsto 1/x$.\color{black}
\begin{lemma}[\color{black}Potter bounds\color{black}]\label{lem:Potter0}
  Let $F\in\RV^{0^+}_\sigma$ with $\sigma \geq 1$. Then for any
  $\delta>0$ and any $c>1$ there exists $x_0=x_0(\delta,c)>0$ such that for all
  $x_1, x_2 \in(0,x_0)$,
  \[
    c^{-1} \min\left\{\left(\frac{x_2}{x_1}\right)^{\sigma+\delta},\
      \left(\frac{x_2}{x_1}\right)^{\sigma-\delta}\right\} \leq
    \frac{F(x_2)}{F(x_1)}\leq
    c \max\left\{\left(\frac{x_2}{x_1}\right)^{\sigma+\delta},\
      \left(\frac{x_2}{x_1}\right)^{\sigma-\delta}\right\}.
  \]
  Equivalently, for all $x\in(0,x_0)$ and all $\lambda>0$ with
  $\lambda x \in (0,x_0)$,
  \[
    c^{-1} \min\{\lambda^{\sigma+\delta},\ \lambda^{\sigma-\delta}\} \leq \frac{F(\lambda x)}{F(x)}\leq
    c \max\{\lambda^{\sigma+\delta},\ \lambda^{\sigma-\delta}\}.
  \]
\end{lemma}

\begin{lemma}[\color{black} Uniform convergence theorem\color{black}]
  \label{lem:UCT}
  If $F\in \RV^{0^+}_{\sigma}$ with $\sigma\ge 1$, then for every compact
  $[\lambda_1,\lambda_2]\subset(0,\infty)$,
\[
  \sup_{\lambda\in[\lambda_1,\lambda_2]}\Big|\frac{F(\lambda x)}{F(x)}-\lambda^{\sigma}\Big|\to0, \quad \text{as
    $x\to 0^+$}.
\]
In particular, for any $\eta\in(0,1)$ there exist $x_0(\eta)\in(0,1)$ and
$\omega_F(\eta)\to 0^+$ as $\eta\to 0^+$ such that, for all
$x\in(0,x_0(\eta))$ and $\lambda\in[1-\eta,1+\eta]$,
\[
\frac{F(\lambda x)}{F(x)}\in[ 1-\omega_F(\eta), 1+\omega_F(\eta) ].
\]
\end{lemma}

\section{Proofs associated with joint extremes}
\label{app:jtex}

\begin{proof}[Proof of Proposition~\ref{prop:jtexbg}]
We have
\begin{align*}
   g(\bm{z}^A) &= g_{C_1}(\bm z^A_{C_1}) + \sum_{k=2}^{M} [g_{C_k}(\bm{z}^A_{C_k})- |z^A_{D_k}|] +\sum_{k=M+1}^N [g_{C_k}(\bm{z}^A_{C_k})- |z^A_{D_k}|]
\end{align*}
with $g_{C_1}(\bm z^A_{C_1}) \geq 1$ and $g_{C_M}(\bm{z}^A_{C_M}) \geq 1$ since at least one element of each of $\bm z^A_{C_1}$ and $\bm z^A_{C_M}$ is one. Furthermore, since $D_k \subset C_k$, each term $g_{C_k}(\bm{z}^A_{C_k})- |z^A_{D_k}| \geq 0$. Putting this together, we conclude that if $|z^A_{D_M}|<1$, we have $g_{C_1}(\bm z^A_{C_1}) \geq 1$, $g_{C_M}(\bm{z}^A_{C_M})- |z^A_{D_M}| > 0$ and all other terms non-negative, so $g(\bm{z}^A)>1$.

Suppose now that $z^A_{D_M} = 1$. By a similar argument, we conclude that if $|z^A_{D_{M-1}}|<1$, we have $g_{C_1}(\bm z^A_{C_1}) \geq 1$, $g_{C_{M-1}}(\bm{z}^A_{C_{M-1}})- |z^A_{D_{M-1}}| > 0$ and all other terms non-negative, so $g(\bm{z}^A)>1$. Iterating this argument, it is clear that a minimum requirement for $g(\bm{z}^A)=1$ is that $z^A_{D_M} = z^A_{D_{M-1}} = \cdots = z^A_{D_2} = 1$, i.e. $\cup_{i=1}^M D_i \subset A$.
\end{proof}

\begin{proof}[Proof of Lemma~\ref{lem:bvad}]
    To get the marginal gauge function, we minimize over $x_2$. In order to do this, we consider the different linear representations of the gauge function based on the relative orderings of $x_1,x_2,x_3$. Specifically, we have six cases:

\newcommand{\thab}{\theta_{12}}
\newcommand{\thbc}{\theta_{23}}
\newcommand{\gaab}{\gamma_{12}}
\newcommand{\gabc}{\gamma_{23}}


\begin{align*}
    g(x_1,&x_2,x_3) =\\
    &\begin{cases}
    \frac{x_1}{\thab} + \frac{x_2}{\gaab} + \left(1-\frac{1}{\thab}-\frac{1}{\gaab}\right) x_2 +  \frac{x_2}{\thbc} + \frac{x_3}{\gabc} + \left(1-\frac{1}{\thbc}-\frac{1}{\gabc}\right) x_2 - x_2, & x_2 \leq x_1\leq x_3,\\
    \frac{x_1}{\thab} + \frac{x_2}{\gaab} + \left(1-\frac{1}{\thab}-\frac{1}{\gaab}\right) x_2 +  \frac{x_2}{\thbc} + \frac{x_3}{\gabc} + \left(1-\frac{1}{\thbc}-\frac{1}{\gabc}\right) x_2 - x_2, & x_2 \leq x_3 \leq x_1,\\
    \frac{x_1}{\thab} + \frac{x_2}{\gaab} + \left(1-\frac{1}{\thab}-\frac{1}{\gaab}\right) x_1 +  \frac{x_2}{\thbc} + \frac{x_3}{\gabc} + \left(1-\frac{1}{\thbc}-\frac{1}{\gabc}\right) x_2 - x_2, & x_1 \leq x_2 \leq x_3,\\
     \frac{x_1}{\thab} + \frac{x_2}{\gaab} + \left(1-\frac{1}{\thab}-\frac{1}{\gaab}\right) x_2 +  \frac{x_2}{\thbc} + \frac{x_3}{\gabc} + \left(1-\frac{1}{\thbc}-\frac{1}{\gabc}\right) x_3 - x_2, & x_3 \leq x_2 \leq x_1,\\
    \frac{x_1}{\thab} + \frac{x_2}{\gaab} + \left(1-\frac{1}{\thab}-\frac{1}{\gaab}\right) x_1 +  \frac{x_2}{\thbc} + \frac{x_3}{\gabc} + \left(1-\frac{1}{\thbc}-\frac{1}{\gabc}\right) x_3 - x_2, & x_1\leq x_3 \leq x_2,\\
    \frac{x_1}{\thab} + \frac{x_2}{\gaab} + \left(1-\frac{1}{\thab}-\frac{1}{\gaab}\right) x_1 +  \frac{x_2}{\thbc} + \frac{x_3}{\gabc} + \left(1-\frac{1}{\thbc}-\frac{1}{\gabc}\right) x_3 - x_2, & x_2 \leq x_3 \leq x_1.
    \end{cases}
\end{align*}
The corresponding derivatives with respect to $x_2$, and subsequent minimizers $x_2^\star$, are:
\begin{align*}
   \frac{\partial g(x_1,x_2,x_3)}{\partial x_2} =
    \begin{cases}
    1- \frac{1}{\thab}-\frac{1}{\gabc} <0, & x_2 \leq x_1 \leq x_3 \Rightarrow x_2^\star = x_1 = \min(x_1,x_3),\\
     1- \frac{1}{\thab}-\frac{1}{\gabc} <0, & x_2 \leq x_3 \leq x_1 \Rightarrow x_2^\star = x_3 = \min(x_1,x_3),\\
    \frac{1}{\gaab}-\frac{1}{\gabc}, & x_1 \leq x_2 \leq x_3 \Rightarrow x_2^\star = \begin{cases}x_3 = \max(x_1,x_3), & \gabc<\gaab\\ x_1 = \min(x_1,x_3), & \gabc>\gaab\end{cases},\\
     \frac{1}{\thab}-\frac{1}{\thbc}, & x_3 \leq x_2 \leq x_1 \Rightarrow x_2^\star = \begin{cases}x_3 = \min(x_1,x_3), & \thbc<\thab\\ x_1 = \max(x_1,x_3), & \thbc>\thab\end{cases},\\
    \frac{1}{\gaab}+\frac{1}{\thbc} - 1 >0, & x_1 \leq x_3 \leq x_2 \Rightarrow  x_2^\star = x_3 = \max(x_1,x_3),\\
     \frac{1}{\gaab}+\frac{1}{\thbc} - 1 >0, & x_2 \leq x_3 \leq x_1 \Rightarrow  x_2^\star = x_1 = \max(x_1,x_3).\\
    \end{cases}
\end{align*}

In Case 3, if $\gamma_{12}=\gamma_{23}=\gamma$, then $g(x_1,x_2,x_3) = x_3/\gamma+(1-1/\gamma)x_1 = \max(x_1,x_3)/\gamma+(1-1/\gamma)\min(x_1,x_3)$ does not depend on $x_2$ in the region $x_1\leq x_2 \leq x_3$. Similarly, in Case 4, if $\theta_{12}=\theta_{23}=\theta$, then $g(x_1,x_2,x_3) = x_1/\theta+(1-1/\theta)x_3 = \max(x_1,x_3)/\theta+(1-1/\theta)\min(x_1,x_3)$ does not depend on $x_2$ in the region $x_3\leq x_2 \leq x_1$. 

Substituting in these minimizers, and combining with the observations made for $\gamma_{12}=\gamma_{23}$ on $x_1\leq x_2 \leq x_3$, and $\theta_{12}=\theta_{23}$ on $x_3\leq x_2 \leq x_1$ gives

\begin{align*}
&g(x_1,x_2^\star,x_3) \\= &\begin{cases}
    \frac{x_1}{\thbc} + \frac{x_3}{\gabc} + \left(1-\frac{1}{\thbc}-\frac{1}{\gabc}\right) x_1 = \frac{\max(x_1,x_3)}{\gabc}+\left(1-\frac{1}{\gabc}\right)\min(x_1,x_3), & [x_1 \leq x_3]\\
        \frac{x_1}{\thab} + \frac{x_3}{\gaab} + \left(1-\frac{1}{\thab}-\frac{1}{\gaab}\right) x_3 = \frac{\max(x_1,x_3)}{\thab}+\left(1-\frac{1}{\thab}\right)\min(x_1,x_3), & [x_3 \leq x_1]\\       
    \frac{x_1}{\thab} + \frac{x_3}{\gaab} + \left(1-\frac{1}{\thab}-\frac{1}{\gaab}\right) x_1 = \frac{\max(x_1,x_3)}{\max(\gaab,\gabc)}+\left(1-\frac{1}{\max(\gaab,\gabc)}\right)\min(x_1,x_3), & [x_1 \leq  x_3; \gabc\leq \gaab]\\    \frac{x_1}{\thbc} + \frac{x_3}{\gabc} + \left(1-\frac{1}{\thbc}-\frac{1}{\gabc}\right) x_1 = \frac{\max(x_1,x_3)}{\max(\gaab,\gabc)}+\left(1-\frac{1}{\max(\gaab,\gabc)}\right)\min(x_1,x_3), & [x_1 \leq  x_3; \gabc \geq \gaab]\\
     \frac{x_1}{\thab} + \frac{x_3}{\gaab} + \left(1-\frac{1}{\thab}-\frac{1}{\gaab}\right) x_3 =  \frac{\max(x_1,x_3)}{\max(\thab,\thbc)}+\left(1-\frac{1}{\max(\thab,\thbc)}\right)\min(x_1,x_3), & [x_3 \leq 
 x_1; \thbc \leq \thab]\\    
     \frac{x_1}{\thbc} + \frac{x_3}{\gabc} + \left(1-\frac{1}{\thbc}-\frac{1}{\gabc}\right) x_3 = \frac{\max(x_1,x_3)}{\max(\thab,\thbc)}+\left(1-\frac{1}{\max(\thab,\thbc)}\right)\min(x_1,x_3), & [x_3 \leq 
 x_1; \thbc \geq \thab]\\
    \frac{x_1}{\thab} + \frac{x_3}{\gaab} + \left(1-\frac{1}{\thab}-\frac{1}{\gaab}\right) x_1 = \frac{\max(x_1,x_3)}{\gaab}+\left(1-\frac{1}{\gaab}\right)\min(x_1,x_3), & [x_1 \leq  x_3]\\
        \frac{x_1}{\thbc} + \frac{x_3}{\gabc} + \left(1-\frac{1}{\thbc}-\frac{1}{\gabc}\right) x_3 = \frac{\max(x_1,x_3)}{\thbc}+\left(1-\frac{1}{\thbc}\right)\min(x_1,x_3), & [x_3 \leq x_1].
\end{cases}    
\end{align*}
Notice now that in the $x_1 \leq  x_3$ cases, the only parameters appearing are $\gaab$ and $\gabc$. In general the function $\max(x,y)/\alpha + (1-1/\alpha)\min(x,y)$ is non-increasing in $\alpha$ for $x,y\geq 0$, so considering $x_1 \leq  x_3$, this is minimized by $\frac{\max(x_1,x_3)}{\max(\gaab,\gabc)}+\left(1-\frac{1}{\max(\gaab,\gabc)}\right)\min(x_1,x_3)$. Similarly, considering the $x_3 \leq x_1$ case, this is minimized by $\frac{\max(x_1,x_3)}{\max(\thab,\thbc)}+\left(1-\frac{1}{\max(\thab,\thbc)}\right)\min(x_1,x_3)$, in other words $g_{\{1,3\}}(x_1,x_3)= \min_{x_2 \geq 0} g(x_1,x_2,x_3)$ can be expressed

\begin{align}
&g_{\{1,3\}}(x_1,x_3) \notag \\= &\begin{cases}
\frac{x_3}{\max(\gaab,\gabc)}+\left(1-\frac{1}{\max(\gaab,\gabc)}\right)x_1 = \frac{\max(x_1,x_3)}{\max(\gaab,\gabc)}+\left(1-\frac{1}{\max(\gaab,\gabc)}\right)\min(x_1,x_3), & x_1 \leq x_3\\
\frac{x_1}{\max(\thab,\thbc)}+\left(1-\frac{1}{\max(\thab,\thbc)}\right)x_3 = \frac{\max(x_1,x_3)}{\max(\thab,\thbc)}+\left(1-\frac{1}{\max(\thab,\thbc)}\right)\min(x_1,x_3), & x_3 \leq  x_1.
\end{cases}
\label{eq:marg13}
\end{align}
Equivalently, equation~\eqref{eq:marg13} can be expressed
\begin{align*}
  g_{\{1,3\}}(x_1, x_3)=  \frac{x_1}{\max(\theta_{12},\theta_{23})}+ \frac{x_3}{\max(\gamma_{12},\gamma_{23})} + \left(1-\frac{1}{\max(\theta_{12},\theta_{23})}-\frac{1}{\max(\gamma_{12},\gamma_{23})}\right)\min(x_1,x_3) .
\end{align*}
\end{proof}

\begin{proof}[Proof of Proposition~\ref{prop:treead}]
Consider indices $k<l \in \Vset$, and let $\Vset'=\{k=i_1,\ldots,i_m=l\}$ be the shortest path from $k$ to $l$. A tree graphical model is a block graphical model, and by Lemma~\ref{lem:blockgraphtopath}, $g_{\{k,l\}}$ can be expressed through the chain graph defined by $\Vset'$.

Firstly suppose that $k=i_1$, $l=i_3$. Then Lemma~\ref{lem:bvad} gives the form of $g_{\{k,l\}}=g_{\{i_1,i_3\}}$, which is of the claimed type. Now let $l=i_4$. Then $g_{\{k,l\}}=g_{\{i_1,i_4\}}$, with
\[
g_{\{i_1,i_4\}}(x_{i_1},x_{i_4}) = g_{\{i_1,i_3\}}(x_{i_1},x_{i_3}) +g_{\{i_3,i_4\}}(x_{i_3},x_{i_4}) - x_{i_3}.
\]
Using Lemma~\ref{lem:bvad} with $\{i_1',i_2',i_3'\} =\{i_1,i_3,i_4\}$ again gives the claimed form for $g_{\{k,l\}}$. We can therefore proceed by induction for any $m\geq 4$.
\end{proof}


\end{document}